\documentclass[11pt]{amsart}

\usepackage{xcolor}
\usepackage{amsmath, amsthm, amssymb, amsfonts, mathtools, enumerate}
\usepackage{hyperref}
\usepackage{verbatim} 
\usepackage{graphicx,epsfig,color}
\usepackage{exscale}

\theoremstyle{plain}

\newtheorem{theorem}{Theorem}[section]
\newtheorem*{th:re}{Theorem \ref{th:re}}
\newtheorem{fact}[theorem]{Fact}
\newtheorem{lemma}[theorem]{Lemma}

\newtheorem{corollary}[theorem]{Corollary}

\newtheorem{proposition}[theorem]{Proposition}

\theoremstyle{definition}

\newtheorem{definition}[theorem]{Definition}

\newtheorem{assumption}[theorem]{Assumption}

\theoremstyle{remark}
\newtheorem{remark}[theorem]{Remark}

\newtheorem{example}[theorem]{Example}

\newtheorem*{jjj}{Remark}

\DeclareSymbolFont{AMSb}{U}{msb}{m}{n}
\DeclareMathSymbol{\N}{\mathalpha}{AMSb}{"4E}
\DeclareMathSymbol{\R}{\mathalpha}{AMSb}{"52}
\DeclareMathSymbol{\Z}{\mathalpha}{AMSb}{"5A}
\DeclareMathSymbol{\D}{\mathalpha}{AMSb}{"44}
\DeclareMathSymbol{\s}{\mathalpha}{AMSb}{"53}

\newcommand{\eL}{L}
\renewcommand{\Im}{\mbox{Im}}

\newcommand{\sF}{{\scriptscriptstyle{F}}}
\newcommand{\sC}{{\scriptscriptstyle C}}
\newcommand{\sB}{{\scriptscriptstyle B}}
\newcommand{\sX}{{\scriptscriptstyle{X}}}

\newcommand{\sN}{{\scriptscriptstyle N}}

\DeclareMathOperator{\I}{I}
\DeclareMathOperator{\J}{J}
\DeclareMathOperator{\tf}{Test}

\DeclareMathOperator{\vol}{vol}
\DeclareMathOperator{\Ch}{Ch}
\DeclareMathOperator{\lip}{Lip}

\DeclareMathOperator{\essup}{ess-sup}

\DeclareMathOperator{\de}{d}

\DeclareMathOperator{\m}{m}
\DeclareMathOperator{\ric}{ric}
\DeclareMathOperator{\diam}{diam}

\renewcommand{\L}{\mbox{L}}
\newcommand{\mwp}{B\times_f F}
\newcommand{\nwp}{B\times_f^\sN F}
\newcommand{\smwp}{{\scriptscriptstyle{B\times_f F}}}
\newcommand{\snwp}{{\scriptscriptstyle B\times_f^\sN F}}

\newcommand{\dint}{\de}

\newcommand{\RCD}{\mathsf{RCD}}
\newcommand{\CD}{\mathsf{CD}}
\newcommand{\TCD}{\mathsf{TCD}}
\newcommand{\MCP}{\mathsf{MCP}}
\newcommand{\BE}{\mathsf{BE}}
\newcommand{\sW}{{\scriptstyle W}}

\title[Warped products and the RCD condition]{Warped products over  one-dimensional base spaces  and the RCD condition}
\author{Christian Ketterer}
\thanks{Department of Mathematics \& Statistics, Logic House,
South Campus,
Maynooth University,
Ireland.\ \ {{\it Email address:} {\tt christian.ketterer@mu.ie}}}

\thanks{{\it 2020 Mathmatics Subject Classification.} Primary:  51M15, 53C21,  49Q22; \ \ Keywords: warped products, metric measure spaces, Riemannian curvature-dimension condition}
\address{Department of Mathematics \& Statistics, Logic House,
South Campus,
Maynooth University,
Ireland}\email{christian.ketterer@mu.ie}
\begin{document}
\begin{abstract} We prove the Riemannian curvature-dimension condition $\RCD(KN,N+1)$ for  an  $N$-warped product $B\times_f^\sN F$ over a one-dimensional base space $B$ with a Lipschitz function $f: B\rightarrow \R_{\geq 0}$,   provided  {(1)} $f$  is  $Kf$-concave,  {(2)} $f$ satisfies a sub-Neumann boundary condition $\frac{\partial f}{\partial n}\geq 0$ on $\partial B\backslash f^{-1}(0)$ and $F$ is a compact metric measure space  satisfying (3) the condition $\RCD(K_F (N-1), N)$ with $K_F:= \sup_B \{ (Df)^2 + Kf^2\}$.  The result is sharp, i.e. we show that {(1), (2)} and {(3)} are  necessary for the validity of statement provided $K_F\geq 0$.  In general, only a weaker statement is true. If $f$ is assumed to be $Kf$-affine, then the condition $\RCD(K N, N+1)$ for the $N$-warped product holds if and only if  the condition $\RCD(K_F(N-1), N)$ holds  for $F$ for any $K_F\in \R$.
\end{abstract}
\maketitle
\tableofcontents

\section{Introduction}
The theory of curvature-dimension conditions for  metric measure spaces, such as the Riemannian Curvature-Dimension condition $\RCD(K,N)$, has emerged as a central framework in the study of synthetic lower Ricci curvature bounds. A core challenge in this field is the construction and analysis of spaces satisfying such conditions, especially in the presence of non-smooth structures. In this article, we provide a broad and sharp characterization of the $\RCD(K,N)$ condition in the setting of warped product spaces over one-dimensional base spaces, thereby significantly extending the scope of known results.

Warped products generalize the classical Cartesian product of metric spaces and serve as a versatile construction in differential geometry, geometric analysis, and mathematical physics. They are essential tools for modeling spaces with both lower and upper curvature bounds \cite{lobaem, albi, coldingnaberII, hangigliwarped}, appearing as  model spaces in rigidity theorems \cite{conesplittingtheorems, ketterer3, almostrigidity, cdnpsw, cz,  DGi, chen_lina_almost_volume} and yielding a rich source of new examples \cite{anderson, cheegercoldingI, sturm25} and counterexamples \cite{hnw}. Notable special cases are Euclidean cones and spherical suspensions.

In this work, we focus on  warped product spaces endowed with a natural reference measure. Specifically, we investigate $N$-warped products, $B \times_f^\sN F$, where $B$ is a $1$-dimensional Riemannian manifold, $f: B \to [0,\infty)$ is a Lipschitz continuous  function,  $(F, \de_\sF, \m_\sF)$ is a compact metric measure space and $N\in [0, \infty)$ is a parameter. We establish necessary and sufficient conditions under which such a space satisfies the curvature-dimension condition $\RCD(KN, N+1)$, for $K \in \mathbb{R}$. These results unify and extend previous work by the author in \cite{ketterer2}.

Our  theorems reveal a precise relationship between the curvature of the fiber $F$, the geometry of the base $B$, and the  properties of the warping function $f$. In particular, we show that the curvature lower bound on $F$ is governed by the quantity $\essup_B \{(f')^2 + Kf^2\}=K_F$ 
and that the $fK$-concavity of $f$ plays a central role in controlling the geometry of the warped product space.

Our  main result is the following theorem.
\begin{theorem}\label{maintheorem1} Let $F$ be a compact metric measure space,  let $B$ be  a $1$-dimensional Riemannian manifold, and  let $f: B\rightarrow [0,\infty)$  be Lipschitz continuous.    Let $K\in \R$ and $N\in [1, \infty)$. We assume that 
\begin{enumerate}
\item $f$ is $fK$-concave,  \item $\frac{\partial}{\partial n} f\geq 0$ on $\partial B\backslash f^{-1}(\{0\})$ for the outer normal vector $n$. 
\item $F$ satisfies the condition $\RCD(K_F(N-1), N)$ where 
$$K_F=\sup_B \left\{ (Df)^2+ Kf^2\right\}$$
and $\diam_F\leq \pi{\scriptstyle \sqrt{\frac{N-1}{K_F}}}$ if $N=1$ and $K_F>0$.

\end{enumerate}
Then $B\times_f^\sN F$ satisfies the condition $\RCD(KN, N+1)$. 
\end{theorem}
\noindent
Here $Df=\max\{f^+, -f^-,0\}$ where $f^+$ and $f^-$ are the right and left derivative, respectively. $Df$ coincides a.e. with $|f'|$.
\smallskip

The Riemannian curvature-dimension condition $\RCD$  for a metric measure space  is defined as the combination of the curvature-dimension condition $\CD$ together with the property that  the underlying metric measure space is  infinitessimal Hilbertian, i.e. the Cheeger energy is a quadratic form. We refer to Subsection \ref{subsection:rcd} for  details. 
\smallskip

Our second theorem shows that the conditions {\it (1), (2)} and {\it (3)} in Theorem \ref{maintheorem1}  are not only sufficient but also necessary. 
\begin{theorem}\label{maintheorem2}
Let $K\in \R$ and $N\in [1, \infty)$.  Let $F$ be a geodesic metric measure space,  let $B$ be  a one-dimensional Riemannian manifold, and  let $f: B\rightarrow [0,\infty)$  be a Lipschitz function.
We assume that  $B\times_f^\sN F$ satisfies the condition $\RCD(KN, N+1)$.
Then
\begin{enumerate}
\item $f$ is $fK$-concave,
\item $\frac{\partial}{\partial n} f\geq 0$ on $\partial B\backslash f^{-1}(\{0\})$ for the outer normal vector $n$. 
\end{enumerate}
 If  $\sup_B \{ (Df)^2+ K f^2\}=K_F\geq 0$, then 
\begin{enumerate}
\item[(3)] $F$ satisfies the condition $\RCD(K_F(N-1), N)$ and  $\diam_F\leq \pi \sqrt{\frac{N-1}{K_F}}$ if $N=1$ and $K_F>0$.
\end{enumerate}
If $K_F<0$, then 
\begin{enumerate}
\item[(4)] $F$ satisfies the condition $\RCD(K_FN, N+1)$.
\end{enumerate}
\end{theorem}

\begin{remark}
The condition {\it (2)} is equivalent to:
\smallskip\\
{\it 
$(\dagger)$ If $B^\dagger$ is the result of gluing two copies of $B$ together along the boundary component $\partial B\backslash f^{-1}(\{0\})$, and $f^\dagger: B^\dagger \rightarrow [0,\infty)$ is the tautological extension of $f$ to $B^\dagger$, then $(f^\dagger)''+ K f^\dagger\leq 0$ is satisfied on $B^\dagger$. }
\end{remark}
\begin{example}\label{expexp}
The conditon $\sup_B \{ (Df)^2+ K f^2\}\geq 0$ is necessary for {\it (3)}  as the following examples demonstrates. 
Let $B=\R$ and let $F$ be an $n$-dimensional  Riemannian manifold of constant  curvature $-1$. In particular, $F$ is Einstein with  Ricci curvature  equal to $-(n-1)$.  {\color{black} The  Riemannian product $\R\times F$ is  a warped product w.r.t. $f(r)\equiv 1$, and it satisfies the  lower bound $\ric_{\scriptscriptstyle \R\times F}\geq -(n-1)g_{\scriptscriptstyle\R\times F}=  n Kg_{\scriptscriptstyle\R\times F}$ with $K:=-\frac{n-1}{n}$. This bound cannot be improved since $\ric_{\scriptscriptstyle \R\times F}= -(n-1)g_{\scriptscriptstyle\R\times F}$ in direction of unit vectors in $0\oplus TF$.  }
The function $f\equiv 1$ satisfies $f'' +K f\leq 0$, and $(f')^2 +K f^2=-\frac{n-1}{n}<0$. But $F$ doesn't have Ricci curvature bigger than $(n-1)\left(Kf^2 + (f')^2\right)=(n-1)K=-(n-1)\frac{n-1}{n}$ since  $F$ is Einstein with $\ric_\sF= - (n-1) g_\sF$ and $-(n-1)\frac{n-1}{n}>-(n-1)$. {\color{black}  Hence the lower curvature bound in {\it (4)} is sharp. We conjecture that the dimension bound $N+1$ can be improved to $N$.
}
\smallskip\\
One also should  compare this with Theorem \ref{th:albi1} in  \cite{albi} for spaces with Alexandrov lower curvature bounds where no such restriction is needed. 
\end{example}
Combining the Theorem \ref{maintheorem1} and Theorem \ref{maintheorem2} we obtain the following characterization of synthetic Riemannian Ricci curvature bounds for $N$-warped products. 
\begin{corollary}Let $K\in \R$ and $N\in [1, \infty)$. 
Let $F$ be a compact, geodesic metric measure space,  let $B$ be  a $1$-dimensional Riemannian manifold, and  let $f: B\rightarrow [0,\infty)$  be Lipschitz continuous such that $$\essup_B \{ (f')^2+ K f^2\}=K_F\geq 0.$$ 
Then $B\times_f^\sN F$ satisfies the condition $\RCD(KN, N+1)$ if and only if
\begin{enumerate}
\item $f''+ Kf\leq 0$,  \item $\frac{\partial }{\partial n}f \geq 0$ on $\partial B\backslash f^{-1}(\{0\})$ for the outer normal vector $n$. 
\item $F$ satisfies the condition $\RCD(K_F(N-1), N)$ and  $\diam_F\leq \pi \sqrt{\frac{N-1}{K_F}}$ if $N=1$ and $K_F>0$.

\end{enumerate}
\end{corollary}
If $K_F<0$, we still have  (4) in  Theorem \ref{maintheorem2}. But this is not sharp. Especially one would expect  the dimension parameter to be $N$.
If  we  strengthen the properties of $f$, assuming that $f$ is  $Kf$-affine, we have the following result.
{
\begin{theorem}\label{lastcorollary} Let $K\in \R$ and $N\in [1, \infty)$.
Let $F$ be a compact, geodesic metric measure space,  let $B$ be  a $1$-dimensional Riemannian manifold, and   $f: B\rightarrow [0,\infty)$ satisfies $f''+ Kf =0$.
Then $B\times_f^\sN F$ satisfies the condition $\RCD(KN, N+1)$ if and only if
$F$ satisfies the condition $\RCD(K_F(N-1), N)$ where 
$K_F:=  (f')^2+ Kf^2$
and $\diam_F\leq \pi \sqrt{\frac{N-1}{K_F}}$ if $N=1$ and $K_F>0$.
\end{theorem}}

While the author's prior work has addressed only particular cases, such as the spherical suspension (e.g. $B=[0,\pi]$ and $f(r)=\sin r$) and the Euclidean cone (e.g., $B = [0,\infty)$ and $f(r) = r$), this article generalizes these results significantly in multiple directions:
\begin{itemize}
\item We allow for general, possibly non-compact one-dimensional base spaces $B$, 
\item  We don't require any  assumption of smoothness on the warp function $f: B\rightarrow [0, \infty)$.
\item We prove the sharpness of our assumptions: the conditions on $f$ and $F$ are not merely sufficient but also necessary for the warped product to satisfy the curvature-dimension condition.
\item We adapt our framework to a nonsmooth differential calculus, closer in spirit to Gigli's nonsmooth differential calculus \cite{giglinonsmooth}, as opposed to the Dirichlet form-based framework employed in the author's earlier work. 
\end{itemize}

Our methods combine careful differential analysis with synthetic tools from metric geometry, optimal transport and the calculus of metric measure spaces. While some technical ideas parallel those in the author's  prior work \cite{ketterer2}, we emphasize a cleaner and more general formulation.
\smallskip

This article can also be viewed in comparison with 
  recent work  by Calisti, S\"amann and the author. In \cite{cks} we study warped products and   curvature-dimension conditions such as $\CD$ for metric measure spaces as well as timelike curvature-dimension conditions for measured Lorentzian length spaces such as $\TCD$, without assuming that  metric measure spaces are infinitessimal Hilbertian.
 Results, ideas and methods in \cite{cks} are almost completely independent from the present article. 

\subsubsection{Methods} The general strategy for this work is the same as in \cite{ketterer2}.
 To prove Theorem \ref{maintheorem1} we exploit the characterization of the Riemannian Curvature-Dimension in terms of  the Bakry-Emery condition (Definition \ref{def:rcd}). 
In \cite{ketterer2} we still rely in several points heavily on the smoothness of $f$, on $\partial B\subset f^{-1}(0)$,  on the differential equality $f''+Kf=0$ as well as partly on compactness of $B$. For instance, the theorem in \cite{ketterer2} that shows the Riemannian curvature-dimension for the Euclidean cone $[0, \infty)\times_r^\sN F$, circumvents  compactness of $B$ by  using a blow up argument based on Gromov-Hausdorff stability of the $\RCD$ condition. However this works only for the Euclidean cone.

In the present work we now remove any restriction on $f$.  On the one hand,  we allow  noncompact spaces $B$ by refining the methods  in \cite{ketterer2}. On the other hand, a key point is an  approximation that uses the so-called fiber independence of warped products (Theorem \ref{th:albi0}). We believe this will be useful also in different places. 
\subsubsection{Applications}
Theorem \ref{maintheorem1} is a  quite  flexible tool for  constructing $\RCD$ spaces as the following example illustrates.
\begin{example}Let $f: \mathbb S^1 \rightarrow (0, \infty)$ be smooth. For $K<0$ such that $-K$ is sufficiently large we have that $f'' \leq - Kf^2$.  If $K_F\geq K f^2$ for a constant $K_F\in \R$, then $K_F\geq \sup_B\{ (f')^2 + Kf^2\}$ (Proposition \ref{prop:albi2}). Then the $N$-warped product $B\times_f^\sN F$ satisfies the Riemannian curvature-dimension condition $\RCD(KN, N+1)$ for any compact $\RCD(K_F(N-1),N)$ space $F$.
\end{example}
A  consequence of Theorem \ref{maintheorem1} is the sharp Brunn-Minkowski inequality \cite{stugeo2}.
\begin{corollary}[Brunn-Minkowski inequality]
Let $B, f$ and $F$ be as in the previous theorem. Then for all measurable subsets $A_0, A_1\subset B\times_f^\sN F$ with $\m^\sN(A_0)\m^\sN(A_1)>0$ it holds that 
$$\m(A_t)^\frac{1}{N}\geq \tau_{KN,N+1}^{(1-t)}(\Theta) \m(A_0)^\frac{1}{N} + \tau_{KN,N+1}^{(t)}(\Theta) \m(A_1)^\frac{1}{N}$$
where $A_t$ is the set of all $t$-midpoints of geodesics which start in $A_0$ and end in $A_1$ and $\Theta$ is defined as 
\begin{center}$\Theta= \begin{cases} \inf\limits_{v\in A_0, w\in A_1} \de_{B\times_f F}(v,w) & \ \mbox{ if } K\geq 0, \\
\sup\limits_{v\in A_0, w\in A_1} \de_{B\times_f F}(v,w) & \ \mbox{ if } K<0.
\end{cases}$
\end{center}
and $\tau_{KN, N+1}^{(t)}(\Theta)$ are the distortion coefficients defined in \eqref{distortioncoef}. 
\end{corollary}

From Theorem \ref{maintheorem2} we also obtain refined information about  the quotient spaces that appear in rigidity theorems in \cite{cdnpsw} and in \cite{gigli_marconi}. For instance, we have the following result.
\begin{theorem} Let $N\in (1, \infty)$ and let $X$ be an $\RCD(-N, N+1)$ space. If there exist a function $u\in D_{loc}(L^\sX)$ such that $|\nabla u|=1$ $\m_\sX$-a.e. and $L^\sX u= N$, then $X$ is isomorphic to the $N$-warped product $\R\times_{\exp}^{\scriptscriptstyle{N}} Y$ where $Y$ is an $\RCD(0, N)$ space. 
\end{theorem}

\subsubsection{Restrictions}
Theorem \ref{maintheorem1}  and Theorem \ref{lastcorollary} consider only  warped products over  one-dimensional base spaces. One reason for this is that we use  a sharp theorem by Herman Weyl on self-adjontness of Schr\"odinger operators  with Dirichlet boundary conditions on a one-dimensional base space. Another reason is that we use   Theorem \ref{th:mcpfornwp} from \cite{cks} to establish some a-priori regularity of $N$-warped products (Section \ref{sec:metricstructure}). Conjectures about constructions with a more general base are formulated in \cite{kettererwp}. 
The other major restriction in our theorem is the compactness assumption of $F$. This is because we  use the discreteness of the spectrum of the  Laplace operator of  $F$. Under this assumption we are able to  reduce the problem of essential self-adjointness of the Laplace operator on the warped prduct to Schroedinger operators on the base space. 

\smallskip
We will adress the removal of both restrictions, i.e. the dimensionality of the base space $B$ and the compactness of the fiber space $F$, in  upcoming publications.  The latter requires a finer spectral analysis of the Laplace operator on $B\times_f^\sN F$ and of its connection to  operators on the underlying spaces $B$ and $F$. 
\smallskip

\subsubsection{Outline of the article} In Section 2 we recall  several topics:  properties of $fK$-concave functions, the differential calculus for metric measure spaces and Dirichlet forms, the Riemannian curvature-dimension condition,  the class of one-dimensional weighted Riemannian manifolds satisfying a curvature-dimension condition,   Schr\"odinger operators on one-dimensional spaces  and  sharp self-adjointness criteria. 

In Section 3 we define the warped product between metric spaces and the $N$-warped product of metric measure spaces. Provided the warp function $f$ is smooth we define an energy functional for $N$-warped products that mimics the Dirichlet energy for warped products in the smooth case. Then we study the associated operator and semi group.  Finally we deduce   a priori regularity results for the semi group of a class of Schroedinger operators on $B$ equipped with a weight given by $f$. 

In Section 4 we show that under certain regularity assumptions for  the fiber $F$ the energy defined in the  previous section, is indeed the Cheeger energy of the warped product.

In section 5 we first derive a formula for the $\Gamma_2$ operator associated to the energy defined in Section 3, that mimics a formula one  computes for the Ricci tensor of smooth warped product. For this we still assume that $f$ is smooth on $B$. Then, in several approximation steps we show that the $N$-warped product satisfies the Bakry-Emery curvature-dimension condition and therefore the Riemannian curvature-dimension condition. Finally we remove the smoothness assumption on $f$ by approximation. Here we use the stability  of the Riemannian curvature-dimension condition w.r.t. Gromov-Hausdorff convergence. 

In Section 6 we show that the precise assumptions on $B$, $f$ and $F$ are  not only sufficient to infer the Riemannian curvature-dimension condition $\RCD(KN, N+1)$ for the $N$-warped product but also necessary. Finally we discuss the proof of Theorem \ref{lastcorollary}. 
\subsubsection*{Acknowledgements} Parts of this work were written when the author stayed at the Hausdorff Institute in Bonn during the Trimester Program: {\it Metric Analysis}. I want to thank the organizers of the trimester programm  and the Hausdorff Institute for providing    an excellent and stimulating research evironment. 
This work was finished during a research visit of the author at  the  Universit\'e de Haute-Alsace in Mulhouse as part of the program  "poste rough" funded by the  Institut National des Sciences Math\'ematiques of the CNRS.
I  want to thank  my local host Nicolas Juillet for many inspiring dicussions about topics connected to this work.
\section{Preliminaries}
\subsection{Semi-concave functions}
For $\kappa\in\R$ let the {\it generalized sine functions}  $\sin_\kappa: [0, \infty)\rightarrow \R$ be the solution of 
$$u''+ \kappa u=0, \ \ u(0)=0, \ u'(0)=1.$$
Then, for $t\in[0,1]$, we define the \emph{volume distortion coefficients} for $\kappa=\tfrac{K}{N}$ with $K\in\R$ and $N\geq1$ as
\begin{equation*}
    \sigma_{\kappa}^{(t)}(\theta):=\begin{cases}
        \tfrac{\sin_\kappa(t\theta)}{\sin_\kappa(\theta)}\quad&\mathrm{if}\,\kappa\theta^2\neq0\,\mathrm{and}\,\kappa\theta^2<\pi^2,\\
        t&\mathrm{if}\,\kappa\theta^2=0,\\
        +\infty&\mathrm{if}\,\kappa\theta^2>+\infty\,,
    \end{cases}
\end{equation*}
and set $\sigma_{K,N}^{(t)}(0)=t$. Define then
\begin{align}\label{distortioncoef}
    \sigma_{K,N}^{(t)}(\theta):=\sigma_{\frac{K}{N}}^{(t)}(\theta), \ \ \ \ \ \ \ 
    \tau_{K,N}^{(t)}(\theta):=(t\cdot\sigma_{K,N-1}^{(t)}(\theta)^{N-1})^{\frac{1}{N}}.
\end{align}
When $N=1$ we set $\smash{\tau_{K,1}^{(t)}(\theta)=t}$ if $K\leq0$ and $\smash{\tau_{K,1}^{(t)}(\theta)=+\infty}$ if $K>0$. 
\smallskip

Let $f: [a,b]\rightarrow [0, \infty)$ be a Lipschitz function. The following statements are equivalent.
\begin{enumerate}
\item For all $  t_0, t_1\in [a,b]$ it holds
$$f((1-s) t_0 + s t_1) \geq \sigma_{K}^{(1-s)}(t_1-t_0) f(t_0) + \sigma_{K}^{(s)}(t_1-t_0) f(t_1) \ \forall s\in [0,1], $$
\item $f''+ K f\leq 0$ in the distributional sense.
\end{enumerate}
If (1) or (2) hold, we say $f$ is $fK$-concave. 
\smallskip

We call a Lipschitz function $f: B\rightarrow [0, \infty)$ $fK$-conave if $f\circ \gamma=h$ is $hK$-concave for every distance preserving map $\gamma:[a,b]\rightarrow B$.
 If $f: B\rightarrow [0, \infty)$ is $fK$-concave, we also write $f''+Kf \leq 0$. 

 In particular $f$ is semi-concave. In this case the left and right derivative $f^-$ and $f^+$ exist in every point, and are right, respectively, left continuous.  We call  $Df= \max\{ f^+ , - f^-\}$ the Alexandrov derivative of $f$, and $Df$ coincides a.e. with the derivative $f'$.

\smallskip

A  complete $1$-dimensional Riemannian manifold $B$ is, up to isometries, either $\mathbb S^1=\R/2\pi \mathbb Z$, $\R$, $[0, \pi]$ or $[0, \infty)$. 

\begin{proposition}[{\cite[Proposition 3.1]{albi}}]\label{prop:albi2}
We consider a complete $1$-dimensional Riemannian manifold $B$, and a Lipschitz continuous function $f: B\rightarrow [0, \infty)$ such that $f''+ Kf\leq 0.$
We assume that $f$ satisfies $(\dagger)$.  
 We set $ f^{-1}(\{0\})\cap \partial B=X$. 
The following two statements are equivalent: 
\begin{enumerate}
\item $K_F\geq |Df|^2 + K f^2$ on $B$.
\smallskip
\item 
$\begin{cases} K_F\geq K  \inf f^2 & \mbox{ if } X= \emptyset\\
K_F\geq |Df|^2 $ on $f^{-1}(\{0\}) & \mbox{ if } X\neq \emptyset
\end{cases}.
$
\end{enumerate}
$(\dagger)$ If $B^\dagger$ is the result of gluing two copies of $B$ together along the boundary component $\partial B\backslash f^{-1}(\{0\})$, and $f^\dagger: B^\dagger \rightarrow [0,\infty)$ is the tautological extension of $f$ to $B^\dagger$, then $(f^\dagger)''+ K f^\dagger\leq 0$ is satisfied on $B^\dagger$.
\end{proposition}
\begin{corollary} \label{cor:special}
Consider $B$ and $f$ as in the previous proposition. 
\begin{enumerate}
\item If $K>0$, we have  $f^{-1}(\{0\})\neq \emptyset$, $B\simeq [0,a]$ and  $$(f')^2+ Kf^2>0.$$
\item Assume $f^{-1}(\{0\})\neq \emptyset$. Then
$$\sup_{f^{-1}(\{0\})}|f'|_\sB^2= \sup_B\{ |f'|^2_\sB+ K f^2\}>0.$$
\item Assume $f^{-1}(\{0\})=\emptyset$. 
Then $K\leq 0$ and $$K_F\geq   K \inf_B f^2 \mbox{ if and only if  }K_F\geq \sup_B\{ |f'|^2_\sB+ K f^2\}.$$
\item Assume $f^{-1}(\{0\})=\emptyset$, $K<0$ and $\sup_B \{(f')^2 + K f^2\}\geq 0$. Then \begin{center}$\sup_B \{(f')^2 + K f^2\}=0= K\inf_B f^2$. \end{center} In particular $\inf_B f=0$. 
\item Assume $f^{-1}(\{0\})=\emptyset$ and  $K= 0$. Then $f$ is constant and $K_F\geq 0$.
\end{enumerate}
\end{corollary}
\subsection{Differential calculus on metric measure spaces}

A metric measure space $X$ is a triple $(X,\de_\sX, \m_\sX)$ where $(X, \de_\sX)$ is a complete, separable metric space and $\m_\sX$ is a locally finite Borel measure on $X$ with full support, i.e. $\m_\sX(B_r(x))\in (0, \infty)$ for every $x\in X$ and $r>0$ sufficiently small.  We also write mm space when we mean metric measure space.  

Let $\gamma:[a,b]\rightarrow X$ be a continuous map. We call $\gamma$ a path in $X$. The length of $\gamma$ is defined via
$$
\sup \sum_{i=1}^N \de_\sX(\gamma(t_{i-1}), \gamma(t_i))=: L^{\sX}(\gamma)$$
where the supremum is w.r.t. to  every partitions $a=t_0\leq t_1 \leq   \dots \leq t_N=b$ of $[a,b]$.

A metric space $(X, \de_\sX)$ is called  intrinsic (or a length space) if  for every pair of points $x,y\in X$ it holds $\de_\sX(x,y)= \inf L^{\sX}(\gamma)$ where the infimum is w.r.t. all rectifiable curves
whose endpoints are $x$ and $y$.  
Assuming $x,y\in X$ admit a rectifiable curve $
\gamma$ connecting them such that $L^\sX(\gamma)=\de_\sX(x,y)$, then the curve $\gamma$ is called a  minimal geodesic or  just  geodesic. If every pair $x,y\in X$ admits a minimal geodesic connecting them,  we call $X$ strictly intrinsic or a geodesic space. 

\subsubsection{Cheeger energy}
We will denote by $\lip(X)$ space of Lipschitz functions and by $\lip_b(X)$ the space of bounded Lipschitz functions on $X$. For $f\in \lip(X)$ we denote the local slope by 
$$\lip f(x):= \limsup_{y\rightarrow x} \frac{|f(x)- f(y)|}{\de_{\sX}(x,y)}.$$

We denote with $L^p(\m_\sX)$ the $L^p$ spaces.
The Cheeger energy $\Ch^\sX: L^2(\m_\sX)\rightarrow [0,\infty]$ is the convex and lower semicontinuous functional defined through
$$\liminf_{ \lip(X)\cap L^2(\m_\sX)\ni f_n\overset{L^2}{\rightarrow} f} \int_X(\lip f_n)^2 \de\m_\sX=:\Ch^\sX(f)$$
The finiteness domain of $\Ch$ equipped with the norm $\left\| f\right\|_{W^{1,2}}^2= \left\| f \right\|_{L^2}^2+ \Ch(f)$ we call $W^{1,2}(X)$.

 It is possible to identify  a function $|\nabla f|= |\nabla f|_\sX\in L^2(\m_\sX)$ such that 
$$\Ch^\sX(f)= \int |\nabla f|^2 \de\m_\sX, \ \ \forall f\in W^{1,2}(X). $$

Consider  $f\in L^2(\m_\sX)$.  A function $G\geq 0$ in $L^2(\m_\sX)$ is called  a weak upper gradient of $f$ if 
\begin{align*}\int |f(e_0(\gamma))- f(e_1(\gamma))| \de \Pi \leq \int  \int_0^1 G(\gamma_t)|\dot \gamma_t| \de t  \de \Pi(\gamma)\end{align*}
for every test plan $\Pi$ on $X$. A test plan is a probability measure $\Pi \in \mathcal P(C([0,1], X))$ such that 
\smallskip
\begin{itemize}
\item There exists $C>0$ such that $(e_t)_\sharp\Pi = \mu_t\leq C\m_{\sX}$ for every $t\in [0,1]$, 
\item It holds $\int_0^1 \int |\dot \gamma(t)|^2 \de \Pi(\gamma) \de t< \infty$. 
\end{itemize}
\smallskip
Then $|\nabla f|$ is {\it the minimal} weak upper gradient in the following sense: if $G$ is a weak upper gradient, then $|\nabla f|\leq G$ $\m_\sX$-a.e.  
\medskip
\begin{remark}\hspace{0mm}
\begin{itemize}
\item
For a Lipschitz function $u$ one has that $\lip u$ is a weak upper gradient, and hence $|\nabla u|\leq \lip u$ but no equality in general. 
\item
A Borel function $g: X\rightarrow [0, \infty)$ is an upper gradient of a continuous function $u: X\rightarrow \R$ if 
$$ |f(e_0(\gamma))- f(e_1(\gamma))| \leq   \int_0^1 G(\gamma_t)|\dot \gamma_t| \de t  $$
holds  for any absolutely continuous curve $\gamma:[0, 1]\rightarrow X$.  An upper gradient $g$ for a continuous function $u$ is also a weak upper gradient.
\end{itemize}
\end{remark}
\begin{lemma}\label{lem:stabilitywug}
 If $u_n\in W^{1,2}(X)\!\rightarrow\!u\in L^2(\m_\sX)$ p.w. a.e. and $|\nabla u_n|$ converges $L^2$-weakly  to $g\in L^2(\m_\sX)$, then $u\in W^{1,2}(X)$ s.t. $|\nabla u|\leq g$ $\m_\sX$-a.e. 
\end{lemma}

A metric measure space $X$ satisfies the Sobolev-to-Lipschitz property if any $f\in W^{1,2}(X)\cap L^\infty(\m_\sX)$ with $|\nabla f|\leq 1$ $\m_\sX$-a.e. has a representative $\tilde f\in \lip_b(X)$ with $\lip(\tilde f)\leq 1$. 

\subsubsection{Laplace operator}
\begin{definition}[\cite{giglistructure}]
Any mm space $X$ such that $\Ch^\sX$ is a quadratic form is said to be infinitesimally Hilbertian. 
\end{definition} Under this assumption there exists a symmetric bilinear form $$(f,g)\in W^{1,2}(X)\times W^{1,2}(X)\mapsto \langle \nabla f, \nabla g\rangle\in L^1(\m_\sX).$$

The Laplace operator $L^\sX: D(L^\sX)\rightarrow L^2(\m_\sX)$ is defined as follows. We say $f\in W^{1,2}(X)$ is in the domain $D(L^\sX)=D_{L^2}(L^\sX)$ of $L^\sX$ if there exists $L^\sX f\in L^2(\m_\sX)$ such that 
$$\int \langle \nabla f, \nabla \phi\rangle \de \m_\sX= - \int L^\sX f \phi \de \m_\sX, \ \ \forall \phi\in W^{1,2}(X).$$
Since $X$ is infinitesimal Hilbertian,  it holds that $L^\sX$ is linear. 

The heat flow $P_t^\sX$ is  the $L^2(\m_\sX)$ gradient flow of $\Ch^\sX$. 
In the case of an infinitesimal Hilbertian mm space $X$  the heat flow $P_t^\sX$ is a linear, continuous and self-adjoint contraction semigroup  characterized by saying that for any $u\in L^2(\m_\sX)$ the curve $t\mapsto P_t^\sX u\in L^2(\m_\sX)$ is locally absolutely continuous in $(0, \infty)$ and satisfies
$$\frac{d}{dt} P^\sX_t u= L^\sX P_t^\sX u \  \mbox{ for } \mathcal L^1\mbox{-a.e. } t\in (0, \infty), \\  \lim_{t\downarrow 0}P_t u= u \mbox{ in } L^2(\m_\sX).$$
The semigroup $P_t^\sX$ has a unique $L^p$ continuous extension from $L^2\cap L^p$ to $L^p$ for any $p\in [1, \infty)$, and by duality a weak$^*$-continuous extension to $L^\infty(\m_\sX)$. 

\subsubsection{Doubling property}
We say that a metric measure space $X$ satisfies a local  doubling property if for every bounded subset $Y$ in $X$ there exists a constant $C_Y>0$ such that for all $x\in X$ and $0<r<\diam(X,\de_\sX)$ with $B_r(x)\subset Y$ we have 
$$\m_\sX(B_{2r}(x))\leq C_Y \m_\sX(B_r(x)).$$
If one choose $Y=X$, then we say $X$ satisfies a global doubling property. 
\subsubsection{Local Poincar\'e inequality}
We say $X$ supports a weak local $(q,p)$-Poincar\'e inequality with $1\leq p\leq q<\infty$ if for every compact subset $Y$ there exists constants $C>0$ and $\lambda\geq 1$ such that for every Lipschitz function $u$, any point $x\in X$ and $r>0$ with $B_{\lambda r}(x)\subset Y$, it holds
\begin{align}\label{ineq:pi}\left( \int_{B_r(x)} |u - {\textstyle \int_{B_r(x)} u\de\m_\sX}|^q \de \m_\sX\right)^{\frac{1}{q}}\leq C r\left( \int_{B_{\lambda r}(x)} \lip u^p \de \m_\sX\right)^{\frac{1}{p}}.
\end{align}
If $\lambda=1$, we say $X$ supports a strong $(p,q)$-Poincar\'e inequality. 
\begin{remark}\hspace{0mm}
\begin{itemize}
\item Under a doubling property a weak local Poincar\'e inequality implies a strong one. 
\item By H\"older's inequality a weak local $(1,p)$-Poincar\'e inequality implies a weak local $(1,p')$-Poincar\'e inequality for $p'\geq p$.  
\item If a metric measure space satisﬁes a local doubling property, Hajlasz and Koskela proved in \cite{koskela} that
a weak local $(1,p)$-Poincaré inequality also implies a $(q,p)$-Poincaré inequality for $q< \frac{pN}{N-p}$ for $N$ such that  the doubling
constant satisﬁes $C\leq 2^N$. 
\end{itemize}
\end{remark}
	\begin{theorem}[\cite{cheegerlipschitz}]\label{th:cheegerlipschitz}
			If $X$ is a complete, locally compact and intrinsic metric measure space that satisfies a doubling property and supports a $(1,2)$-Poincar\'e inequality, then for every function $u:X\rightarrow \R$ that is locally Lipschitz, it holds $\lip u=|\nabla u|$. 
		\end{theorem}

\subsection{Curvature-dimension conditions}\label{subsection:rcd}
We will introduce the Riemannian curvature-dimension condition via its characterization in terms of the Bakry-Emery condition for the associated Cheeger energy.
\smallskip

The {\it carr\'e-du-champ operator} (or {\it $\Gamma_2$-operator}) associated to $\Ch^\sX$ is a multilinear form defined via 
$$\Gamma^\sX_2(u,v;\phi)= \int \langle \nabla u, \nabla v \rangle L^\sX \phi \de\m_\sX - \int \langle \nabla u, \nabla L^\sX v\rangle \phi\de\m_\sX$$
for $u, v\in D_{W^{1,2}}(L^\sX),  \phi \in D_{L^\infty}(L^\sX)\cap L^\infty(\m_\sX).$ We set $\Gamma^\sX(u,u)=:\Gamma^\sX(u)$. 
\begin{definition}\label{def:rcd}
A mm space $X$ satisfies the Riemannian curvature-dimension condition $\RCD(K,N)$ for $K\in \R$ and $N\in [1, \infty)$ if 
\begin{enumerate}
\item $\m_\sX(B_r(x_o))\leq C e^{cr^2}$ for some $c, C>0$ and $x_o\in X$, 
\item $X$ is infinitesimal Hilbertian, 
\item $X$ satisfies the Sobolev-to-Lipschitz property, 
\item The Bakry-Emery condition $\BE(K,N)$ holds:
$$\Gamma^\sX_2(u; \phi) \geq K \int |\nabla u|^2 \phi \de\m_\sX + \frac{1}{N} \int (L^\sX u)^2 \phi \de\m_\sX$$
for any $u\in D_{W^{1,2}}(L^\sX)$ and any $\phi\in D_{L^\infty}(L^\sX)\cap L^\infty(\m_\sX)$ with $\phi\geq 0$. 
\end{enumerate}
\end{definition}
\begin{remark}
Equivalently, a metric measure space $X$ satisfies the condition $\RCD(K,N)$ if and only if $X$ satisfies the curvature-dimension $\CD(K,N)$  in the sense of Lott-Sturm-Villani \cite{stugeo1, stugeo2, lottvillani} and is infinitesimally Hilbertian \cite{giglistructure}.  This was the  definition  proposed in \cite{giglistructure}. The condition $\RCD(K, \infty)$ was introduced in \cite{agsriemannian}.   The main contributions for the equivalence with the properties in the Definition \ref{def:rcd} are \cite{giglikuwadaohta}, \cite{agsbakryemery}, \cite{erbarkuwadasturm}, \cite{amsnonlinear} and \cite{cavmil}. We  refer to \cite{gigli_gromov} for further informations on the historical developments, in particular  the Bibliographical Notes of Section 4.4.
\end{remark}
We collect some properties of $\RCD$ spaces that we need later.
\begin{remark}
The condition $\RCD(K,N)$ for $K\in\R $ and $N\geq 1$ is stable w.r.t.  pointed measured Gromov-Hausdorff convergence. Moreover, the class of  pointed $\RCD(K,N)$ spaces $(X,o)$ such that $\m_X(B_1(o))\leq V$ is compact w.r.t. pointed measured GH convergence \cite{stugeo1, stugeo2, lottvillani, gmsstability, erbarkuwadasturm}.
\end{remark}

\begin{remark}The condition $\CD(K,N)$ (and hence the condition $\RCD(K,N)$) implies the measure contraction property $\MCP(K,N)$ \cite{rajala2}, i.e. for a measurable subset $A\subset X$ (such that $A\subset B_{\scriptscriptstyle\pi\sqrt{(N-1)/K}}(x)$ if $K>0$), there exists an $L^2$-Wasserstein geodesic $\Pi$ such that $(e_0)_\# \Pi= \delta_x$ and $(e_1)_\# \Pi= {\m_\sX(A)}^{-1} \m|_A$ and 
$$\m\geq (e_t)_\#\left( \tau^{(1-t)}_{K,N}(L(\gamma)) \m(A) \Pi\right).$$
\end{remark}
\begin{enumerate}
\item This version of the measure contraction property was introduced by Ohta in \cite{ohtmea} (see also \cite{stugeo2} and \cite{kush_bg}).
\item It is known that a metric measure space $X$ that satisfies the condition $\RCD(K,N)$ or the condition $\MCP(K,N)$ has a local doubling property. If $K\geq 0$, $N=1$ or if $X$ is compact, a global doubling property holds such that the doubling constant $C_X\leq 2^N$. 
\item Moreover, a metric measure space $X$ that satisfies the condition $\RCD(K,N)$ or that is nonbranching and satisfies  $\MCP(K,N)$ supports a weak local $(1,1)$-Poincar\'e inequality. 
\end{enumerate}

\subsubsection{Dirichlet forms}
Given a locally compact metric measure space $X$ we recall that a symmetric, quadratic form $\mathcal E: L^2(\m_\sX)\rightarrow [0, \infty]$ that is $L^2(\m_\sX)$-lower semicontinuous and satisfies the Markov property, is called a Dirichlet form. A Dirichlet form is closed if $D(\mathcal E)=\{ u\in L^2(\m_\sX): \mathcal E(u)< \infty\}$ is a Hilbert space.  A Dirichlet form is called regular if it possesses a core $\mathcal C$, i.e. a subset that is dense in $D(\mathcal E)$ w.r.t. the energy norm and dense in $C_0(X)$ w.r.t. uniform convergence. We say $\mathcal E$ is strongly local if $\mathcal E(u,v)$ whenever $u,v\in D(\mathcal E)$ and $(u+a)v=0$ $\m_\sX$-a.e. for all $a\in \R$. For any such form $\mathcal E$ we have that for any $u\in D(\mathcal E)$ there exists a measure $\de\Gamma(u)$, the energy measure, such that $\mathcal E(u)= \int d\Gamma(u)$. If $\de \Gamma(u)= \Gamma(u)\de\m_\sX$ for any $u\in D(\mathcal E)$, where $\sqrt{\Gamma(u)}\in L^2(\m_\sX)$ we say $\mathcal E$ admits a $\Gamma$-operator.  In this case one can define $D_{loc}(\mathcal E)$ as follows. $u\in D_{loc}(\mathcal E)$ if $u\in L^2_{loc}(\m_\sX)$ and there exists $K$ compact such that there exist $v\in D(\mathcal E)$ and $u=v$ $\m_\sX$-a.e. in $K$. The energy measure defines an intrinsic distance 
$$\de_{\mathcal E}(x,y)= \sup\{u(x)-u(y): u\in D_{loc}(\mathcal E)\cap C(X), \de \Gamma(u)\leq \de\m_\sX\mbox{ on } X\}.$$
The distance $\de_{\mathcal E}$ may be degenerated in the sense that $\de_{\mathcal E}(x,y)=\infty$ and $\de_{\mathcal E}(x,y)=0$ for $x,y\in X$ is possible. 
The Dirichlet form $\mathcal E$ is called strongly regular if it is regular and the topology induced by $\de_\mathcal{E}$ coincides with the topology on $X$. In particular $\de_{\mathcal E}$ is nondegenerated. 
\begin{remark}
 The Cheeger energy $\Ch^\sX$ of an $\RCD$ space $X$ is a strongly local and strongly regular Dirichlet form that admits a $\Gamma$-operator. In particular, a core is given by compactly supported Lipschitz function and  the intrinsic distance $\de_{\Ch^\sX}$ associated to $\Ch^\sX$ coincides with the distance $\de_\sX$ \cite{agsriemannian}.
\end{remark}
\begin{remark}\label{rem:feller}
Given a Dirichlet form on $X$ there is a self-adjoint operator associated to it, as well as semi-group $P_t$, that coincide with the Laplace operator and the heat flow in the case of the Cheeger energy.  A Dirichlet form satisfies the  local doubling property if the  space $(X, \de_{\mathcal E}, \m_\sX)$ satisfies a local doubling property. Similarly, a Dirichlet form $\mathcal E$ supports a weak local $(2,2)$-Poincar\'e inequality if \eqref{ineq:pi} holds for all $u\in D(\mathcal E)$ with $\de\Gamma(u)$ in place of $\lip(u)^2 \de \m_\sX$. If in addition closed balls w.r.t. $\de_{\mathcal E}$ are compact, one can infer the following properties for $P_t$ \cite{sturmdirichlet2, sturmdirichlet3}.
\begin{enumerate}
\item $P_t$  is a Feller semi-group. 
\item $P_t$ is $L^2\rightarrow L^\infty$-ultracontractive. 
\item If $\m(X)< \infty$, then harmonic functions are constant. 
\end{enumerate}
\end{remark}
\noindent
Koskela and Zhou proved the following Theorem \cite{koskelazhou}.
\begin{theorem}\label{th:kz}
Let $\mathcal E$ be a strongly local, strongly regular, symmetric Dirichlet form on $L^2(\m_\sX)$. Assume $X$ equipped with $\de_{\mathcal E}$ and $\m_\sX$ satisfies a doubling property. Then $\lip(X)\subset D_{loc}(\mathcal E)$ and every $u\in \lip(X)$ admits a $\Gamma$-operator such that $\Gamma(u)\leq \lip(u)^2$ $\m_\sX$-a.e.
\end{theorem}
\subsection{Second order calculus on $\RCD$ spaces} For the following we refer to \cite{savareself, giglinonsmooth}.

Let $X$ be an $\RCD(K,N)$ space. 
The space of test functions is
$$ 
\tf(X)= \left\{ f\in D(L^\sX)\cap \lip_b(X):   L^\sX f \in W^{1,2}(X)\right\}
$$
The space $\tf(X)$ is an algebra and for every $f\in \tf(X)$ it holds
\begin{enumerate}
\item Let $-g=\langle \nabla f, \nabla \Delta f\rangle + K |\nabla f|^2.$ Then $\int g\de\m_{\sX}\geq 0$ and $|\nabla f|^2 \in W^{1,2}(X)$ with
$$\Ch(|\nabla f |^2)\leq \lip(f)^2 \left( \left\| |\nabla f|\right\|_{L^2} \left\| |\nabla L^\sX f|\right\|_{L^2} + K^-\left\| |\nabla f|\right\|_{L^2}^2\right).$$
\item There exists a unique finite, signed Borel measure $\mu:= \mu^+ - g\m_\sX$ with $\mu_+\geq 0$ and $\mu_+(X) \leq \int_X g\de\m_\sX$ such that 
\begin{enumerate}
\item every $\Ch$-polar set ist $|\mu|$-neglligible, 
\item the quasi-continuous representative $\tilde \phi$ of any  function $\phi \in W^{1,2}(X)$ is in $L^1(|\mu|)$, 
\item it holds
$$\int \langle \nabla u, \nabla \phi \rangle \de\m_\sX= - \int \tilde \phi \de \mu, \ \ \forall \phi \in W^{1,2}(X).$$
\end{enumerate}
We will write ${\bf L}^\sX u:= \mu$. 
\item We denote by ${\bf \Gamma}_2^{\sX}(f)$ the finite, signed Borel measure 
$${\bf \Gamma}_2^{\sX}(f):= \frac{1}{2} {\bf L}^\sX |\nabla f|^2 - \langle \nabla f , \nabla L^\sX f\rangle \de \m_\sX.$$
${\bf \Gamma}_2^\sX(f)$ has finite total variation, and vanishes on sets of $0$ capacity. We write 
$${\bf \Gamma}_2^{\sX}(f) = \gamma_2^{\sX}(f)\m_\sX + {\bf \Gamma}_2^{\sX, \perp}, $$
where $0\leq {\bf \Gamma}_2^{\sX, \perp}\perp \m_{\sX}$. It holds $$\gamma_2^{\sX}(f)\geq K|\nabla f|^2 + \frac{1}{N} (L^\sX f)^2 \mbox{ $\m_\sX$-a.e. in $X$ }$$
as well as 
\begin{align}\label{ineq:bo}{\bf \Gamma}_2^{\sX}(f) \geq \left[K |\nabla f|^2 + \frac{1}{N} (L^\sX f)^2\right] \m_\sX.\end{align}
\end{enumerate}
\begin{corollary}\label{cor:bochnerext}Let $u\in D_{W^{1,2}}(L^\sX)$ and $\psi= \phi + \lambda$ with $\psi\geq 0$, $\phi \in D_{L^\infty}(L^\sX)\cap L^\infty(\m_\sX)$ and $\lambda\in \R$. Then it holds 
$$ \Gamma_2^\sX(u; \psi)\geq  \int\left[K |\nabla u|^2  +{\frac{{1}}{\textstyle N}  }(L^\sX)^2\right]\psi \de\m_\sX$$
where $\Gamma_2^\sX(u; \psi):= \Gamma_2^\sX(u; \phi) - \lambda\int \langle \nabla u, \nabla L^\sX u\rangle \de\m_\sX.$
\end{corollary}
\begin{proof}
{\bf (1)} We pick $f\in \tf(X)$, and $\psi= \phi+ \lambda$ as in the assumptions.  

Since ${\bf \Gamma}_2^\sX(f)$ is a finite, signed measure and $\psi \in L^\infty(\m_\sX)$, it follows that $\int \psi \de {\bf \Gamma}_2^{\sX}(f)$ is well-defined. From \eqref{ineq:bo} we obtain
$$ \int \psi \de{\bf \Gamma}_2^\sX(f) = \int (\phi+ \lambda) \de {\bf \Gamma}_2^\sX(f)\geq \int \left[ K |\nabla f|^2 + \frac{1}{N} (L^\sX f)^2 \right] \psi \de \m_\sX.$$
{\bf (2)} {\it Claim.} $\int \frac{1}{2} \de{\bf L}^\sX |\nabla f|^2 =0$.

Indeed, we can argue as follows. 
Let $\phi_n \in L^2(\m_\sX)$ such that $\phi_n\uparrow 1$. Then, $P_t \phi_n \uparrow 1$ for every $t>0$.  For every $f\in \tf(X)$ it follows that 
\begin{align*} \int  P_t \phi_n  \frac{1}{2}\de{\bf L}^\sX |\nabla f|^2 =&\int\frac{1}{2} \langle P_t\phi_n, \nabla |\nabla f|^2\rangle \de\m_\sX\\
=& \int \frac{1}{2} L^\sX P_t \phi_n |\nabla f|^2 \de \m_\sX\\
=&\int \frac{1}{2} L^\sX P_{\frac{t}{2}} \phi P_{\frac{t}{2}}|\nabla f|^2 \de\m_\sX\\
=&  \int \frac{1}{2} P_{\frac{t}{2}} \phi_n  L^\sX P_{\frac{t}{2}} |\nabla f|^2\de\m_\sX \\
&\rightarrow \int \frac{1}{2} L^{\sX} P_{\frac{t}{2}}|\nabla f |^2 \de\m_\sX=0 \mbox{ as $n\rightarrow \infty$}. \end{align*}
The left-hand side converges to $\frac{1}{2} \int \de {\bf L}^\sX |\nabla f|^2 .$ This proves the claim. \qed
\smallskip

It follows
\begin{align*}\int \psi \de{\bf \Gamma}_2^\sX(f)=&\int\left[ \frac{1}{2} {\bf L}^\sX |\nabla f|^2  - \langle \nabla f , \nabla L^\sX f\rangle \right] \psi \de \m_\sX\\
=& \int\left[\frac{1}{2} |\nabla f|^2 L^\sX \phi  - \langle \nabla f, \nabla L^\sX f\rangle \psi\right] \de \m_\sX = \Gamma_2(f; \psi)
.\end{align*}
{\bf (3)} Let $u\in D_{W^{1,2}}(L^\sX)$ and choose $u_n\in \tf(X)$ such that $u_n\rightarrow u$ in $D(L^\sX)$. 
Since $\phi, L^\sX\phi \in L^\infty(\m_\sX)$, it follows 
\begin{align*} 
\frac{1}{2} \int |\nabla u_n|^2 L^\sX \phi &\rightarrow \frac{1}{2} \int |\nabla u|^2 L^\sX\phi \de\m_\sX\\
K \int |\nabla u_n |^2 \psi \de \m_\sX&\rightarrow K \int |\nabla u|^2 \psi \de \m_\sX\\
\frac{1}{N} \int (L^\sX u_n)^2 \psi \de \m_\sX& \rightarrow \frac{1}{N} \int (L^\sX u)^2 \psi \de\m_\sX. 
\end{align*}

To treat $\int \langle \nabla u_n , \nabla L^\sX u_n\rangle \psi \de\m_{\sX}$
let us pick $P_t\phi$ in place of $\phi$. Then, since $X$ satisfies the condition $\RCD(K,N)$, we have that $P_t\phi \in \lip_b(X)$. Consequently, we can compute that
\begin{align*}\int \langle \nabla u_n , \nabla L^\sX u_n\rangle (P_t \phi + \lambda) \de\m_{\sX}=& - \int (L^\sX u_n)^2 (P_t \phi  + \lambda) \de\m_\sX\\
& \ \ \ \ - \int \langle \nabla u_n , \nabla P_t\phi\rangle L^\sX u_n \de \m_\sX.
\end{align*}
Since $u_n\rightarrow u$ in $D(L^\sX)$, for $n\rightarrow \infty$ the left- and right-hand side converge to 
\begin{align*}&\int \langle \nabla u , \nabla L^\sX u\rangle (P_t \phi + \lambda) \de\m_{\sX}=\\
&- \int (L^\sX u)^2 (P_t \phi  + \lambda) \de\m_\sX - \int \langle \nabla u , \nabla P_t\phi\rangle L^\sX u \de \m_\sX.\end{align*}
Hence,  the desired inequality follows for $u \in D_{W^{1,2}}(L^\sX)$ and $\psi= P_t\phi +\lambda$. 

Finally, we can replace $P_t\phi$ with $\phi$ by an application of the dominated convergence theorem since $P_t\phi \rightarrow \phi$ pointwise $\m_\sX$-a.e.
\end{proof}

\subsection{Weighted $1$-dimensional Riemannian manifolds}
Let $B$ be a complete $1$-dimensional Riemannian manifold, i.e.  $B$ is isometric to $[0,\pi]$, $[0, \infty)$, $\R$ or $\mathbb S^1$.  We  write $u'v'= \langle u', v' \rangle_\sB$ for smooth functions $u, v$ on $B$.

Let $f: B\rightarrow [0,\infty)$ be smooth and not identical $0$.
\smallskip

\noindent
{\bf We assume} that $$\partial B= f^{-1}(\{0\})$$ 
in case either set is non-empty.
%
\begin{remark}
If we assume that\begin{center} $f''+K\leq 0$ and $\partial B\subset f^{-1}(\{0\})$, \end{center}
then it follows  from the maximum principle that $\partial B= f^{-1}(\{0\})$.
\end{remark}
\smallskip

We set $\m_\sB^\sN=f^\sN\vol_\sB$. Then we consider the pair $(B, \m_\sB^\sN)$ and the  quadratic form 
$$\mathcal E^{\sB, \sN}(u)= \int |u'|^2 \de \m_{\sB}^\sN, \ \ u\in C^\infty(B).$$

The form closure of $\mathcal E^{\sB, \sN}$ in $L^2(\m^\sN_\sB)$ is the Cheeger energy  $\Ch^{\sB, \sN}$ of the metric measure space $(B,\m_\sB^\sN)$ where the domain of $\Ch^{\sB, \sN}$ is $W^{1,2}(B, \m_\sB^\sN)$, the space of $L^2$ Sobolev functions. 

We can also consider the closure $W_0^{1,2}(\mathring B, \m_\sB^\sN)$  of $C^\infty_c(\mathring B)$ in $W^{1,2}(B, \m_\sB^\sN)$ and the following lemma explains the relation between $W^{1,2}(B, \m_\sB^\sN)$ and $W_0^{1,2}(B, \m_\sB^\sN)$.
\begin{lemma} If $N\geq 1$,
$W_0^{1,2}(\mathring B, \m_\sB^\sN)= W^{1,2}(B, \m_\sB^\sN).$
\end{lemma}
\begin{proof}
First we recall the $2$-capacity of a set $K\subset B$: 
$$\mbox{Cap}_2(K)= \inf \left\| u\right\|_{W^{1,2}}$$
where the infimum is w.r.t. all functions $u$ such that $0\leq u\leq 1$ and $u\geq 1$ on a neighborhood of $K$. 

We can  construct a Lipschitz function such that $u\equiv 1$ on $B_{\frac{\epsilon}{3}}(\partial B)$, $u\equiv 0$ on $B\backslash B_{\epsilon}(\partial B)$ and  $|u'|_\sB \leq \frac{3}{\epsilon}$.  Since $f$ is $\lambda$-concave, there exists $g>0$ such that  $$f(t)\leq f(t_0) + g (t-t_0) + o(t) \ \forall t_0\in \partial B.$$ One can compute that
$$\int_B |u'|_{\sB}^2 \m_\sB^\sN \leq \int_{\frac{\epsilon}{3}}^\epsilon \left(\frac{3}{\epsilon}\right)^2 f^\sN(t) dt\leq C(f, N) \frac{1}{\epsilon^2}\int_{\frac{\epsilon}{3}}^\epsilon t^\sN dt\leq  C(f, N) \epsilon^{\sN-1} .$$
Applying Mazur's lemma and letting $\epsilon\downarrow 0$ we find  a sequence that strongly converges to $0$  in $W^{1,2}$. It follows that $\mbox{Cap}_2(\partial B)=0$. 
%
\end{proof}
The space $C_c^\infty( B)$ is a core of $\Ch^{\sB, \sN}$. A core is a subset of $W^{1,2}(B, \m_\sB^\sN)\cap C_c(B)$  that is dense in $W^{1,2}(B, \m_\sB^\sN)$ w.r.t. to the norm 
$$\left\| u\right\|^2_{W^{1,2}}= \left\| u\right\|^2_{L^2} + \Ch^{\sB, \sN}(u)$$
and dense in $C_c(B)$ w.r.t. uniform convergence. 

The domain $D(L^{\sB, \sN})$ of the generator $L^{\sB, \sN}$ associated to $\Ch^{\sB, \sN}$ is the set of $u\in W^{1,2}(B, \m_\sB^\sN)$ such that $ \exists g\in L^2(\m_\sB^\sN) \mbox{ with }$$$ \Ch^{\sB, \sN}(u,v)= -\int g v \m_\sB^\sN \ \forall v\in W^{1,2}(B, \m_\sB^\sN).$$
We write $L^{\sB, \sN} u:= g$ and for $u\in C^\infty_{c} (\mathring B)$ it follows that
$$
L^{\sB, \sN}u= u'' - \frac{N}{f} \langle f', u'\rangle_\sB.
$$
\begin{proposition} Let $f: B\rightarrow [0,\infty)$ be a smooth function such that $\partial B= f^{-1}(\{0\})$. Then the following statements are equivalent: 
\begin{enumerate} 
\item $f$ satisfies $f''+ K f\leq 0$, 
\item The space $(B, \m^\sN_\sB)$ satisfies the condition $\RCD(KN,N+1)$.  
\end{enumerate}
\end{proposition}
\begin{proof}
%
We note that $f$ is smooth and therefore the $\RCD(KN,N+1)$ condition holds if and only if   the Bakry-Emery $N$-Ricci tensor satisfies \begin{center}$\ric^{N+1, f}_\sB= - \frac{1}{N}\frac{f''}{f} g_\sB \geq  K g_\sB \mbox{ on } B\backslash \partial B$
\end{center}Hence $f''+ Kf\leq 0$ on $B$.
\end{proof}
\begin{example}
Let $B=[0, \pi]$ and $f(r)= \sin(r)$. Then the mm space $([0,\pi], f^{N-1}\de r)$ satisfies the condition $\RCD(N-1, N)$. 
\end{example}
\subsection{Schr\"odinger operators on $1$-dimensional spaces}\label{subsec:schroedinger}
As in the previous section we assume that $B$ is a $1$-dimensional Riemannian manifold and $ f\in C^2(B)$ s.t. $f\geq 0$ 
and $\partial B= f^{-1}(\{0\})$.
\smallskip

{\bf We  assume $N>1$} and consider the measure $\mu= f^{N-2} \vol_\sB$ and a  constant $\lambda>0$. One can define a quadratic form 
$$\mathcal E^{\sB, \sN, \lambda}(u)= \Ch^{\sB, \sN}(u)+\lambda \int u^2 \dint\mu$$
for  $u \in W^{1,2}(B, \m_\sB^\sN)$ with 
$$\lambda\int u^2 \dint \mu = \lambda\int u^2 f^{N-2}\dint \vol_\sB<\infty.$$
It is known that $\mathcal E^{\sB, \sN, \lambda}$ is a Dirichlet form in $L^2(\m_{\sB}^{\sN})$ and  has the domain 
$$D(\mathcal E^{\sB, \sN, \lambda})= \{u \in W^{1,2}(B, \m_\sB^\sN): \lambda \int u^2 f^{N-2} \dint \vol_\sB<\infty\}.$$
We recall the following facts:
\begin{fact}\label{fact:newfact}
Since $N>1$,  $\int u^2 f^{N-2} \dint \vol_\sB<\infty$ for $u\in C_c^\infty(B)\neq C_c^\infty(\mathring B)$, $C_c^\infty(B)$ is a core of $\mathcal E^{\sB, \sN, \lambda}$, and $C_c^\infty(\mathring B)$ is dense in $D(\mathcal E^{\sB, \sN, \lambda})$. 
\end{fact}
The  generator $L^{\sB, \sN, \lambda}$ of $\mathcal E^{\sB, \sN, \lambda}$  is defined as follows.  We say $u\in D(\mathcal E^{\sB, \sN, \lambda})$ is in the domain $D(L^{\sB, \sN, \lambda})$ of $L^{\sB, \sN, \lambda}$ if $\exists g=: L^{\sB, \sN, \lambda} u\in L^2(\m_\sB^\sN)$ such that
$$\int \langle u', v'\rangle_\sB \dint \m_\sB^\sN + \lambda \int u v \dint \mu= - \int g v \dint \m_\sB^\sN\ \forall v\in D(\mathcal E^{\sB, \sN, \lambda}).$$
\begin{fact}  For $u\in C_c^\infty(\mathring B) $ it follows that  $u\in D(L^{\sB, \sN, \lambda})$ and
$$L^{\sB, \sN, \lambda} u= L^{\sB, \sN} u - \frac{\lambda }{f^2} u = u'' + \frac{N}{f} \langle f', u'\rangle_\sB - \frac{\lambda}{f^2} u.$$
\end{fact}
\begin{proof}
Using the Leibniz rule, and since $u, v, f$ are smooth, we compute that
$$\int \langle u', v'\rangle_\sB \dint \m_\sB^\sN = \int \langle u', (v f^\sN)' \rangle_\sB \dint \vol_\sB - \int \frac{N}{f}\langle u', f' \rangle_\sB v \dint \m^\sN_\sB.$$
Since $u\in C_c^\infty(\mathring B)$, we also have
$$\int \langle u', (v f^\sN)' \rangle_\sB \dint \vol_\sB= - \int u'' v \dint\m_\sB^\sN.$$
Hence 
$$-\int L^{\sB, \sN, \lambda } u v \dint\m_\sB^\sN= - \int u'' v \dint\m_\sB^\sN  - \int \frac{N}{f} \langle f', u'\rangle_\sB v \dint \m_\sB^\sN  + \lambda \int \frac{uv}{f^2} \dint \m_\sB^\sN$$
for all $v\in C_c^\infty(\mathring B)$. This implies the formula and $L^{\sB, \sN, \lambda}u\in L^2(\m_\sB^\sN)$.
\end{proof}
\subsubsection{Essentially self-adjointness}
We say a set $\mathcal C\subset D(L^{\sB, \sN, \lambda})$ is dense in the domain of the operator $L^{\sB, \sN, \lambda}$ if the domain $D(L^{\sB, \sN, \lambda})$ is the closure of $\mathcal C$ w.r.t. the graph norm 
$$\left\| u\right\|^2_{D(L^{\sB,\sN,\lambda})} = \left\| u\right\|_{L^2(\m_\sB^\sN)}^2 + \left\| L^{\sB, \sN, \lambda}u\right\|_{L^2(\m_\sB^\sN)}^2 . $$
The operator $L^{\sB, \sN, \lambda}|_{\mathcal C}$ restricted to $\mathcal C$ is called essentially self-adjoint if it has a unique  self-adjoint extension. 

It  is a general fact about essentially self-adjoint operators that in  this case $\mathcal C$ is dense w.r.t. $\left\|\cdot \right\|_{D(L^{\sB,\sN,\lambda})}$ in the domain of this extension. 
\begin{proposition}\label{prop:essential}Assume $B\subset \R$ is a closed interval, $f: B\rightarrow [0, \infty)$ is smooth,  $\partial B= f^{-1}(\{0\})$, and $\max_{r\in \partial B} |f'(r)|\leq 1.$

Assume $\lambda > 1$ if $f^{-1}(\{0\})\neq \emptyset$. Let $\mathring B= B\backslash f^{-1}(\{0\})$. 
Consider the operator $L^{\sB, \sN, \lambda}$ for $u\in C^\infty_{c}(\mathring B)$. 

Then 
$L^{\sB, \sN, \lambda}|_{ C^\infty_{c}(\mathring B)}$
is essentially self-adjoint. 
\end{proposition}
{
\begin{proof} 
{\bf 1.}
If $f^{-1}(\{0\}) =\partial B=\emptyset$, the statement can be deduced from general principles about essentially self-adjoint operators \cite{reedsimon}.
\smallskip\\
{\bf 2.}
$f^{-1}(\{0\})=\partial B\neq \emptyset$. In this case $B\simeq [0, a]$ or $B\simeq [0, \infty)$. \begin{center}It holds $|f'|\leq 1$ on $f^{-1}(\{0\})= \partial B$. \end{center}
We consider the orthogonal transformation $U: L^2(B, \m_\sB^\sN)\rightarrow L^2(B, \mathcal L^1|_B)$  that is given by 
$$ U(\phi)= f^{\sN/2} \phi, \ \mbox{ as well as } \ U^{-1}(\phi)=f^{-\sN/2} \phi.$$
The tranformation $U$ leaves $C_0^\infty(\mathring B)$ invariant and 
\begin{align*}-U L^{\sB, \sN, \lambda} U^{-1} &= -U\left( \frac{d^2}{dr^2} +  \frac{N}{f} \frac{d f}{dr} \frac{d}{dr} - \frac{\lambda}{f^2}\right) U^{-1} \\
&= -\frac{d^2}{dr^2} + \left( \frac{N^2 -2 N}{4} {\left( \frac{df}{dr}\right)^2}+\frac{N}{2} f \frac{d^2f}{dr^2}+\lambda \right)\frac{1}{f^2}\\
&=:- \frac{d^2}{dr^2} + V.\end{align*}

We set $T=-\frac{d^2}{dr^2} + V(r)$.
A sufficient condition for $T|_{C_c^{\infty}(\mathring B)}$ being essentially self-adjoint is, by Theorem X.7 in \cite{reedsimon},  that $T=-\frac{d^2}{dr^2} + V(r)$ is in the limit point case  at all points $r\in \partial B$. For instance, if $B=[0, a]$, this  follows if  $V(r)\geq \frac{3}{4 r^2}$ in a neighborhood of $0$ and $V(r-a)\geq \frac{3}{4 (a-r)^2}$ in a neighborhood of $a$ \cite[Theorem X.10]{reedsimon}.

Moreover, we compute 
\begin{align*}V(r)&=\left( \frac{N^2 - 2N + 1}{4} (f')^2- \frac{1}{4} {(f')^2}+ \frac{N}{2} f f'' + \lambda\right) \frac{1}{f^2}\\
&= \left( \frac{(N-1)^2}{4} (f')^2 - \frac{1}{4} (f')^2 + \frac{N}{2} f f'' + \lambda \right) \frac{1}{f^2}\\
&\geq \left(  - \frac{1}{4} (f')^2 + \frac{N}{2} f f'' + \lambda \right) \frac{1}{f^2}\\
&\geq \left( - \frac{1}{4} + \lambda \right) \frac{1}{r^2}.\end{align*}
Since $\lambda> 1$, it follows that $V(r) > \frac{3}{4 r^2}$ in a neighborhood of $0$. 
\end{proof}}

\section{Warped products  over a $1$-dimensional base space}
{
We consider a metric measure space $(F, \de_\sF, \m_\sF)$ such that $(F, \de_\sF)$ is a complete and locally compact length space, hence  also a geodesic space, and such that $\m_\sF$ is a locally finite measure with full support.  For instance, we can assume that $F$ satisfies a Riemannian curvature-dimension condition.}  

The following statements about warped products are often valid for the more general case when $B$ is arbitrary geodesic metric space. 
However, in view of our main results, we will consider  $B$ as before, that is a $1$-dimensional Riemannian manifold. Let $f: B\rightarrow [0, \infty)$ be a Lipschitz function. 

We call $\gamma=(\alpha, \beta): [a,b] \rightarrow B \times F$ admissible if $\alpha$ and $\beta$ are Lipschitz continuous. We note that every rectifiable curve admits a reparametrization that is Lipschitz. For $\gamma$ that is admissible we define its length as
$$\L(\gamma)= \int_a^b \sqrt{|\dot \alpha|^2 + (f\circ \alpha)^2 |\dot \beta|^2} \ dt.$$

The warped product metric $\de_{\smwp}$ on $B\times F$ is defined as the intrinsic metric associated to the length structure $\L$, i.e. for two points $(p,x)$ and $(q,y)$ we define 
$$\de_{\smwp}((p,x), (q,y)):= \inf \L(\gamma)$$
where the infimum is w.r.t. all admissible curves $\gamma$ s.t. $\L(\gamma)<\infty$, that connect the points $(p,x)$ and $(q,y)$.  The infimum is finite since there are rectifiable curves between $p$ and $q$ in $B$, and between $x$ and $y$ in $F$.  $\de_{\smwp}$ is symmetric and satisfies the $\triangle$-inequality. 

\begin{definition} The warped product metric space $B\times_f F$ between $B$, $F$ and $f$ is given by 
	$$( B\times F/\sim, \de_{\smwp}) \  \mbox{where } (p,x)\sim (q,y)  \Longleftrightarrow  \de_{\smwp}((p,x), (q,y))=0.$$
	The warped product $B\times_f F$ is the intrinsic metric space associated to the length structure $\L$.

We also write $[(p,x)]$ for the equivalence class of $(p,x)$ w.r.t. $\sim$.
\end{definition}

\begin{remark}[Topology of a warped product]
	If $(p,x)\in B\times F$ is a point such that $f(p)>0$, then one can easily check that the topology of $\de_{\smwp}$ in a neighborhood of $[(p,x)]\in\mwp$ coincides with the product topology of $B\times F$. 
	
	Let $[(p,x)]\in B\times F/\sim$   be a point where  $f(p)=0$. 
	If $[(q,y)]\neq [(p,x)]$ such that $p\neq q$, then  an dmissible path  $\gamma=(\alpha, \beta)$ always satisfies
	$$\L(\gamma)\geq \L^\sB(\alpha)\geq \inf \L^\sB(\alpha)>0$$
	where the last infimum is then w.r.t.   curves $\alpha$ such that  $\gamma=(\alpha, \beta)$  is rectifiable, $\beta$ is constant and $\alpha$  connects $p$ and $q$.  In particular, for a minimizer $\gamma= (\alpha, \beta)$ it follows that $\beta$ is constant and $\alpha$ is a minimizer in $F$.  
	
	If $[(q,y)]\neq [(p,x)]$ such that $x\neq y$ but $p=q$ and $f(p)=f(q)\neq 0$, then we have for every admissible path $\gamma$ that 
	$\L(\gamma)\geq f(p) \L^\sF(\beta)>0$. If $f(p)=f(q)=0$, then one can check that the infimimum of $\L(\gamma)$ w.r.t. all admissible paths connecting $[(p,x)]$ and $[(q,y)]$ is $0$ (for instance consider a small loop $\alpha$ in $B$). This would imply $[(p,x)]=[(q,y)]$ which is a contradiction. Hence, if $[(q,y)]\neq [(p,x)]$ such that $x\neq y$ and $p=q$, it follows $f(p)>0$.
	
	Therefore  $L$ is consistent with the topology of $B\times F/\sim$ in the sense of \cite[Chapter 2]{bbi}, and hence $L$ is a lower semi-continuous length structure on the class of admissible paths . For this we also note that every admissible path is also continuous in $B\times F/\sim$.
	
	As  a consequence we obtain that  the induced length of $\de_{\smwp}$ coincides with the length structure $L$ according to \cite[Theorem 2.4.3]{bbi}.
\end{remark}

\begin{theorem}[Alexander-Bishop, {\cite{albi0}}]\label{th:albi0}
	Let $\gamma=(\alpha, \beta)$ be a minimizer w.r.t $L$ in $B\times_f F$ parametrized proportional to arclength. Assume $f>0$. Then 
	\begin{enumerate}
		\item[(a)] $\beta$ is a minimizer in $F$; 
		\item[(b)] (Fiber independence) $\alpha$ is independent of $F$, except for the total height, i.e. the length $L^\sF(\beta)$ of $\beta$. More precisely, if $\hat F$ is another strictly intrinsic metric space and $\hat \beta$ is a minimizing geodesic in $\bar F$ with the same length and speed as $\beta$, then $(\alpha, \hat \beta)$ is a minimizer in $B\times_f\hat F$. 
		\item[(c)] (Energy equation, version 1) $\beta$ has speed $\frac{c_\gamma}{f^2\circ \alpha}$ for a constant $c_\gamma$;
		\item[(d)] (Energy equation, version 2) 
		$\alpha$ satisfies $\frac{1}{2} |\alpha'|^2 + \frac{1}{2f^2\circ \alpha}= E$ a.e. where $E$ is the proprotionality constant of the parametrization of $\gamma$. 
	\end{enumerate}
\end{theorem}
\begin{remark} If we assume that $B$ and $F$ are locally compact, complete, strictly intrinsic metric spaces, the existence of minimizing curves $\gamma=(\alpha, \beta)$ for $L$ is guaranteed by the Arzela-Ascoli theorem. In particular, one has the following corollary
\end{remark}
\begin{corollary} If $B$ and $F$ are locally compact, complete, instrinsic metric spaces, then the warped product $\mwp$ is  a locally compact, complete and  intrinsic metric space. 
\end{corollary}
We also recall the following general statement about warped products and Alexandrov  lower curvature bounds. For the definition of  Alexandrov curvature bounded from below, CBB,  we refer to \cite{bbi}. We assume that the Hausdorff dimension is finite. 

\begin{theorem}[Alexander-Bishop, {\cite{albi}}]\label{th:albi1} Let $B$ and $F$ be complete, locally compact intrinsic metric spaces. Let $f: B\rightarrow [0, \infty)$ be a Lipschitz function.

Then the  warped product 
$B\times_ f F$ has CBB by $K$ if and only if 
\begin{enumerate}
	\item \begin{itemize}\item[(a)] $B$ has CBB by $K$,\item[(b)] $f$ is $Kf$-concave,
\item[(c)] If $B^\dagger$ is the result of gluing two copies of $B$ along the closure of the set of boundary points where $f$ is nonvanishing, and $f^\dagger: B^\dagger \rightarrow [0,\infty)$ is the tautological extension of $f$, then $B^\dagger$ has CBB by $K$ and $f^\dagger$ is $fK$-concave. 
\end{itemize}

\item $F$ has CBB by $K_F=\sup_B\{ |D f|^2 {\displaystyle+} Kf^2\}$, 

\end{enumerate}
\end{theorem}

\subsubsection{$N$-warped products} For $N\in [1, \infty)$
a measure on $B\times_f F$ is defined via 
$$f^\sN \vol_\sB\otimes \m_\sF= \m_\sB^\sN \otimes \m_{\sF} =: \m^\sN.$$
\begin{definition} For $N\in [1, \infty)$ the  metric measure space $$(B\times_f F, \m^\sN)=: B\times_f^\sN F =: C$$ is called the $N$-warped product between $B$, $f$ and $F$.  
\end{definition}
\begin{example}
Let us choose  $F=[0,L]$ and $N=1$. Then the warped product metric on $B\times [0, L]$ w.r.t. $f$ coincides with the induced metric of  the continuous Riemannian metric  $g=(\de t)^2 + f^2 (\de r)^2$, and the measure $f(t)\mathcal L^1(\de t)\otimes \mathcal L^1(\de r)$ is the Riemannian volume measure of $g$, that is also the $2$-dimensional Hausdorff measure of the metric. 
\end{example}
\subsection{Energy functionals on warped products}
We will assume $N>1$. Let $\Ch^\sF$ be the Cheeger energy of $F$. For $u\in \lip(F)$ let $|\nabla u|_\sF$ be the minimal weak upper gradient.   Let $B$ and $f$ be as in Subsection \ref{subsec:schroedinger}.

Then we consider 
$$C_c^\infty({B})\otimes \lip(F)= \left\{ \sum_{i=1}^k u_1^i u_2^i: k\in \N, u_1^i\in C^\infty_c({B}), u_2^i \in \lip(F)\right\}$$
For $u\in C^\infty_c({B})\otimes \lip(F)$ we define 
\begin{align*}|\nabla u|^2_{*}(t,x) :=& \sum_{i=1}^k |\nabla u^i_1|_{\sB}^2(t) (u_2^i)^2(x) +\frac{1}{f^2(t)} \sum_{i=1}^k(u_1^i)^2(t) |\nabla u_2^i|^2_\sF(x)
\\
=& |\nabla\left( \sum_{i=1}^k u^i_1(\cdot) u_2^i(x)\right)|_{\sB}^2(t)+ \frac{1}{f^2(t)} |\nabla \left(\sum_{i=1}^k u^i_1(t) u_2^i(\cdot)\right)|_\sF^2 (x)\\
=& |\nabla u^x|_{\sB}^2(t) + \frac{1}{f^2} |\nabla u^t|_{\sF}^2(x)  \  
\end{align*}$\mbox{for $\m^\sN$-a.e. $(t,x)\in B\times F$}$
where $u^x= u(\cdot, x)$ and $u^t= u(t, \cdot)$.
\smallskip

We consider a  quadratic form for $u\in C_c^\infty(B)\otimes \lip(F)$ defined via 
\begin{align*}\mathcal E^*(u) =& \int_{B\times F} |\nabla u|_{*}^2\dint \m^\sN\\
=& \int_F \Ch^{\sB, \sN}(u^x)\de \m_\sF(x) + \int_B \Ch^\sF(u^r) f^{N-2}(r)\de \mathcal L^1(r).
\end{align*}
In particular, it holds that  $$t\in B\mapsto \sum_{i=1}^k(u_1^i)^2(t) \Ch^\sF(u^i_2)$$ is integrable w.r.t. $f^{\sN-2}(t) \de \mathcal L^1(t)$ and hence $\int_B \Ch^\sF(u^r) f^{N-2}(r)\de \mathcal L^1(r)<\infty$ for $u\in C_c^\infty(B)\otimes \lip(F)$. 

The quadratic form $\mathcal E^*$ defined on $C_c^\infty(B)\otimes \lip(F)\subset L^2(\m^\sN)$ is closable. 
\begin{definition}
The {\bf $N$-Skew product} between $B$, $f$ and $\Ch^\sF$ is the  closure of the quadratic form
$\mathcal E^*
$ 
in $L^2(\m^\sN)$, that we also denote with $
\mathcal E^*$.

The underlying topological space is $B\times F/\sim$ where
$$(s,x)\sim (t, y) \ \Longleftrightarrow \ \  \begin{cases} s=t, x=y \mbox{ if } s \mbox{ or } t \mbox{ are  in } \mathring B\\ 
s=t
\end{cases}
$$
Let $D(\mathcal E^*)$ be the domain of ${\mathcal E}^*$ in $L^2(\m^\sN)$ equipped with $\|\cdot \|_{\mathcal E^*}^2= \left\|\cdot \right\|_{L^2}^2 + \mathcal E^*.$
\begin{remark}\label{properties:estar}\hspace{0mm}
\begin{enumerate}
\item
Directly from the definition of the closed form $\mathcal E^*$ and the underlying topology we see that  $C_c^\infty(B)\otimes \lip(F)$ is a core.  Hence $\mathcal E^*$ is a strongly local, regular Dirichlet form. 
The associated generator  $L^*$ of $\mathcal E^*$ is defined in the same way as the  Laplace operator associated to the Cheeger energy of a metric measure space.
\item The set $C_c^\infty(\mathring B)\otimes \lip(F)$ is dense in $D(\mathcal E^*)$. Indeed, given \begin{center} $u= \sum_{i=1}^k u_1^iu_2^i\in C^\infty_c(B)\otimes \lip(F)$ \end{center} we can approximate each $u_1^i$ with functions $\tilde u_1^i\in C_c^\infty(\mathring B)$ in  both $W^{1,2}(B, \m^\sN_\sB)$ and $L^2(f^{\sN-2}\dint \vol_\sB)$ according to Fact \ref{fact:newfact}.  Then, it follows that 
$$\Ch^{\sB, \sN}(\tilde u^x) \rightarrow \Ch^{\sB, \sN}(u^x) \ \ \mbox{ for } \m^\sF\mbox{-a.e. } x\in F$$
and hence $\int \Ch^{\sB, \sN}(\tilde u^x) \dint \m_\sF(x) \rightarrow  \int \Ch^{\sB, \sN}( u^x) \dint \m_\sF(x)$, by the dominated convergence theorem, and it also follows that 
$$\int \Ch^\sF(\tilde u^t) f^{\sN-2}(t) \dint t \rightarrow \int \Ch^\sF( u^t) f^{\sN-2}(t) \dint t.$$
\item The Dirichlet form $\mathcal E^*$ admits a $\Gamma$-operator $\Gamma^*$, i.e. 
$$\mathcal E^*(u)= \int \Gamma^*(u,u) \de \m^\sN \ \forall u\in D(\mathcal E^*)$$
where  $u\in D(\mathcal E^*)\mapsto \Gamma^*(u,u)\in L^1(\m^\sN)$ is a positive semidefinite, symmetric bilinear form.  

Strong locality of $\mathcal E^*$ implies strong locality of $\Gamma^*$, that is equivalent to the Leibniz rule. 
\item
 If ${\mathcal E^*(u_n-u)}\rightarrow 0$, then 
$$\int f  \Gamma^*(u_n, u_n)\de \m^\sN\rightarrow \int f \Gamma^*(u,u)\de\m^\sN \ \forall f\in L^\infty(\m^\sN).$$
\end{enumerate}
\end{remark}
\end{definition}

We assume that the  operator $L^\sF$ associated to $
\Ch^\sF$ has a discrete spectrum \begin{center}$\lambda_0\leq \lambda_1\leq \lambda_2 \leq \dots \subset \R_{\geq 0}$. \end{center} 
This is the case when the mm space $F$ has finite measure and satisfies a volume doubling condition and a weak local Poincar\'e inequality, for instance, if $F$ is a compact $\RCD(K,N)$ space with $N\in [1, \infty)$.

Let  $E(\lambda_i)$ be the eigenspace of $\lambda_i$.  The first eigenvalue $\lambda_0$ is $0$, and $E_0$ are the constant real functions on $F$. 

\begin{proposition}\label{prop:wpop}\hspace{0mm}
\begin{enumerate}
\item  It holds that $ C_c^\infty(\mathring B)\otimes D(L^F) \subset D_{L^2}(L^*)$, and 
$$(L^{*} u)(r,x)= (L^{\sB, \sN} u^x)(r) + \frac{1}{f^2(r)} (L^\sF u^r)(x)$$
for $\m^\sN\mbox{-a.e. } (r,x)\in B\times F$ and $u\in  C_c^\infty(\mathring B)\otimes D(L^\sF).$
\smallskip
\item If $u_2\in E(\lambda)$ and $u_1\in D(L^{\sB, \sN, \lambda})$, then we have $u_1\otimes u_2\in D(L^*)$ and 
$$L^* u= L^{\sB, \sN, \lambda} u_1 \otimes u_2.$$
\end{enumerate}
\end{proposition} 
\begin{proof}
(1) We pick $u\in C_c^\infty(\mathring B) \otimes D(L^\sF)$ and $v\in C_c^\infty(\mathring B) \otimes W^{1,2}(F)$ and compute
\begin{align*}
\mathcal E^*(u,v)=&
\int \int  \langle (u^x)', (v^x)'\rangle_\sB  \dint \m^{\sN}_\sB \dint \m_\sF(x)\\
&+ \int \frac{1}{f^2} \int \langle \nabla u^r, \nabla u^r\rangle_\sF \dint \m_\sF \dint \m_\sB^\sN(r)\\
=&  -\int \int \left[ L^{\sB, \sN} u^x\right] v^x \dint \m_\sB^\sN \dint \m_\sF\\
&- \int \int f^{-2}  \left[L^\sF u^r \right]v^r \dint \m_\sF \dint \m_\sB^\sN\\
=& - \int \left[ L^{\sB, \sN} u^x + f^{-2} L^\sF u^r\right] v \dint \m^\sN.
\end{align*}
Since $C_c^\infty(\mathring B)\otimes W^{1,2}(F)$ is dense in $D(\mathcal E^*)$ w.r.t. $\left\| \cdot \right\|_{\mathcal E^*}$, this identity extends to all $v\in D(\mathcal E^*)$. 

Moreover, since $r\mapsto u(r,x), L^{\sB, \sN}u(r,x), L^\sF u(r,x)$ belong to $C^\infty_c(\mathring B)$ for $\m_\sF$-a.e. $x\in F$, it follows that 
$$L^{\sB, \sN}u^x + f^{-2} L^\sF u^r \in L^2(\m^\sN).$$
Hence $u\in D_{L^2}(L^*)$ and the desired formula for $Lu$ holds. 
\smallskip

(2) We pick $u_2\in E(\lambda)$ and $u_1\in D(L^{\sB, \sN, \lambda})$, and set $u_1 \otimes u_2 =u$. We notice first that $u^x(r)= u_1(r) u_2(x)$, and hence $L^{\sB, \sN, \lambda} u^x= \left[ L^{\sB, \sN, \lambda} u_1\right] u_2(x)\in L^2(\m^\sN)$.  We compute for any $v\in C_c^\infty(\mathring B) \otimes W^{1,2}(F)$
\begin{align*}&
\mathcal E^*(u,v)= \\
&\int \int  \langle (u^x)', (v^x)'\rangle_\sB  \dint \m^{\sN}_\sB \dint \m_\sF(x)+ \int \frac{1}{f^2} \int \langle \nabla u^r, \nabla u^r\rangle_\sF \dint \m_\sF \dint \m_\sB^\sN(r)\\
&=\int \int  \langle (u^x)', (v^x)'\rangle_\sB  \dint \m^{\sN}_\sB \dint \m_\sF(x)- \int \frac{\lambda}{f^2} \int u^r v^r \dint \m_\sF \dint \m_\sB^\sN(r)\\
&=\int \int  \left[\langle (u^x)', (v^x)'\rangle_\sB  -  \frac{\lambda}{f^2}  uv\right] \dint \m_\sF \dint \m_\sB^\sN(r)\\
&= \int \mathcal E^{\sB, \sN, \lambda}(u^x, v^x) \dint \m_\sF(x) \\
&=- \int \int L^{\sB, \sN, \lambda} u^x v^x \dint \m_\sB^\sN \dint \m_\sF=- \int \left[L^{\sB, \sN, \lambda} u\right] v \dint \m^\sN
\end{align*}
This identity again extends to all $v\in D(\mathcal E^\snwp)$. Therefore we obtain the claim. 
\end{proof}
\begin{proposition}
Let $P_t^{\sB, \sN, \lambda}$ the semi-group induced by $L^{\sB, \sN, \lambda}$. 
For $u=u_1\otimes u_2\in C^\infty_c(\mathring B)\otimes E(\lambda)\subset D_{L^2}(L^\sC)\cap L^\infty(\m_\sF)\cap \lip(F)$ we have that 
$$P_t^* u= P_t^{\sB, \sN, \lambda} u_1 \otimes u_2$$
where $P_t^*$ is the semi-group associated to $L^*$. 

In particular, for $u=u_1\otimes u_2\in C^\infty_c(\mathring B)\otimes E(\lambda)$ we have a  formula for $P_t^* u$ in terms of $P_t^{\sB, \sN, \lambda} u_1$ and $u_2$. 
\end{proposition}
\begin{proof}
Indeed, since $P_t^{\sB, \sN, \lambda} C_c^\infty(\mathring B)\subset D(L^{\sB, \sN, \lambda})$ and since
$$\frac{d}{dt} P_t^{\sB, \sN, \lambda} u_1 = L^{\sB, \sN, \lambda} P_t^{\sB, \sN, \lambda} u_1, \ \ u_1\in C_c^\infty(\mathring B),$$
we also have 
$$\frac{d}{dt} P_t^{\sB, \sN, \lambda} u_1 \otimes u_2= L^{\sB, \sN, \lambda} P_t^{\sB, \sN, \lambda} u_1 \otimes u_2= L^*(P_t^{\sB, \sN, \lambda} u_1 \otimes u_2)$$
where the last equality is the second statement in the  Proposition \ref{prop:wpop}.
\end{proof}
\begin{definition}
We define 
\begin{align*}\Xi' &= \sum_{i=0}^\infty P_t^{\sB, \sN, \lambda} C_c^\infty(\mathring B) \otimes E(\lambda_i) \\
&= \left\{ \sum^k_{i=0} u^i_1\otimes u^i_2: u_1^i\in P_t^{\sB, \sN, \lambda}C_c^\infty(\mathring B), u^i_2\in E(\lambda_i), k\in \N\right\}.
\end{align*}
The class $\Xi'$ is dense in $D(L^\sC)$ and stable w.r.t. $P_t^\sC$. 
\end{definition}
\subsection{Regularity of $N$-warped products with one-dimensional fiber}
Towards our main theorem we can make use of the fact that the result is already established for the case of a smooth $f$ and  a one-dimensional fiber space $F$. This follows since for smooth $f$ the warped product  is a smooth weighted Riemannian manifold away from points of degeneration of $f$. The following theorem is a direct corollary of Theorem 1.1 in \cite{ketterer}.
\begin{theorem} 
Assume $(B,g_\sB)$ is a Riemannian manifold that has Alexandrov curvature bounded from below, $f$ is smooth on $B$,  $\partial B\subset f^{-1}(\{0\})$, $N>1$, and
\smallskip
\begin{enumerate}
\item $\nabla^2 f+ K f g_\sB \leq 0$, 
\medskip
\item $|\nabla f|_\sB^2 + K f^2 \leq K_F$, 
\end{enumerate} 
Then the $N$-warped product satisfies $B\times_f^N \left( [0,\frac{\pi}{\sqrt{K_F}}], \sin^{N-1}(\sqrt{K_F}r) dr\right)$ satisfies the condition $\RCD((N+d-1)K,N+d)$.
\end{theorem}
\begin{corollary} \label{th:boundschroedinger}
Assume $B\subset \R$ is a closed interval, $f$ is smooth on $B$,  $\partial B\subset f^{-1}(\{0\})$, $N>1$ and $\lambda>0$.
Assume
\smallskip
\begin{enumerate}
\item $f''+ K f\leq 0$, 
\medskip
\item $|f'|^2 + K f^2 \leq K_F$. 
\end{enumerate} 
If $K_F>0$, we assume $\lambda\geq \frac{K_F N}{N-1}$. 
Consider the semi-group $\left( P_t^{\sB, \sN, \lambda}\right)_{t>0}$ associated to the operator $L^{\sB, \sN, \lambda}$. 
Then 
$$ P_{t}^{\sB, \sN, \lambda} u, |\nabla P_{t}^{\sB, \sN, \lambda} u|\in L^\infty(\m_{\sB}^\sN),  \ \forall t>0, \ \ u\in C^\infty_c(B).$$
\end{corollary}
\begin{proof}
The proof of this corollary is the same as the corresponding result for spherical $N$-suspension in Section 3.4 of \cite{ketterer2}.\end{proof}
From this we  can deduce the following important regularity property of elements in $\Xi'$ (this is Remark 3.21 in \cite{ketterer2}).
\begin{corollary}
If $u\in \Xi'$, then $u, |\nabla u|, L^\sC u \in L^\infty(\m^\sN)$. 
\end{corollary}
\section{Metric structure of $N$-warped products over $\RCD$ spaces}\label{sec:metricstructure}

Let $F$ be a compact $\RCD(K_\sF(N-1),N)$ space where $N>1$ and $K_F\in \R$. In particular $\m_\sF$ is finite.

Assume $B$ is a  $1$-dimensional Riemannian manifold, $f: B\rightarrow [0,\infty)$ is smooth,  $\partial B\subset f^{-1}(\{0\})$, and
\smallskip
\begin{enumerate}
	\item $f''+ K f\leq 0$, 
	\medskip
	\item $|f'|^2 + K f^2 \leq K_F$. 
\end{enumerate} 
\smallskip

The $N$-warped product $\nwp$ is a complete metric measure space. Hence, we can also consider its Cheeger energy $\Ch^\snwp$ and the associated space of Sobolev functions $D(\Ch^\snwp)=W^{1,2}(\nwp)$. In the following we investigate the relation between the energy $\mathcal E^*$ and $\Ch^\snwp$.
\begin{proposition}\label{prop:energyineq} It holds $D(\mathcal E^*)\subset W^{1,2}(\nwp)$ and for all $u\in C^\infty_c(\mathring B)\otimes \lip(F)$ we have 
\begin{align}\label{ineq:wugs}|\nabla u|^2_{\snwp} \leq |\nabla u|_\sB^2 + \frac{1}{f^2} |\nabla u|^2_{\sF} = |\nabla u|_{*}^2  \ \ \mbox{$\m^\sN$-a.e}.
\end{align}
and for all $u\in D(\mathcal E^*)$ we have \begin{align}\label{ineq:wugenergy}\mbox{$|\nabla u|^2_\snwp\leq \Gamma^*(u,u)$ $\m^\sN$-a.e.  }
\end{align}
where $\Gamma^*$ is the $\Gamma$-operator associated to $\mathcal E^*$ and $|\nabla (\cdot)|_{\snwp}$ is the minimal weak upper gradient of $\Ch^\snwp$.
\end{proposition}
\begin{proof}
{\bf (1)}
Let $u\in C^\infty_c(\mathring B)\otimes \lip(F)$, i.e. $u= \sum_{i=1}^N u^i_1\otimes u_2^i$ with $u_1^i\in C^\infty_c(\mathring B)$ and $u_2^i\in \lip(F)$.

Let $\gamma=(\alpha, \beta):[0,1]\rightarrow \mathring B\times F$ be a continuous curve in $AC^2(\mathring B\times F)$. It is straightforward to check that $\alpha \in AC^2(\mathring B)$ and $\beta \in AC^2(F)$. Hence 
\begin{align*} 
\left| u(\alpha(s), \beta(t))- u(\alpha(s), \beta(t'))\right|&\leq L \de_\sF(\beta(t), \beta(t'))\\
&\leq \int_t^{t'} g(\tau)\dint \tau\leq |v(t)- v(t')|\\
\left| u(\alpha(s), \beta(t))- u(\alpha(s'), \beta(t))\right|&\leq L  |\alpha(s)- \alpha(s')|\\
&\leq \int_{s}^{s'} g(\tau)\dint \tau\leq |v(s)- v(s')|
\end{align*}
where $v(t) =\int_0^t g(\tau) \dint \tau$ for an integrable function $g: [0, 1] \rightarrow [0,\infty)$ that depends on $\gamma= (\alpha, \beta)$. 

Then, we can use Lemma 4.3.4 in \cite{agsgradient} to obtain 

\begin{align*}
\left|\frac{d}{dt} (u\circ \gamma)(t)\right|\leq &\limsup_{h\rightarrow 0} \frac{|u(\alpha(t-h), \beta(t))- u(\alpha(t), \beta(t))|}{h}\\
&  \ \ \ \ \ \ \ + \limsup_{h\rightarrow 0} \frac{|u(\alpha(t), \beta(t+h))- u(\alpha(t), \beta(t))|}{h}\  
\end{align*}
$\mbox{for  $\mathcal L^1$-a.e. } t\in [0,1].$

Applying the definition of the local Lipschitz constant it follows that 
\begin{align}\label{ineq:loclip}
&\left|\frac{d}{dt} (u\circ \gamma)(t)\right|\leq \lip u^{\beta(t)}(\alpha(t)) |\dot{\alpha}(t)| + \lip u^{\alpha(t)}(\beta(t))|\dot \beta(t)| 
\end{align}
 for  $\mathcal L^1$-a.e. $ t\in [0,1].$
 
 We note  $u^{r}=u(r,\cdot)$  is locally Lipschitz in  $F$ for every $r\in B$. Hence, since $F$ is $\RCD$, and therefore satisfies a doubling property and supports a weak local Poincar\'e inequality, by Theorem \ref{th:cheegerlipschitz} it follows that  $\lip u^{r}= |\nabla u^{r}|_\sF$.  Moreover, $u^{
 x}$ is smooth on $B$ for every $x\in F$, and therefore $\lip u^{x}= |(u^{
 x})'|$.

We set $$G_u(r,x):= \sqrt{ \left((u^{x})'(r)\right)^2 + {\textstyle \frac{1}{(f(r))^2}}|\nabla u^{r}|^2_\sF(x)}.$$

Hence, in combination with the Cauchy-Schwarz inequality it follows from \eqref{ineq:loclip} that
$$
	\left|\frac{d}{dt} (u\circ \gamma)(t)\right|\leq G_u(
	\alpha(t), \beta(t))  \sqrt{  |\dot{\alpha}(t)| + (f\circ \alpha(t))^2 |\dot \beta(t)|^2 } 
$$

We integrate this inequality w.r.t. $t$ and obtain
\begin{align}\label{ineq:wug}|u(
\gamma(1))- u(
\gamma(0))|\leq \int_0^1 G_u(\gamma(t)) |\gamma'(t)| \de t.
\end{align}
Since $u$ is compactly supported in $\mathring B\times F$, it already follows from the inequality \eqref{ineq:wug}, that holds for every $\gamma=(\alpha, \beta):[0,1]\rightarrow \mathring B\times F$ in $AC^2(\mathring B\times F)$, that $G$ is a weak upper gradient of $u$. Hence 
$$|\nabla u|^2 \leq G_u \ \m^\sN\mbox{-a.e.}$$
\smallskip
{\bf (2)}
We have that $C_c^\infty(\mathring B)\otimes \lip(F)$ is dense in $D(\mathcal E^*)$ by definition of the closed form $\mathcal E^*$. Hence, if $u\in D(\mathcal E^*)$ we can pick a sequence $u_n\in C_c^\infty(\mathring B)\otimes \lip(F)\rightarrow u$ w.r.t. $\left\|\cdot \right\|_{\mathcal E^*}^2$.  In particular, we have  $u_n\rightarrow u$ in $L^2(\m^\sN)$. 
$$\Ch^\snwp(u_n)\leq \mathcal E^*(u_n)\rightarrow \mathcal E^*(u)<\infty.$$
Hence $u\in W^{1,2}(\nwp)$, and $\Ch^\snwp(u)\leq \mathcal E^*(u)$. 
After extracting a subsequence $|\nabla u_n|$ converges pointwise $\m^\sN$-a.e. to a weak upper gradient of $u$ by the stability property of minimal weak upper gradients. 

On the other hand 
$$\sqrt{ (\lip u_n^{x})2 + \frac{1}{f^2} |\nabla u_n|^2_F} \rightarrow \Gamma^*(u,u)$$
in duality with $L^\infty(\m^\sN)$ by property (4) in Remark \ref{properties:estar}.
Since the left hand side is a weak upper gradient of $u_n$, after taking another subsequence, it converges  to a weak upper gradient of $u$ by Lemma \ref{lem:stabilitywug}. Hence, the desired inequality \eqref{ineq:wugenergy} follows. \end{proof}
We recall the following result from \cite{cks}.
\begin{theorem}\label{th:mcpfornwp}
Let $F, I$ and $f$ be as before. Then $\nwp$ satisfies the measure contraction property $\MCP(KN, K+1)$. 
\end{theorem}
Recall that the intrinsic distance $\de_{\mathcal E^*}$ of the strongly local, regular Dirichlet form $\mathcal E^*$ is defined through
\begin{align*}&\de_{\mathcal E^*}((s,x),(t,y))\\
	&\ \ \ \ \ = \sup\{ u(s,x)-u(t,y): u\in D_{loc}(\mathcal E^*)\cap C(B\times F/\sim), \Gamma^*(u,u)\leq 1\}.
	\end{align*}
\begin{proposition}\label{prop:distancewp} Let $F, B$ and $f$ as before. The intrinsic distance of $\mathcal E^*$ coincides with the distance on $\mwp$. 
\end{proposition}
\begin{proof} {\bf (1)}
We know by Proposition \ref{prop:energyineq} that $D(\mathcal E^*)\subset W^{1,2}(\nwp)$, and for any $u\in D(\mathcal E^*)$ we have 
$$\left| \nabla u\right|^2:= \left| \nabla u\right|_\snwp^2\leq\Gamma^*(u,u).$$
Then, $\Gamma^*(u,u)\leq 1$ implies that $|\nabla u|\leq 1$.\smallskip\\
{\it Claim:} Since $\nwp$ satisfies the measure contraction property $\MCP(KN, N+1)$, it also satisfies the Sobolev-to-Lipschitz property.
\smallskip

 Indeed, for points $p, q\in \nwp$ we set $\mu_0= \m(B_\epsilon(q))^{-1}\m^\sN|_{B_\epsilon(q)}$ and $\mu_1=\delta_p$. Let $\Pi$ be the unique optimal dynamical plan  between $\mu_0$ and $\mu_1$.  Hence the restriction $(e_{[0, t_0]})_\sharp \Pi$ is a $2$-test plan. Since $|\nabla u|$ is in particular a weak upper gradient it follows
\begin{align*}&\int |u(e_1(\gamma))- u(e_0(\gamma))| \dint \Pi(\gamma)
	\\
	&\ \ \ \ \ \leq \int \int_0^1 |\nabla u|(e_t(\gamma))L(\gamma) \de t \de \Pi(\gamma)\leq W(\mu_0, \mu_1)\end{align*}
where we used $|\nabla u|\leq 1$ and  the Cauchy-Schwarz inequality in the last inequality. 

If we send $\epsilon\rightarrow 0$, we obtain $|u(p)- u(q)|\leq \de_{\smwp}(p,q).$
This yields 
$$\de_{\mathcal E^*}(p,q)= \sup\{ u( p) - u( q)\} \leq \de_{\smwp}(p,q).$$
{\bf (2)}
On the other hand,  we pick $p=(s,x)$ and $q=(t,y)$ in $\mwp$, and let $\gamma=(\alpha, \beta): [0,1]\rightarrow \mwp$ be the geodesic between $p$ and $q$.  By Theorem \ref{th:albi0} $\gamma$ is determined by $s, t$ and $L^\sF(\beta)=L$. Hence 
$$L(\gamma)= L(\tilde \gamma)$$
where $\tilde \gamma$ is the geodesic between the points $(s,0)$ and $(t, L)$ in $B\times_f \R$. 
Hence \begin{align}\label{formula:gencos}\de_{\smwp}(p,q)=\de_{B\times_f \R}((s,0), (t, L))=: h(t, L)= h(t, \de_\sF(x,y)) =:g(t,y).
\end{align}
Since $h$ is just the distance function to $(s,0)$ in $B\times_f \R$ we have  $|\nabla^{(t,L)} h|\leq 1$ in $B\times_f \R$. 

The chain rule for the $\Gamma$-operator $\Gamma^*$ applied to the function $g$, that is a composition of $h$ and $\de_\sF(x, \cdot)$, yields 
\begin{align*}\Gamma^*(g,g)(t,y)&= \left( \frac{\partial}{\partial t} h(t, L)\right)^2+ \frac{1}{f^2(t)} \left(\frac{\partial}{\partial L} h(t, L) \right)^2|\nabla \de_\sF(x, \cdot)|^2(y)\\
&\leq\left( \frac{\partial}{\partial t} h(t, L)\right)^2+ \frac{1}{f^2(t)} \left(\frac{\partial}{\partial L} h(t, L) \right)^2\\
&=|\nabla^{(t,L)} h|(r,L)^2\leq 1\end{align*}
Hence 
$$
\de_{\smwp}(p,q)= g(q)- g(q)\leq \de_{\mathcal E^*}(p,q).$$
\end{proof}
\begin{assumption}\label{assumption:doubling}
In addition to our  assumptions above  we  assume that \begin{center}
$(\%)$ \ $B\times_f^N F$ satisfies a  {\bf global} doubling property.\end{center}\end{assumption}
\begin{assumption}\label{assumption:doubling2}
The Assumption \ref{assumption:doubling} holds in each of the following cases: 
\begin{enumerate}
\item 
 If $B$ and $F$ are  compact, $B\times_f F$ is compact. Then the property $\MCP(KN,N+1)$ implies $(\%)$.  
\item If $K\geq 0$, then $B\times_f F$ satisfies $\MCP(0,N+1)$ that implies  $(\%)$. 
\item If   $f$  is a bounded function and $F$ is compact, then $(\%)$ holds. 
Indeed, this is true  since $B\times_f^\sN F$  satisfies $\MCP(KN, N+1)$, and since boundedness of $f$ and compactness of $F$ imply that the volume of balls of radius $r$ grows at most like $\sim r$. 
\end{enumerate}
\end{assumption}
\begin{corollary} 
	Let $B, f$ and $F$ as before and  $(\%)$ holds. Then $\mathcal E^*= \Ch^\snwp$. 
\end{corollary}
\begin{proof}
	We have $\de_{\mathcal E^*}= \de_{\smwp}$ by the previous proposition. Hence, $\de_{\mathcal E^*}$ induces the topology of $B\times F/\sim$. By Theorem \ref{th:mcpfornwp} $B\times_f^N F$ satisfies the property $MCP(KN, N+1)$. Because of the property $(\%)$ $\nwp$ satisfies a doubling property. Hence $B\times F/\sim$ with $\de_{\mathcal E^*}$ and $\m^\sN$ satisfies a doubling property.  Then, it follows by Theorem \ref{th:kz} that any for any Lipschitz function $u$ w.r.t. $\de_{\mathcal E^*}= \de_{\smwp}$ we have that $u\in D_{loc}(\mathcal E^*)$ and $\Gamma^*(u,u)\leq \lip(u)^2$.  Since $\nwp$ also satisfies a local Poincar\'e inequality, by a theorem of Cheeger we have $\lip(u)=|\nabla u|$. On the other hand, the Proposition \ref{prop:energyineq} says that 
	$|\nabla u|^2\leq \Gamma^*(u,u)$.  Hence, by integration w.r.t. $\m^\sN$ we have that $\mathcal E^*(u)= \Ch^{\snwp}(u)$ for any Lipschitzfunction. We infer that $\mathcal E^*= \Ch^\snwp$.
\end{proof}
\section{The $\RCD$ condition for $N$-warped products}

\subsection{The carr\'e-du-champ operator on $N$-warped products}\label{subsec:gamma2}

Let $F$ be a compact length space with a  finite measure such that $L^\sF$ has a discrete spectrum. Let $N>1$.


We assume $B$ is a  $1$-dimensional Riemannian manifold, $f: B\rightarrow [0,\infty)$ is smooth, and  $\partial B= f^{-1}(\{0\})$.

\begin{proposition}\label{prop:oneill}
We consider $u \in C^\infty_{c}(\mathring B)\otimes D_{\sW^{1,2}}(L^\sF)\subset D_{W^{1,2}}(L)$ and $\phi \in P_t^{\sB, \sN, \lambda} C_c^\infty(\mathring B)\otimes E(\lambda)$.
Then we have 
\begin{align}\label{id:formula}
\Gamma_2(u; \phi)
=& \int \Gamma_2^{\sB, \sN}(u; \phi) \dint \m_\sF + \int  \frac{1}{f^4} \Gamma_2^{\sF}(u;\phi) \dint \m_\sB^{\sN}\nonumber\\
&+\int\left[ 2\langle {\textstyle \frac{f'}{f}}, u'\rangle_\sB\frac{ \eL^\sF u}{f^2}  - \frac{f^\#}{f^2} | \nabla u|^2_{\sF}+2\left| \nabla \big({\textstyle \frac{u}{f}}\big)'\right|^2_{\sF}\right]
\phi \de\m^\sN
\end{align}
where $f^\#:=\frac{ \Delta^\sB f }{f}+ {(N-1)} \frac{|f'|_\sB^2}{f^2}$. 
\begin{remark} 
\hspace{0mm}
\begin{enumerate}\item $u'(r,x)=\frac{d}{dr} \big|_r u(\cdot,x)$ for $\m_\sF$-a.e. $x\in F$ and $\forall r\in B$, and $|\nabla u|_\sF(r,x)=|\nabla u(r, \cdot)|_\sF(x)$  $\forall r\in B$ and $\m_\sF$-a.e. $x\in F$. 
\item  $\Gamma_2^{\sB, \sN}(u;\phi)$ is a short-hand notation for 
$$ \int \frac{1}{2} | (u^x)'|^2 L^{\sB, \sN}\phi^x \de\m_\sB^\sN - \int \langle (u^x)', \left( L^{\sB, \sN} u^x\right)' \rangle  \phi^x \dint \m_\sB^\sN, $$ for $\m_\sF$-a.e. $x\in F$. 
In particular, the first integral is well-defined since $\phi^x\in C^\infty(\mathring B)$ and $u^x\in C_c^\infty(\mathring B)$. 
\item  $\Gamma_2^\sF(u;\phi)$ is a short-hand notation for 
$$\int  \frac{1}{2} |\nabla u^r|_\sF L^\sF \phi^r\de\m_\sF -\int \langle \nabla u^r, \nabla L^\sF u^r\rangle_\sF \phi^r  \dint \m_\sF$$
for $\m_\sB^\sN$-a.e. $r\in \mathring B$. 
\end{enumerate}
\end{remark}
\end{proposition}
We move the lengthy proof of Proposition \ref{prop:oneill} to  Appendix \ref{appendix}.
\begin{corollary}
We consider $u \in C^\infty_{c}(\mathring B)\otimes D_{\sW^{1,2}}(L^\sF)\subset D(L)$ and $\phi \in \Xi'$.
Then we have 
\begin{align}\label{id:formulav2}
\Gamma_2(u; \phi)
=& \int \left[\Gamma_2^{\sB, f^\sN}(u; \phi) +   \frac{1}{f^4} \Gamma_2^{\sF}(u;\phi)\right] d\m^{\sN}\nonumber\\
&+\int\left[ 2\langle {\textstyle \frac{f'}{f}}, u'\rangle_\sB\frac{ \eL^\sF u}{f^2}  - \frac{f^\#}{f^2} | \nabla u|^2_{\sF}+2\left| \nabla \big({\textstyle \frac{u}{f}}\big)'\right|^2_{\sF}\right].
\phi \de\m^\sN
\end{align}
\end{corollary}
\begin{proof}
This follows  from the previous proposition and from linear dependency of the formula in $\phi$.
\end{proof}
\begin{proposition}
Formula \eqref{id:formulav2} holds for $u+v=u_1\otimes u_2+ v_1\otimes v_2$ with  $v_1\otimes v_2 \in C_c^\infty(\mathring B) \otimes D_{W^{1,2}}(L^\sF)$, $$u_1\in \bigcup_{t>0} P_t^{\sB, \sN}C_{c}^\infty(\mathring B),$$ $u_2\equiv c\in \R$ and $\phi \in \Xi'$. 
\end{proposition}
\begin{proof} W.l.o.g. we can assume that $\phi=\phi_1\otimes \phi_2 \in P_t^{\sB, \sN, \lambda} C_c^\infty(\mathring B)\otimes E(\lambda)$. By linearity of $\Gamma_2(\cdot, \cdot; \phi)$ in $\phi$ this extends to arbitrary $\phi \in\Xi'$.
\smallskip

 If $u_1= P_t^{\sB, \sN} \tilde u_1\in D_{W^{1,2}}(L^{\sB, \sN})$ for some $\tilde u_1\in C_c^\infty(\mathring B)$ and $t>0$, then it holds for $u_2\equiv const=c\in \R$  that 
$$u_1 \otimes u_2 = P_t^{\sB, \sN}\tilde u_1 \otimes c= P_t (\tilde u_1\otimes c)\in D_{W^{1,2}}(L).$$
Hence
\begin{align*}
\Gamma_2(u,v;\phi)=&\underbrace{ \int \frac{1}{2} \langle \nabla u, \nabla v\rangle \eL \phi d\m^\sN }_{=: (I)} - \underbrace{\int \langle \nabla u, \nabla \eL v\rangle \phi d\m^\sN}_{=:(II)}.
\end{align*}
is well-defined.

Since $\langle \nabla u, \nabla v\rangle = \langle u_1', v_1'\rangle_\sB u_2 v_2= \langle u_1', v_1'\rangle_\sB c v_1$, it follows that 
\begin{align*}2(I)&=c\int \langle u_1', v_1'\rangle_\sB v_2 L\phi d\m^\sN\\&= c\int \langle u_1', v_1'\rangle_\sB v_2 \left[ L^{\sB, \sN, \lambda}\phi_1  \phi_2\right] \dint \m^\sN.\\
&=  c \int   \langle u_1', v_1'\rangle_\sB L^{\sB, \sN, \lambda} \phi_1 \dint \m_{\sB}^\sN \int v_2 \phi_2  \dint \m_\sF.
\end{align*}
Since $v_1\in C^\infty_c(\mathring B)$, we have $\langle u_1', v_1'\rangle_\sB\in C_c^\infty(\mathring B)$. Hence

\begin{align*}2(I)=&  c \int   L^{\sB, \sN, \lambda} \langle u_1', v_1'\rangle_\sB \phi_1 \dint \m_{\sB}^\sN \int v_2 \phi_2  \dint \m_\sF\\
=&   c \int   L^{\sB, \sN} \langle u_1', v_1'\rangle_\sB \phi_1 \dint \m_{\sB}^\sN \int v_2 \phi_2  \dint \m_\sF\\
&+c \int   \langle u_1', v_1'\rangle_\sB \phi_1 \dint \m_{\sB}^\sN \underbrace{\int v_2(- \lambda)\phi_2  \dint \m_\sF}_{\int L^\sF v_2 \phi_2 \dint \m_\sF}
\end{align*}
Moreover, since $v_1\in C_c^\infty(\mathring B)$, we have $$\langle \nabla u, \nabla ( L^{\sB, \sN} v_1 v_2 + \frac{v_1}{f^2} L^\sF v_2)\rangle= \langle u_1', (L^{\sB, \sN} v_1)'\rangle_\sB c v_2+ \langle u_1',\left(\frac{v_1}{f^2}\right)'\rangle L^\sF v_2.$$ Hence
\begin{align*}
(II)&= c^2\int \langle u_1', \left( L^{\sB, \sN} v_1 \right)'\rangle \phi_1 \dint \m_{\sB}^\sN \int v_2 \phi_2 \dint \m_\sF.
\end{align*}
We obtain
\begin{align*}&\Gamma_2(u,v; \phi) \\
&= \int \Gamma_2^{\sB,\sN}(u,v;\phi) \dint \m_\sF + \int\frac{1}{f^2} \langle u_1', v_1'\rangle_\sB L^\sF v_2 \phi_2 + \langle u_1',\left(\frac{v_1}{f^2}\right)'\rangle L^\sF v_2\dint \m_\sF.
\end{align*}
This formula corresponds to  \eqref{id:next} in the proof of Proposition \ref{prop:oneill}.
\smallskip

Similarly, one computes that
$\Gamma_2(v,u;\phi) = \int \Gamma_2^{\sB,\sN}(u,v;\phi) \dint \m_\sF.$
\smallskip\\
Finally, we compute $\Gamma_2(u,u; \phi)$.  
Since $\langle \nabla u, \nabla u\rangle =  \langle u_1', u_1'\rangle_\sB c^2$, we have 
\begin{align*}2(I)&=c^2\int \langle u_1', u_1'\rangle_\sB  L\phi d\m^\sN\\&= c^2\int \langle u_1', u_1'\rangle_\sB \left[ L^{\sB, \sN, \lambda}\phi_1  \phi_2\right] \dint \m^\sN.\\
&=  c^2 \int   \langle u_1', u_1'\rangle_\sB L^{\sB, \sN, \lambda} \phi_1 \dint \m_{\sB}^\sN \int  \phi_2  \dint \m_\sF.
\end{align*}
We have $\int \phi_2 \dint \m_\sF=0$ if and only if $\lambda>0$, since in this case $\phi_2$ is a nonconstant eigenfunction. Similarly for $(II)$. Hence, in any case we have
$$\Gamma_2(u,u;\phi)
=c^2\Gamma_2^{\sB, \sN}(u_1,u_1;\phi_1)\int  \phi_2  \dint \m_\sF= \int \Gamma_2^{\sB, \sN}(u,u;\phi) \dint \m_\sF. $$

Together with the formula for $\Gamma_2(v;\phi)$ that we computed before, this yields the desired formula for $\Gamma_2(u+v; \phi)$. 
\end{proof}
\begin{corollary}
For $u \in  \bigcup_{t>0} P_t^{\sB, \sN}C_{c}^\infty(\mathring B) +C^\infty_{c}(\mathring B)\otimes D_{\sW^{1,2}}(L^\sF)$, and $\phi \in \Xi'$
 we have 
\begin{align}\label{equ:oneill2}
\Gamma_2(u; \phi)
=& \int \left[\Gamma_2^{\sB, f^\sN}(u; \phi) +   \frac{1}{f^4} \Gamma_2^{\sF}(u;\phi)\right] d\m^{\sN}\nonumber\\
&+\int\left[ 2\langle {\textstyle \frac{f'}{f}}, u'\rangle_\sB\frac{ \eL^\sF u}{f^2}  - \frac{f^\#}{f^2} | \nabla u|^2_{\sF}+2\left| \nabla \big({\textstyle \frac{u}{f}}\big)'\right|^2_{\sF}\right]
\phi \de\m^\sN.
\end{align}
\end{corollary}

A constant function $\phi\equiv c$ is not in  $D(L^{\sB,\sN})$ if $B$ is noncompact. However, it will be useful to consider constant  functions as test functions in \eqref{equ:oneill2} also in the noncompact case. For this purpose we extended the domain of the Bochner formula in Corollary \ref{cor:bochnerext}. Similarly, we will extend the domain of formula \eqref{id:formulav2}.
\smallskip

If $B$ is noncompact,  for $u \in C^\infty_{c}(\mathring B)\otimes D_{\sW^{1,2}}(L^\sF) +  \bigcup_{t>0} P_t^{\sB, \sN}C_{c}^\infty(\mathring B)$  we define 
$$\Gamma_2(u;1):= - \int \langle \nabla u, \nabla Lu\rangle \dint \m^\sN$$
as well 
$$\int \Gamma_2^{\sB, \sN}(u; 1) \dint \m_\sF:= - \int \langle u' , (L^{\sB, \sN} u)' \rangle_\sB \dint \m_\sB^\sN \dint \m_\sF$$
and 
$$\int \Gamma_2^{\sF}(u;1) \dint \m_\sB^\sN:= - \int \langle \nabla u, \nabla L^\sF u\rangle_\sF \dint \m_\sF \dint \m_\sB^\sN.$$
This is of course consistent with the case when $B$ is compact. 
\begin{corollary}
We consider $u \in \bigcup_{t>0} P_t^{\sB, \sN}C_{c}^\infty(\mathring B)+ C^\infty_{c}(\mathring B)\otimes D_{\sW^{1,2}}(L^\sF)  $.
Then
\begin{align*}
\Gamma_2(u; 1)
=& \int \left[\Gamma_2^{\sB, \sN}(u; 1) +   \frac{1}{f^4} \Gamma_2^{\sF}(u;1)\right] d\m^{\sN}\nonumber\\
&+\int\left[ 2\langle {\textstyle \frac{f'}{f}}, u'\rangle_\sB\frac{ \eL^\sF u}{f^2}  - \frac{f^\#}{f^2} | \nabla u|^2_{\sF}+2\left| \nabla \big({\textstyle \frac{u}{f}}\big)'\right|^2_{\sF}\right] \de\m^\sN.
\end{align*}
\end{corollary}
\begin{proof}
We pick $\phi_n= \phi_{1,n}\otimes 1 \in P_t^{\sB, \sN} C_0^\infty(\mathring B)\otimes E(0)$ where $\phi_{1,n}= P_t^{\sB, \sN}\psi_n$ for a sequence  $(\psi_n)_{n\in \N}\subset  C^\infty_0(\mathring B)$   such that $\psi_n\uparrow 1$ pointwise $\m_\sB^\sN$-a.e. in $B$.  Then, $P_t^{\sB, \sN} \psi_n \uparrow 1$ for every $t>0$. 

Let $u$ be as in the assumptions. With $P_t(\psi_n \otimes 1)= (P_t^{\sB, \sN}\psi_n )\otimes 1= P_t^{\sB, \sN}\psi$ it follows that 
\begin{align*}\int |\nabla u|^2 L \left( P_t^{\sB, \sN} \psi_n \otimes 1 \right)\dint \m^\sN&= \int |\nabla u|^2 L P_t( \psi_n\otimes 1)\dint \m^\sN
\\
= \int L P_{t/2}|\nabla u|^2 P_{t/2}^{\sB, \sN} \psi_1 &\rightarrow \int LP_{t/2} |\nabla u|^2 \dint \m^\sN=0
\end{align*}
as well as 
\begin{align*} \int \langle \nabla u, \nabla Lu\rangle P^{\sB, \sN}_t\psi_n \dint\m^\sN &\rightarrow \int \langle \nabla u, \nabla Lu\rangle \dint\m^\sN.
\end{align*}
Similarly, one checks that 
\begin{align*}\int  |u'|_\sB^2L^{\sB, \sN} \phi_{1,n} \dint \m_\sB^\sN \rightarrow&\  0 \\ 
\int \langle u', (L^{\sB, \sN} u)'\rangle_\sB \phi_{1,n}\dint \m_\sB^\sN\rightarrow &\int  \langle u', (L^{\sB, \sN} u)'\rangle_\sB\dint \m_\sB^\sN\end{align*}
where these limits hold $\m_\sF$-almost everywhere in $F$. By Lebesgue's dominant convergence theorem it then follows
\begin{align*}
\int \int L^{\sB, \sN} \phi_{1,n} |u'|_\sB^2 \dint \m_\sB^\sN \dint \m_\sF \rightarrow&\  0 \\ 
\int \int \langle u', (L^{\sB, \sN} u)'\rangle_\sB \phi_{1,n}\dint \m_\sB^\sN\dint \m_\sF\rightarrow &\int  \Gamma_2^{\sB, \sN}(u; 1)\dint \m_\sF.
\end{align*}
On the other hand, we have that 
\begin{align*}\Gamma_2^{\sF}(u;\phi)= &\int |\nabla u|_\sF^2 L^\sF 1 \dint\m_\sF \phi_{1,n} - \int \langle \nabla u, \nabla (L^\sF u)\rangle_\sF \dint \m_\sF \phi_{1,n}\\
=&-\int \langle \nabla u, \nabla (L^\sF u)\rangle_\sF \dint \m_\sF \phi_{1,n}\rightarrow \Gamma_2^{\sF}(u;1)\end{align*}
Moreover $\left[ 2\langle {\textstyle \frac{f'}{f}}, u'\rangle_\sB\frac{ \eL^\sF u}{f^2}  - \frac{f^\#}{f^2} | \nabla u|^2_{\sF}+2\left| \nabla \big({\textstyle \frac{u}{f}}\big)'\right|^2_{\sF}\right]\phi_n = f_n$ is uniformily integrable w.r.t. $\m^\sN$. So 
$$\int f_n \dint\m^\sN\rightarrow\int \left[ 2\langle {\textstyle \frac{f'}{f}}, u'\rangle_\sB\frac{ \eL^\sF u}{f^2}  - \frac{f^\#}{f^2} | \nabla u|^2_{\sF}+2\left| \nabla \big({\textstyle \frac{u}{f}}\big)'\right|^2_{\sF}\right]\dint \m^\sN.$$

Hence, the desired formula follows from the formula in the previous corollary with $\phi=\phi_n$, and then letting $n$ go to $ \infty$. 
\end{proof}

\begin{corollary}\label{cor:formula}
We consider $u \in  \bigcup_{t>0} P_t^{\sB, \sN}C_{c}^\infty(\mathring B) +C^\infty_{c}(\mathring B)\otimes D_{\sW^{1,2}}(L^\sF)$, and $\phi \in \Xi'$. We set $\psi= \phi +\lambda$ for $\lambda \in\R$. 
Then we have 
\begin{align}\label{formula:important}
\Gamma_2(u; \psi)
=& \int \left[\Gamma_2^{\sB, f^\sN}(u; \psi) +   \frac{1}{f^4} \Gamma_2^{\sF}(u;\psi)\right] d\m^{\sN}\nonumber\\
&+\int\left[ 2\langle {\textstyle \frac{f'}{f}}, u'\rangle_\sB\frac{ \eL^\sF u}{f^2}  - \frac{f^\#}{f^2} | \nabla u|^2_{\sF}+2\left| \nabla \big({\textstyle \frac{u}{f}}\big)'\right|^2_{\sF}\right]
\psi \de\m^\sN
\end{align}
where $\Gamma_2(u; \psi):= \Gamma_2(u; \phi) + \lambda \Gamma_2(u; 1)$ and similarly for $\Gamma_2^{\sB, \sN}$ and $\Gamma_2^\sF$. 
\end{corollary}
\subsection{Spectral decomposition of $N$-warped products}\label{subsec:spectral}

In this subsection we assume that $F$ is a compact $\RCD(K_\sF(N-1),N)$ space where $N>1$ and $K_F\in \R$. In particular $\m_\sF$ is finite.
In particular, the operator $L^\sF$ has a discrete spectrum $\{\lambda_i\}_{i\in \N_0}$.

Let $E(\lambda_i)$ be the eigenspace for the eigenvalue $\lambda_i$. In particular, we have  the spectral decomposition
\begin{align*}
&\bigoplus_{i=0}^\infty \left(E(\lambda_i) , \left\|\cdot \right\|_{D(L^\sF)}\right) =\left\{ \sum_{i=0}^\infty v_i: \sum_{i=0}^\infty \left\|v_i\right\|_{D(L^\sF)}<\infty\right\}= D(L^\sF).
\end{align*} 
We also define 
$$
\sum_{i=0}^\infty E(\lambda_i)= \left\{ \sum_{i=1}^k v_i : v_i\in E(\lambda_i) , k\in \N\right\}.
$$
\begin{proposition} Let $F$ be a compact metric measure space, and let $B$ and $f: B\rightarrow [0,\infty)$ be as before.   We assume that 
\begin{enumerate}
\item $f''+ Kf\leq 0$,  
\item $F$ satisfies the condition $\RCD(K_F(N-1), N)$ where 
$$K_F> \sup_B\{ (f')^2+ Kf^2\}.$$
\end{enumerate}
Then, for $u\in D_{W^{1,2}}(L)$ and $\psi= \phi+\lambda$ where $\phi \in \Xi'$ and $\lambda\in \R$ with $\psi\geq 0$, we have \begin{align}\label{ineq:bochner2}\Gamma_2(u;\psi)\geq {KN} \int |\nabla u|^2 \psi d\m^\sN + \frac{1}{N+1} \int \left( L u\right)^2 \psi d\m^\sN.
\end{align}
\end{proposition}
\begin{proof}  {\bf (1)} Assumption (1) yields that the mm space $(B, \m_\sB^\sN)$ satisfies the condition $\RCD(KN, N+1)$.

{\it Claim:}
\begin{align*}\Gamma_2(u;\psi)\geq {KN} \int |\nabla u|^2 \psi \dint \m^\sN + \frac{1}{N+1} \int \left( L u\right)^2 \psi \dint \m^\sN
\end{align*}
where $\psi$ is as in the assumptions and 
$u \in  \bigcup_{t>0} P_t^{\sB, \sN}C_{c}^\infty(\mathring B) +C^\infty_{c}(\mathring B)\otimes D_{\sW^{1,2}}(L^\sF)$.

{\it Proof of the Claim:}
The key steps  are the same as in the proof of Theorem 3.9 in \cite{ketterer2}. We indicate the main points of the proof. 

From formula \eqref{formula:important} in Corollary \ref{cor:formula} we get
\begin{align*}
\Gamma_2(u; \psi)
\geq& \int \left[\Gamma_2^{\sB, f^\sN}(u; \psi) +   \frac{1}{f^4} \Gamma_2^{\sF}(u;\psi)\right] \de \m^{\sN}\nonumber\\
&+\int\left[ 2\langle {\textstyle \frac{f'}{f}}, u'\rangle_\sB\frac{ \eL^\sF u}{f^2}  - \frac{f^\#}{f^2} | \nabla u|^2_{\sF}\right]
\psi \de\m^\sN
\end{align*}
 In combination with Corollary \ref{cor:bochnerext},  the $\RCD(K_F(N-1),N)$ condition for $F$, the properties of $f$ and since 
 $$f^\#=\frac{ \Delta^\sB f }{f}+ {(N-1)} \frac{| f'|_\sB}{f^2}$$
 it follows  that 
\begin{align*}
&\Gamma_2(u; \psi)\geq\\
& \int KN |u'|^2_\sB +  (\Delta^\sB u)^2+ \frac{1}{f^2} K N |\nabla u|_\sF^2 +\frac{1}{N} \left( N\langle {\textstyle \frac{f'}{f}}, u'\rangle_\sB+\frac{ \eL^\sF u}{f^2} \right)^2\de \m^\sN.
\end{align*}
Finally we use  $a^2 + \frac{1}{N} b^2= \frac{1}{N+1} (a+b)^2 + \frac{1}{(N+1)N} ( b- N a)^2$ to deduce the estimate in the claim. \hfill \qed
\smallskip
\\
{\bf (2)}
If $f^{-1}(\{0\})\neq \emptyset$, we have $K_F>\sup_B\{ (f')^2 + K f^2\}$. We can  rescale $f$ and $F$ such that $K_F>\sup_B\{(f')^2+ Kf^2\}=1$.

In particular, $F$ is  still an $\RCD(K_F(N-1), N)$ space with $K_F>0$. Then any eigenvalue $\lambda$ of $L^\sF$ satisfies \begin{center}$\lambda\geq K_F N> N\geq 1$\end{center} by the Lichnerowicz spectral estimate. 

{
Hence,  by Proposition \ref{prop:essential} we have  that
$$L^\sC u= L^{\sB,\sN, \lambda} u_1\otimes u_2, \ \ u_1\otimes u_2 =u \in C_c^\infty(\mathring B)\otimes E(\lambda)$$
is essentially self-adjoint for any positive eigenvalue $\lambda$ of $L^\sF$ --  independently of whether $f^{-1}(\{0\})$ is empty or not. }

Moreover 
$$L^\sC u= L^{\sB, \sN} u_1 \otimes u_2, \ \ u_1\otimes u_2 = u\in P_t^{\sB, \sN} C_c(\mathring B)\otimes E(0)$$
is essentially self-adjoint. 

Hence, the operator 
$$ u\in \bigcup_{t>0} P_t^{\sB, \sN}C_c(\mathring B)\otimes E(0) + \sum_{i=1}^\infty C_c^\infty(\mathring B)\otimes E(\lambda_i)=:\Xi\mapsto L^\sC u $$
is essentially self-adjoint.  

Hence, there is a unique self-adjoint extension  that must necessarily be the Laplace operator of the Cheeger energy associated to $B\times_f^\sN F$. In particular, $\Xi$ is dense in $D(L^\sC)$ w.r.t.  the operator norm. 
\smallskip\\
{\bf (3)} If we pick $u\in D_{W^{1,2}}(L)$ there exists a sequence $(u^n)_{n\in \N}$ in $\Xi$ such that $u_n\rightarrow u$ in $D_{L^2}(L)$.  Then 
\begin{align*}
\int | \nabla u_n|^2 L\phi d\m^\sN \rightarrow &\int | \nabla u|^2 L\phi d\m^\sN\\
\int |\nabla u|^2 \psi d\m^\sN \rightarrow& \int |\nabla u|^2 \psi \dint \m^\sN\\
\int L u_n \psi \dint \m^\sN \rightarrow& \int Lu \psi \dint \m^\sN.
\end{align*}
Here $\phi\in \Xi'$ and hence $\phi+ \lambda= \psi, L\phi \in L^\infty(\m^\sN)$.

We still have to show convergence of $\int \langle \nabla u_n, \nabla Lu_n\rangle \psi \dint \m^\sN$.  Since $u_n, Lu_n, \phi \in W^{1,2}(B\times_f^\sN F)$, we can apply the Leibniz rule. Hence 
\begin{align*}
\int \langle \nabla u_n, \nabla Lu_n\rangle \psi \dint\m^\sN&= \int \left[\langle \nabla u_n, \nabla (\psi Lu) \rangle -\langle \nabla u_n, \nabla \phi \rangle L u_n \right] \dint\m^\sN=(*).
\end{align*}
We have $\psi, |\nabla \psi|\in L^\infty(\m^\sN)$ since $\phi\in \Xi'$. Therefore
\begin{align*}
(*)&=- \int \psi (Lu_n)^2 \dint \m^\sN - \int \langle \nabla u_n, \nabla \phi\rangle Lu_n \dint \m^\sN.
\end{align*}
Now $\int \psi (Lu_n)^2 \dint \m^\sN\rightarrow \int\psi (Lu)^2 d\m^\sN$, and since $|\nabla\phi|\in L^\infty(\m^\sN)$, also
$$\int \langle \nabla u_n, \nabla \phi\rangle Lu_n \dint \m^\sN\rightarrow \int \langle \nabla u, \nabla \phi\rangle Lu \dint \m^\sN.$$
We obtain
$$\int \langle \nabla u_n, \nabla Lu_n\rangle \psi \dint \m^\sN\rightarrow \int \langle \nabla u, \nabla Lu\rangle \psi \dint \m^\sN.$$
This yields the desired inequality for $u\in D_{W^{1,2}}(L)$. 
\end{proof}
\begin{theorem} Let $F$ be a metric measure space, and let $B$ and $f: B\rightarrow [0,\infty)$ be as before.   We assume that 
\begin{enumerate}
\item $f''+ Kf\leq 0$,  
\item $F$ satisfies the condition $\RCD(K_F(N-1), N)$ where 
$$K_F>  (f')^2+ Kf^2.$$
\end{enumerate}
For $u\in D_{W^{1,2}}(L)$ and $\phi \in D(L)$ with $\phi\geq 0$ and $\phi, L\phi\in L^\infty(\m^\sN)$, we have $$\Gamma_2(u;\phi)\geq {KN} \int |\nabla u|^2 \phi d\m^\sN + \frac{1}{N+1} \int \left( L u\right)^2 \phi d\m^\sN.$$
i.e. $\nwp$ satisfies the condition $\BE(K,N)$. 
\end{theorem}
\begin{proof}
{The  Cheeger energy on $B\times_f^\sN F$ is a strongly local, strongly regular Dirichlet form. Moreover
we  know that $B\times_f^\sN F$ satisfies 
\begin{itemize}
\item a local $(2,2)$-Poicar\'e inequality and 
\item a local volume doubling property
\end{itemize}
(Remark \ref{rem:feller}).}

In this situation, we have a Gaussian upper bound for the heat kernel and consequently the heat semi-group $P_t^\sC=P_t$ on $B\times_f^N F=C$ is $L^2-L^\infty$ ultra-contractive, i.e. $P_t^\sC: L^2(\m^\sN) \rightarrow L^\infty(\m^\sN)$ is a bounded operator. 

Let $\phi \in D(L)$ with $\phi\geq 0$ and $\phi, L\phi\in L^\infty(\m^\sN)$, and let $\phi_n\in \Xi'$ be a sequence that converges to $\phi$ in $D(L)$. By $L^2-L^\infty$ ultracontractivity we get that $P_t\phi_n$ as well as $LP_t\phi_n$ converge in $L^\infty(\m^\sN)$ to $P_t\phi$ and $LP_t\phi$ respectively. 

In particular, we have for $n\in \N$ sufficiently large, let's say $n\geq n_0$, that $P_t\phi_n, LP_t\phi_n\geq -\lambda$. Hence, $\psi_n=P_t \phi_n+\lambda\geq 0$ for $n\geq n_0$  and the formula \eqref{ineq:bochner2} from the previous corollary holds for $u\in D_{W^{1,2}}(L)$ and for $\psi$. 

The uniform convergence implies that the formula \eqref{ineq:bochner2} still holds with $P_t\phi+ \lambda$ in place of $P_t\phi_n+\lambda$.  We can send  first $\lambda$ to $0$ and the formula holds for $P_t\phi$. Then we can send $t$ to $0$ and in combination with the dominated convergence theorem we have that $P_t\phi$ and $LP_t\phi$ converge in weak-* sense to $\phi$ and $L\phi$. From this we obtain the desired estimate. 
\end{proof}

\begin{corollary} \label{main:smooth} Let $K\in \R$ and $N\in (1, \infty)$. Let $F$ be a mm space,  let $B$ be  a $1$-dimensional Riemannian manifold. Let $f: B\rightarrow [0,\infty)$  be smooth  such that $\partial B\subset f^{-1}(\{0\})$ and  $(\%)$ holds.   Assuming that 
\begin{enumerate}
\item $f''+ Kf\leq 0$,  
\item $F$ satisfies the condition $\RCD(K_F(N-1), N)$ where 
$$K_F\geq  (f')^2+ Kf^2,$$
\end{enumerate}
then $B\times_f^\sN F$ satisfies the condition $\RCD(KN, N+1)$. 
\end{corollary}

\begin{proof}{We first notice that, if $\partial B\neq 0$, then $K_F>0$. Then we first assume that $K_F>\sup_B\{(f')^2 + Kf^2\}$. 
\begin{enumerate}
\item Since $\nwp$ satisfies $\MCP$, an exponential growth condition holds by the Bishop-Gromov volume comparison theorem.
\item Since $\Ch^\snwp=\mathcal E^*$, $\nwp$ is infinitesimally Hilbertian. 
\item In step {\bf (1)} of  the proof for Proposition \ref{prop:distancewp} we showed that $\nwp$ satisfies the Sobolev-to-Lipschitz property. 
\item The previous theorem shows that the $\BE(KN,N+1)$ holds. 
\end{enumerate}
Thus $\nwp$ satisfies the conditin $\RCD(KN,N+1)$.}

Finally, if $K_F\geq \sup_B\{(f')^2 + K f^2\}$, we can rescale $F$  into $F'$ such that $F'$ satisfies $\RCD(K'_F(N-1), N)$ with $K_F'>\sup_B\{(f')^2 + Kf^2\}$.  The warped product $B\times^\sN_f F'$ satisfies $\RCD(KN, N+1)$ and converges in measured Gromov-Hausdorff sense to $B\times_f^\sN F$. Hence also the limit satisfies $\RCD(KN,N+1)$. 
\end{proof}
\subsection{Removing smoothness of $f$}
\begin{theorem} Let $K\in\R$ and $N\in (1, \infty)$.  Let $F$ be a mm space,  let $B$ be  a 1D Riemannian manifold, and  let $f: B\rightarrow [0,\infty)$  be Lipschitz continuous such that $\partial B\subset f^{-1}(\{0\})$ and one of the points in Assumption \ref{assumption:doubling2} holds.   We assume that 
\begin{enumerate}
\item $f''+ Kf\leq 0$,  
\item $F$ satisfies the condition $\RCD(K_F(N-1), N)$ where 
$$K_F\geq  (f')^2+ Kf^2.$$
\end{enumerate}
Then $B\times_f^\sN F$ satisfies the condition $\RCD(KN, N+1)$. 
\end{theorem}
\begin{proof} We will construct a sequence of intervals $B_i$ and smooth functions $f_i: B_i\rightarrow [0, \infty)$ respectively such that $\partial B_i\subset f^{-1}(\{0\})$, $(\%)$ holds and
\begin{enumerate}
\item $f_i''+ Kf_i\leq 0$,  
\item $F$ satisfies the condition $\RCD((K_F-\epsilon_i)(N-1), N)$ where 
$$K_F-\epsilon_i\geq  (f_i')^2+ Kf_i^2.$$
\end{enumerate}
Moreover, $B_i$ converges in the pointed Gromov-Hausdorff sense to $B$, and $f_i$ converges uniformly to $f$ on any compact subset of $B_i$. This will be done as follows.
\medskip

We will consider the following  cases separately. 
\begin{enumerate}
\item[\bf i.] $\partial B=\emptyset$: $B\simeq \R$,  or $B\simeq \mathbb S^1$ where $\mathbb S^1\simeq \R/{\textstyle (2\pi \Z)}$;
\item[\bf ii.] $\partial B\neq \emptyset$: $B\simeq [0, 2\pi]$, or $B\simeq [0, \infty)$. 
\end{enumerate}
We recall that   $\partial B\subset f^{-1}(\{0\})$ by assumption. 
\medskip\\
{\bf i.} By Corollary \ref{cor:special} we have that $K\leq  0$ and  $$K_F\geq   K \inf_B f^2 \mbox{ if and only if  }K_F\geq \sup_B\{ |f'|^2_\sB+ K f^2\}.$$
We choose  $\phi\in C_c^2((-1,1), [0, \infty))$ with $\int \phi(\tau) \de \tau=1$ and set $\phi_\epsilon(\tau)= \frac{1}{\epsilon} \phi( \frac{1}{\epsilon} \tau)$. We define
$$ 
s\in B\mapsto f_\epsilon(s) = \int_{-\epsilon}^\epsilon \phi_\epsilon(\tau) f(s+\tau) \dint \tau = \int_{-\epsilon}^\epsilon \phi_\epsilon(r-s) f(r) \dint r.$$ 
Then $f_\epsilon$ is $C^2$ and satisfies 
\begin{center}
$f_\epsilon'' + K f_\epsilon \leq 0$ on $B$. 
\end{center}
Therefore $f_\epsilon$ is $K f_\epsilon$-concave. Moreover $f_{\epsilon}$ converges uniformly to $f$ on any compact subset of $B$.

In the following we pick a sequence $\epsilon_n\downarrow 0$ and write $f_{\epsilon_n}= f_n$.
\smallskip\\
{\it Claim.} If $K_F\geq K \inf_B f^2$, then \begin{center}$\forall \epsilon>0$ $\exists n_\epsilon$:  $ (1+\epsilon) K\geq K\inf_B f^2_n$\  \ $\forall n\geq n_\epsilon$.\end{center}
\begin{proof}[Proof of the claim.]
We pick $s_0 \in B$ such that $I:=\inf_B f^2 \geq  f^2(s_0)- \frac{\epsilon}{2}I$. Then we choose $n_{0} \in \N$ such that $\forall n\geq n_0$ we have that \begin{center}$f^2(s_0)\geq f_n^2(s_0)-\frac{\epsilon}{2}I\geq \inf_B f_n^2- \frac{\epsilon}{2}I$.\end{center}
Hence $(1+\epsilon)\inf_B f^2 \geq \inf_B f^2_n$ for all $n\geq n_0\in \N$. \end{proof}

We rescale $\de_\sF$ with $\sqrt{(1+\epsilon)}\de_\sF= \de^\epsilon_F$. Then the  space $F^\epsilon=(F, \de^\epsilon_F, \m_\sF)$ satisfies $\RCD((1+\epsilon)K_\sF (N-1), N)$. 

Moreover, the points in Assumption \ref{assumption:doubling2} are preserved and therefore we still have $(\%)$.

Hence, Corollary \ref{main:smooth} applies with $F^\epsilon$ and $f_{n}$ for $n\geq n_\epsilon$. It follows that $B\times_{f_n}^N F^\epsilon$ satisfies the condition $\RCD(KN, N+1)$. 
\smallskip\\
{\it Claim.} The $N$-warped product $B\times_{f_n}^N F^\epsilon$ converges in pointed measured GH sense to $B\times_f^N F^\epsilon$ as $n\rightarrow \infty$. 

\begin{proof}[Proof of the claim]
{\color{black} We write $F=F^\epsilon$.
Let $p_0=(r_0, x_0)$ and $p_1=(r_1, x_1)$   be two points in $B\times_f F$ such that $r_1, r_0\leq R$ and $\de_F(x_0,x_1)=l\leq L$. It follows by Theorem \ref{th:albi0} that
$$\de_{\smwp}(p_0,p_1)= \de_{\sB \times_f [0,L] }((r_0,0), (r_1,l)).$$
Moreover 
$\de_{_{\sB \times_{f_n} \sF}}(p_0,p_1)= \de_{\sB \times_{f_n} [0,L] }((r_0,0), (r_1,l)).$

On the other hand, 
since $f_n\rightarrow f$ locally uniformly, it follows that 
$$ g_\sB + f_n^2 (dr)^2 \ \rightarrow g_\sB+ f^2 (d r)^2 \mbox{ locally uniformly on } B\times [0, L].$$
Hence, for $R$, $L$ and $\epsilon>0$, there exists $n_0\in \N$  that only depends on $R, L$ and $\epsilon$ s. t. for all $n\geq n_0$ we have 
$$\left| \de_{\sB \times_f [0,L] }((r_0,0), (r_1,l))- \de_{\sB \times_{f_n} [0,L] }((r_0,0), (r_1,l))\right|\leq \epsilon.$$
Therefore it also follows that $\de_{\sB\times_{f_n}F}$ converges locally uniformly to $\de_{\smwp}$ on $B\times F$, and in particular $B\times_{f_n} F \rightarrow \mwp $ in pointed GH sense.
}

Finally, since $f_n$ converges locally uniformly to $f$, clearly $ f_n^N(r) \de r\otimes \de \m_\sF$ converges weakly to $f^N(r) \de r \otimes \de \m_\sF$. 
\end{proof}
Since $B\times_f F^\epsilon$ is the pointed measured GH limit of $\RCD(KN, N+1)$ spaces it satisfies the same condition itself.  Finally, if $\epsilon \downarrow 0$, it follows easyly that $B\times_f F^\epsilon$ converges in measured GH sense to $B\times_f F$ that therefore also also satisfies the condition $\RCD(KN, N+1)$. 
\\
\\
{\bf ii.}  Since $\partial B\neq \emptyset $, by Corollary \ref{cor:special} we have
$$\sup_{\partial B}|Df|_\sB^2= \sup_B\{ |Df|^2_\sB+ K f^2\}>0.$$
Here $|Df|=\max\left\{ f^+, -f^-, 0\right\}$ is  the {\it Alexandrov derivative}  where $\frac{d^+ f}{ds}=f^+$ and $\frac{d^- f}{ds}=f^-$ are the right and the left derivatives of $f$. $f^+$ and $f^-$ exist everywhere because $f$ is semi-concave. $Df$ coincides a.e. with the absolute value of the usual derivative $f'$ that is defined a.e.

There are at most 2 boundary components of $B$, $\alpha$ and $\omega$. $\alpha$ denotes the boundary  on the left end of the interval $B$, and $\omega$ the boundary  on the right end. We consider $B$ equipped with the standard orientation. 

W.l.o.g. we will assume that $B$ has exactly one boundary component $\alpha$. The other case works similarly.  W.l.o.g. we also assume that $\alpha=0$. Hence $B\simeq [0, \infty)$ and $f'= f^+$ in $0.$

We define $[ \epsilon, \infty)=: B^{\epsilon}$ and $f_{\epsilon}$  as above. $f_\epsilon$ is  clearly well-defined for $s\in B^\epsilon$. We also set $f_n= f_{\epsilon_n}$ for $\epsilon_n\downarrow 0$ as $n\in \N\rightarrow \infty$. 

Semi-concavity of $f$ implies the following. The left and the right derivative, $f^+$ and $f^-$, are continuous from the left and from the right, respectively.  We also  recall that $f^-\geq f^+$ and $f^+=f^-$ a.e. 

We note that $f^+(\alpha)>0$ since $f$ is semi-concave and positive away from $\alpha$.

Let $\eta\in (0,\frac{1}{4} f^+(0))$. Then there exists  $\epsilon_\eta>0$ such that  $f(s)\leq\frac{1}{2} \eta$ and $0<f^+(0)-\eta\leq f^+(s)\leq  K_F(1+ \eta)$  for  $s\in (0, 2\epsilon_\eta)$. 
\smallskip\\
{\it Claim.}
It holds 
$$|f_n'|(s)\rightarrow |f'|(s) $$ 
for every $s\in B\backslash \partial B=: \mathring B$ such that $f'(s)$ exists. 
\begin{proof} 
From the uniform convergence of $f_n$ to $f$ and since both $f$ and $f_n$ are  semi-concave, one has
$$\liminf_{n \rightarrow \infty} |f'_n|(s)=\liminf_{n \rightarrow \infty} |Df_n|(s)\geq |Df|(s) \ \ \forall s \in \mathring B.$$
Moreover, it holds
$$f'_\epsilon(s)= \int_{-\epsilon}^\epsilon \phi_\epsilon(\tau) f'(s+ \tau) \de \tau.$$
Hence 
$$|f'_\epsilon(s)|\leq \int_{-\epsilon}^\epsilon \phi_\epsilon(\tau) |f'(\tau + s)| \de \tau=: (|f'|)_\epsilon(s).$$
The left hand side $(|f'|)_\epsilon(s)$  converges pointwise to $|f'|(s)$  as $\epsilon \rightarrow 0$ for $s\in \mathring B$ whenever $f'(s)$ exists. 
Hence 
$$\limsup |f'_n|(s)\leq \lim_{n\rightarrow \infty} (|f'|)_n(s) = |f'|(s) \mbox{ for a.e. $s\in B$.}$$
This proves the claim. 
\end{proof}

We choose $\epsilon\in (0, \epsilon_\eta)$ such that $f'(\epsilon)$ exists,  and let $n_0\in \N$ s.t. $\forall n\geq n_0$ we have  $f_n(\epsilon)\leq \eta$ and $f'(\epsilon)-\eta\leq f'_n(\epsilon)\leq  f'(\epsilon)(1+\eta)$. 

Hence
\begin{center} $\frac{1}{2} f^+(0)<f^+(0)-2\eta\leq f'(\epsilon)-\eta\leq f'_n(\epsilon)\leq  f'(\epsilon)(1+\eta)\leq K_F(1+\eta)^2.
$
\end{center}

\smallskip
\smallskip

We choose $\bar g: [0, \infty)\rightarrow [0, \infty)$ such that $\bar g''- \frac{f''_n(\epsilon)}{f_n{\epsilon}} \bar g=0$ and $\bar g(0)= f_n(\epsilon)$, $\bar g'(0)= -f_n'(\epsilon)\leq- \frac 1 2 f^+(0)=:-\xi$.  We set $-\frac{f''_n(\epsilon)}{f_n(\epsilon)}=: K(\epsilon)\geq K$. Thus $\bar g''+K\bar g\leq 0$. Then, more precisely, we have
$$\bar g(s)= f_n(\epsilon)\cos_{K(\epsilon)}- f'_n(\epsilon) \sin_{K(\epsilon)}(s)$$
where $\cos_{K(\epsilon)}$ and $\sin_{K(\epsilon)}$ are solutions of $u'' + K(\epsilon) u=0$ with initial conditions $u(0)=1, u'(0)=0$ and $u(0)=0, u'(0)=1$, respectively.

By elementary comparison results there exists a constant $C_{K, \xi}(\eta)\in (0, \infty)$ such that $t_0=\inf \{t>0: \bar g(t)=0\}\leq C_{K,\xi}(\eta)$ and $C_{K, \xi}(\eta)\rightarrow 0$ if $\eta\rightarrow 0$. 

Moreover $$\bar g'(t_0)\geq - f_n'(\epsilon)(1+\delta(\eta))\geq - K_F(1+\eta)^2(1+\delta(\eta))=:-K_F^\eta$$
for some $\delta(\eta)\rightarrow 0$ if $\eta \downarrow 0$.

We set $\bar g(-t + \epsilon)= g(t)$. Then $g$ satisfies $g(\epsilon)=f_n(\epsilon)$, $g'(\epsilon)=f_n'(\epsilon)$,  $g''(\epsilon)=-K(\epsilon)g(\epsilon)=  f_n''(\epsilon)$ and $g'(\epsilon -t_0)\leq K_\sF^\eta$.
We set 
$$h_\epsilon(s)= \begin{cases} f_\epsilon(s)& s\in ( \epsilon, \infty) \\
g(s)& s\in  [\epsilon -t_0,  \epsilon].
\end{cases}
$$
Therefore $h_\epsilon$ is $C^2$ by construction and satisfies 
\begin{enumerate}
\item $h_n'' + K h_n\leq 0, $
\item $h'_n(\alpha+\epsilon -t_0) \leq K_F^\eta.$
\item $h_n: [\alpha+\epsilon - t_0, \infty)\rightarrow [0, \infty)$ converges  locally uniformly to $f:[\alpha, \infty)\rightarrow [0, \infty)$. 
\end{enumerate}
{\it Claim.} The $N$-warped product $B\times_{f_n}^N F^\epsilon$ converges in pointed measured GH sense to $B\times_f^N F^\epsilon$ as $n\rightarrow \infty$. 

We can prove this claim  similarly as in {\bf i.}  We omit details but recall the following fact for a geodesic $\gamma=(\alpha, \beta)$ in $B\times_f F$. If $\alpha$ does not intersect $\partial B$ we can proceed as before. If $\gamma$ does intersect $\partial B$, then $\gamma$ is a cancatenation of segments in $B$. This type of geodesic is clearly the limit of geodesics in $B\times_{f_n} F$. 
\end{proof}
\begin{theorem} Let $K\in \R$ and $N\in (1, \infty)$.  Let $F$ be a mm space,  let $B$ be  a $1$-dimensional Riemannian manifold, and  let $f: B\rightarrow [0,\infty)$  be Lipschitz continuous such that $(\dagger)$ holds.   It holds $(\%)$. We assume that 
\begin{enumerate}
\item $f''+ Kf\leq 0$,  
\item $F$ satisfies the condition $\RCD(K_F(N-1), N)$ where 
$$K_F\geq  (f')^2+ Kf^2.$$
\end{enumerate}
Then $B\times_f^\sN F$ satisfies the condition $\RCD(KN, N+1)$. 
\smallskip\\
$(\dagger)$ If $B^\dagger$ is the result of gluing two copies of $B$ together along the boundary component $\partial B\backslash f^{-1}(\{0\})$, and $f^\dagger: B^\dagger \rightarrow [0,\infty)$ is the tautological extension of $f$ to $B^\dagger$, then $(f^\dagger)''+ K f^\dagger\leq 0$ is satisfied on $B^\dagger$. 
\end{theorem}
\begin{proof} We observe that $\partial B^\dagger\subset (f^\dagger)^{-1}(\{0\})$.
Hence, we can apply the previous theorem with $B^\dagger$ and $f^\dagger$ in place of $B$ and $f$ respectively.  We obtain that $B^\dagger\times_{f^\dagger}^\sN F$ satisfies the condition $\RCD(KN,N+1)$. 
\smallskip\\
{\it Claim.} $B\times_f F$ is a geodesically convex subset of $B^\dagger\times_{f^\dagger} F=C^\dagger$. 

Let $\gamma=(\alpha, \beta): [0,1] \rightarrow C^\dagger$ be geodesic such that $\gamma(0), \gamma(L)\in C$. Let $\phi: [0,L]\rightarrow [0,1]$ be a $1$-speed reparametrization of $\beta$. We set $\psi= \phi^{-1}$. 
The warped products $B\times_f [0,L]$ is a geodesically convex subset of   $B^\dagger\times_{f^\dagger}[0,L]$. 
By fiber independence the curve $(\alpha, \psi)$ is a minimal geodesic in $B^\dagger\times_{f^\dagger}[0,L]$ with endpoints in $B\times_f[0,L]$. Since $B\times_f [0,L]$ is geodesically convex, we have $\alpha:[0,1]\rightarrow B$. It follows that $\Im\gamma\subset C$. 
\smallskip

Since $C$ is a geodesically convex subset of $C^\dagger$, we have that the condition $\RCD(KN, N+1)$ for $B\times_f^\sN F$ follows from the corresponding condition for $B^\dagger\times_{f^\dagger}^\sN F$. 
\end{proof}
\begin{proof}[End of the proof of Theorem \ref{maintheorem1}] {
Let us first assume $N=1$. The $\RCD(0,1)$ condition for $F$ yields that $F$ is isometric to  $[0, L]$ or to $\alpha\mathbb S^1$. Then  result follows from Theorem \ref{th:albi1} in combination with \cite{lytchakstadler}.}

  Hence we can assume $N>1$. We  have already finished the proof under the assumption $(\%)$.

Therefore we   have to remove the assumption $(\%)$. The  only case that we have to consider is when $f$ is not bounded. 

In this case, we can find sequences $r^{\pm}_n=r^\pm\rightarrow \pm \infty$ such that $f^-(r^+)\leq  0$ and $f^+(r^-)\leq 0$. Hence $B^r=[r^-, r^+]$ and $f|_{[r^-, r^+]}$ satisfy $(\dagger)$.  Thus $B^r\times _f F$ satisfies the condition $\RCD(KN, N+1)$. 

If we choose a point $p=(r,x)\in \nwp$ and a bounded neighborhood $U$ of $p$ in $\nwp$, then there exsits $n\in \N$ large enough such that $U$ isometrically embeds into $B^r\times_f^N F$. 

Now, since $B^r\times^N_f F$ satisfies $\RCD(KN, N+1)$ and since $p$ and $U$ in $\nwp$ are arbitrary, $\nwp$ satisfies the condition $\CD(KN, N+1)$ locally  in the sense of \cite{stugeo2}. Since $\nwp$ is nonbranching, it therefore satisfies the condition $\CD(KN,N+1)$ globally by \cite{cavmil}. Moreover, by construction and since $F$ is $\RCD$, $\m^\sN$-a.e. point in $\nwp$ admits a Euclidean tangent cone. Hence, it follows from \cite{Kap-Ket-18} that $\nwp$ is $\RCD(KN, N+1)$.

This finishes the proof of Theorem \ref{maintheorem1}. \end{proof}

\section{$N$-warped products satisfying a $\RCD$ condition}
\begin{proof}[Proof of Theorem \ref{maintheorem2}]
{
{\bf (1)} {\it Claim:}  $f''+ Kf \leq 0$.}

We can argue as follows. If we pick a minimal geodesic $\alpha:[a,b] \rightarrow B$, we know that for each $x\in F$ the set $\Im \alpha\times \{x\}$ is the image of the minimal geodesic $\gamma(t)=(\alpha(t), x)$ in $B\times_f F$, and $\Im \alpha \times \{x\}, \ {x\in F}$, is a decomposition of $\Im \alpha \times F$ into geodesic segments. Hence, this yields a disintegration of $\m^\sN|_{\Im\alpha \times F}$, that is 
given through 
$$\m^\sN|_{\Im\alpha \times F}= \int_F \gamma_\sharp (f^N\circ \alpha \de t) \de \m_F.$$
Since $B\times^N_f F$ satisfies the $\RCD(KN,N+1)$ condition, it satisfies the $\CD(KN,N+1)$ condition. 

Hence $f\circ \alpha$ is $\frac{KN}{N} f$-concave, therefore also $f$. 
\begin{jjj}
In particular $f>0$ in $B\backslash \partial B$. 
\end{jjj}
 \medskip
\noindent
{\bf (2)} {\it Claim:}  $\langle f', n\rangle_{\sB} \geq 0$ on $\partial B\backslash f^{-1}(\{0\})$ for the outer normal vector $n$. 

Let $\beta: [0, L] \rightarrow F$ be a geodesic in $F$. We know that $B\times_f \Im \beta$ embeds isometrically into $B\times_f F$ by Theorem \ref{th:albi1}.  Here $B\times_f \Im \beta$ is the product space $B\times \Im \beta$ equipped with the continuous metric $(\de t)^2 + f^2(t) (\de r)^2$. 

We assume the claim is not true, i.e. there exists $r_0 \in \partial B$ such that $f(r_0)>0$ and $\langle f', n\rangle< 0$.  Let us assume that $r_0$ is a boundary point on the left. Then $\langle f', n\rangle \leq 0$ means that $\frac{d^+}{dt} f\big|_{r_0}< 0$.  Then $B\times_f [0,L]$ is not an Alexandrov space. 

On the other hand $B\times_f (0, L)$ is locally an Alexandrov space.  Since $B\times_f [0, L]$ is the closure of $B\times_f (0,L)$ this can only happen if $B\times_f (0,L)$ is not geodesically convex in $B\times_f[0,L]$. Then if follows that there exist a geodesic in $B\times_f [0,L]$ that branches at some intermediat point. But since $B\times_f [0, L]$ embeds isometrically into $B\times_f F$, that is an $\RCD$ space, this is contradiction with the fact that geodesic in $\RCD$ spaces are nonbranching \cite{qindeng}.


\smallskip

\noindent
{\bf (3)}
We consider again two cases. 
\smallskip\\
{\bf i.} $f^{-1}(\{0\})\neq \emptyset$\\
Let $(r, x)=p\in \partial B$. We set $|f'|(r)= \alpha$. The tangent cone at $p$ is unique and given by the warped product $[0, \infty) \times_{\alpha r}^N F= [0, \infty) \times^\sN_r \alpha^{-1} F$. Then, the tangent cone is also an $\RCD(0,N+1)$ space. Hence, by \cite{ketterer2} $\alpha^{-1} F$ is an $\RCD(N-1, N)$ space and $F$ is an $\RCD(\alpha(N-1), N)$ space. Since $p\in \partial B$ was arbitrary, it follows that $F$ satisfies $\RCD(K_F(N-1), N)$ where $K_F=\sup_{\partial B} |f'|$. Hence, we obtain the conclusion with Proposition \ref{prop:albi2}. \\
\\
{\bf ii.} $f^{-1}(\{0\})=\emptyset$\\
By  Corollary \ref{cor:special} we know that $K\leq 0$. If $K= 0$, then $f$ is concave.  {\color{black} It follows that $f$ is constant. Indeed, since $f^{-1}(\{0\})=\partial B$ is empty, we have $B\simeq \R$ or $B\simeq \mathbb S^1$. Since $f$ is concave,  
it follows $f$ is constant.  

If $f$ is constant then $B\times_f^\sN F= B\times F$, that is $B\times F$ equipped with $\ell^2$-product metric $\de_{\sB\times \sF} = \sqrt{{\scriptstyle |\cdot - \cdot |^2 + \de_\sF^2}}$ and the measure $\de r\otimes \m_\sF$. By \cite{giglisplitting} one has that $F$ is $\RCD(0,N)$.}
 
Hence we will assume $K<0$ and again by Corollary \ref{cor:special}  we have $\inf_B f^2=0$.  Since $f^{-1}(\{0\})=\emptyset$, it follows that $B$ is noncompact. In partiuclar, there is a sequence $r_i$ diverging to infinity, i.e. $\de_\sB(r_i, r_{i+j})\rightarrow \infty$ if $j\rightarrow \infty$ and for all $i$, such that $f(r_i)\rightarrow 0$.  The goal is to  prove that $F$ satisfies $\RCD(0,N)$.

We adapt an idea from \cite{albi3}. We set  $f(r_i)=:a_i\rightarrow 0$ and 
$$\lambda_i= \frac{1}{a_i}, \ \ f_i = \lambda_i f(\frac{1}{\lambda_i} \cdot) : \R \rightarrow (0, \infty). $$
Then, $f_i$ is $a_i^2 K$-concave. Moreover, $f_i\geq 0$ and $f_i\leq C$ on $(r_i- R\lambda_i, r_i+ R\lambda_i)$. After extracting a subsequence, by the Arzela-Ascoli theorem $f_i$ converges to a limit function $f_\infty$ on $\R$ such that $f''\leq 0$ and $r_i \rightarrow r_\infty\in \R$ such that $1=f_i(r_i)\rightarrow f_\infty(r_\infty)$.
Hence $f_\infty\equiv 1$. 

Moreover, $B_i\times_{f_i}^N F= \lambda_i B\times_f^N F$ satisfies $\RCD(a_i^2 KN, N+1)$ and $B_i\times_{f_i}^N F$ converges in pointed measured GH sense to $\R\times_1 F$. Hence $\R\times F$ satisfies the condition $\RCD(0,N+1)$. We conclude from \cite{giglisplitting, giglisplittingshort} that $F$ satisfies the condition $\RCD(0, N)$. 
\smallskip\\
{\color{black}
{\bf (4)}  (a) We have $\inf_B f^2>0$. Otherwise $K_F\geq 0$. By rescaling $f$ and $B\times_f^\sN F$ we can also assume $\inf_B f^2=1$. 

We pick a sequence $(r_i)_{i\in \N}\subset B$ such that $f(r_i)\rightarrow 1$. Let $\epsilon_i \downarrow 0$. For all $i \in \N$ there exists $\delta_i\in (0, \epsilon_i)$ such that 
$$ f|_{[r_i-\delta_i, r_i+ \delta_i]}\leq 1+ \epsilon_i.$$

If $r_i \rightarrow \infty$ (or $-\infty$),  we  define $f_i: B= \R \rightarrow (0, \infty)$ via $f_i(r)= f(r-r_i)$. After extracting a subsequence $f_i$ will converge locally uniformily to some $\bar f K$-concave function $\bar f: \R\rightarrow [1, \infty)$ such that $\inf_\sB \bar f = \min_\sB \bar f= \bar f(0).$  Moreover $B\times^\sN_{f_i} F$ converges in pointed measured GH sense to $\R\times^\sN_{\bar f} F$, and hence $\R\times^\sN_{\bar f} F$ still satisfies the condition $\RCD(KN, N+1)$.

Therefore, we can  assume w.l.o.g. that $\inf_\sB f=  \min_\sB f= f(0)$ and $r_i = 0$ $\forall i \in \N$. In this case $x\in F\mapsto (0, x)\in B\times_f F$ is a distance preserving embedding.}
\medskip\\
(b) 
We fix $\epsilon_i>0$ and $\delta_i>0$ as before.

Since $f$ has a minimum in $0$ there exist $r^i_0\in (-\delta_i, 0)$ and $r^i_1\in (0, \delta)$ such that $\frac{d^+ f}{dr}|_{r_0^i}=: f^+(r^i_0)\leq 0$ and $\frac{d^- f}{dr}|_{r_1^i}=:f^-(r^i_1)\geq 0$. 

Let $p_0=(s_0, x_0)$ and $p_1=(s_1, x_1)$ be points $(r_0^i, r_1^i)\times F\subset B\times_f F$. We show that for any minimal geodesic $\gamma=(\alpha, \beta):[0,1]\rightarrow B\times_f F$ between $p_0$ and $p_1$, it follows that $\alpha(t)\in (r_0^i, r_1^i)$. 

To see this we argue as follows. 
We consider $f|_{[r_0^i, r_1^i]}= :g$. The function $g$ is $Kg$-concave and satisfies the condition (1)(c) in Theorem \ref{th:albi1} Hence $[r_0^i, r_1^i]\times_g [0, L]=:X$ with $L=\de_F(x_0, x_1)$ is an Alexandrov space. If $\tilde \gamma=(\tilde \alpha, \tilde \beta)$ is geodesic in $X$ between $(r_0^i, 0)$ and $(r_1^i, L)$, then also $(\alpha, \tilde \beta)$ is a geodesic in $X$ with the same endpoints. This follows from Theorem \ref{th:albi0}. Assume there exist $t_0\in [0,1]$ such that $\alpha(t_0)\in \{r_0^i, r_1^i\}$. From the properties of $g$ also $X^\dagger$ is an Alexandrov space, again by Theorem \ref{th:albi1} where $X^\dagger$ is space that we obtain by gluing together two copies of $X$ at $r_0^i$ and $r_1^i$. By considering the two copies of the geodesic $(\alpha, \tilde \beta)$ in $X^\dagger$ we obtain a geodesic in $X^\dagger$ that branches at some intermediate point. This contradicts with $X^\dagger$ being Alexandrov.

We set $Z_i= [r_0^i, r_1^i]\times F$ and see that $Z_i$ is geodesically convex. Hence, equipped with the measure $\m^\sN|_{Z_i}$, the subset $Z_i$ is still an $\RCD(KN, N+1)$ spaces. 
\medskip\\
(c)
For an admissible curve $\gamma=(\alpha, \beta)$ {\it in} $Z_i$ between points $p_0,p_1\in Z_i$  we have that 
$${ \de_{B\times_f F}|_{Z_i\times Z_i}(p_0,p_1)\leq \int \sqrt{ |\alpha'|^2 + f^2\circ \alpha |\beta'|^2}\leq \int \sqrt{ |\alpha'|^2 + (1+\epsilon_i) |\beta'|^2}}.$$
The infimum of the right hand side w.r.t. all such curves $\gamma=(\alpha, \beta)$ {in} $Z_i$ is $\de_{B\times (1+\epsilon_i)F}|_{Z_i\times Z_i}(p_0,p_1)$.  Hence 
\begin{align}\label{787} \de_{B\times_f F}|_{Z_i\times Z_i}\leq \de_{B\times (1+\epsilon_i)F}|_{Z_i\times Z_i}.
\end{align}

On the other hand, for {\it every} admissible curve $\gamma=(\alpha, \beta)$ in $B\times F$ we have 
$$\int \sqrt{ |\alpha'|^2 +  |\beta'|^2} \leq  \int \sqrt{ |\alpha'|^2 + f^2\circ \alpha |\beta'|^2}.$$
It follows that 
$$\de_{B\times F}|_{Z_i\times Z_i}\leq \de_{B\times_f F}|_{Z_i\times Z_i}\leq \de_{B\times (1+\epsilon_i)F}|_{Z_i\times Z_i}.$$
Hence, we obtain that $(Z_i, \de_{B\times_f F}|_{Z_i\times Z_i})$ converges in GH sense to $F$. 
\medskip\\
(d) We renormalize the measure $\m^\sN|_{Z_i}$ as follows:  $\frac{1}{\lambda_i} f^N \mathcal L^1|_{[r_0^i, r_1^i]} \otimes \m_\sF$ where $ \int_{r_0^i}^{r_1^i} f^N(r) \de r=:\lambda_i$. 
It follows that  $\lambda_i^{-1}\m^N|_{Z_i}$ converges weakly to $\m_F$. Therefore, the $\RCD(KN, N+1)$ spaces $(Z_i, \de_{B\times_f F}|_{Z_i\times Z_i}, \lambda_i^{-1} \m^\sN)$ converge to $(F, \de_\sF, \m_\sF)$ that is consequently also an $\RCD(KN, N+1)$ space. Since we rescaled $B\times_f^\sN F$ and $f$ such that $\min_{r\in B} f(r)=1$, the claim in {\it (4)} follows. 
\medskip\\
This finishes the proof of the theorem.
\end{proof}
\subsubsection{Proof of Theorem \ref{lastcorollary}}
We observe that, up to isomorphisms, the assumption on $f$ in Theorem \ref{lastcorollary} leaves us with one of the following 6 cases. 
\begin{enumerate}
\item $K=K_F=1$, then $B=[0,\pi]$ and $f(r)=\sin(r)$ 
\\
(spherical suspension),
\item $K=0$ and $K_F=1$, then $B=[0, \infty)$ and $f(r)= r$ 
\\
(Euclidean cone), 
\item $K=0$ and $K_F=0$, then $B=\R$ and $f(r)=1$ 
\\
(Cartesian product), 
\item $K=-1$ and $K_F=1$, then $B=[0, \infty)$ and $f(r)= \sinh(r)$ 
\\
(elliptic cone), 
\item $K=-1$ and $K_F=0$, then $B=\R$ and $f(r)=\exp(r)$ 
\\
(parabolic cone), 
\item $K=-1$ and $K_F=-1$, $B=\R$ and $f(r)=\cosh(r)$
\\
(hyperbolic cone).
\end{enumerate}
Moreover, the generalized Pythagorean identity holds in each of these cases: 
$$(f')^2 + K f^2 = K_F, \ \ K, K_F\in \{-1, 0, 1\} .$$

The spherical suspension, the Euclidean cone, and the elliptic cone were  treated in \cite{ketterer2}. 

The  Cartesian product was treated in \cite{giglisplitting}.

The case of the parabolic cone is covered by Theorem \ref{maintheorem2}. 
\smallskip

Hence, the only case that is not covered already is the hyperbolic cone.
However, it can be treated exactly like the  cases in \cite{ketterer2}. 

The proof is verbatim the same. So we will not provide details here and refer to \cite{ketterer2}. The main points one has to notice are: 
\begin{itemize}
\item[(i)]  the generalized Pythagorean identity holds, 
\item[(ii)]  Proposition \ref{prop:oneill} holds
\item[(iii)] $F$ is a compact metric measure space that is geodesic with a finite measure such that doubling property holds and it admits a local Poincar\'e inequality. 
\end{itemize}
\appendix
\section{Proof of Proposition \ref{prop:oneill}}\label{appendix}

Let $u=u_1\otimes u_2, v=v_1\otimes v_2\in C_{c}^\infty(\mathring B)\otimes D_{\sW^{1,2}}(L^\sF)\subset D_{W^{1,2}}(L)$, as well as $\phi=\phi_1\otimes \phi_2\in P_t^{\sB, \sN, \lambda}C_c^\infty(\mathring B)\otimes E(\lambda)$.  
\smallskip

We note that $\phi\in P_t^{\sB, \sN, \lambda} C_c^\infty(\mathring B)\otimes E(\lambda)$ satisfies $\phi \in D_{L^\infty}(L^\sC)\cap L^\infty(\m^\sN). $
Then the $\Gamma_2$-operator of $u,v$ and $\phi$ 
\begin{align*}
\Gamma_2(u,v;\phi)=&\underbrace{ \int \frac{1}{2} \langle \nabla u, \nabla v\rangle \eL \phi d\m^\sN }_{=: (I)} - \underbrace{\int \langle \nabla u, \nabla \eL v\rangle \phi d\m^\sN}_{=:(II)}
\end{align*}
is well-defined. 

Two times the first integral on the RHS  is
%
\begin{align*}
2(I)&=\int\left[ \langle u_1' , v_1'\rangle_\sB u_2 v_2+ \frac{u_1 v_1}{f^2} \langle \nabla u_2, \nabla v_2\rangle_\sF \right] \eL \phi d\m^\sN \\
&=\underbrace{\int\langle u_1' , v_1'\rangle_\sB u_2 v_2 \eL \phi d\m^\sN}_{=:(I)_1}+  \underbrace{\int \frac{u_1 v_1}{f^2} \langle \nabla u_2, \nabla v_2\rangle_\sF \eL \phi d\m^\sN}_{=:(I)_2} 
\end{align*}
We have that $L\phi= L^{\sB, \sN, \lambda} \phi_1 \otimes \phi_2$. 

Since $\phi\in L^\infty(\m^\sN)$, it follows that $\phi_2, L^\sF\phi_2\in L^\infty(\m_\sF^\sN)$.

Therefore we can compute 
\begin{align*} 
(I)_1=& \int\langle u_1' , v_1'\rangle_\sB u_2 v_2 \left[L^{\sB, \sN, \lambda}\phi_1 \otimes \phi_2\right] \de\m^\sN\\
=& \int \left[\int \phi_1 L^{B, f^\sN, \lambda}  \langle u_1', v_1'\rangle_\sB d\m^\sN_B\right] u_2 v_2 \phi_2 \de\m_F\\
=& \int \left[\int \phi_1 L^{B, f^\sN} \langle u_1', v_1'\rangle_\sB d\m^\sN_B\right] u_2 v_2 \phi_2 \de\m_F\\
& \ \ \ \ \  \ \ \ \ \ \ \ \ + \int \left[\int u_2 v_2 (-\lambda) \phi_2 d\m_\sF \right]\langle u_1', v_1'\rangle_\sB \frac{\phi_1}{f^2} \de\m^\sN_\sB\\
=& \int \left[\int \phi_1 L^{B, f^\sN} \langle u_1', v_1'\rangle_\sB d\m^\sN_B\right] u_2 v_2 \phi_2 \de\m_F\\
& \ \ \ \ \  \ \ \ \ \ \ \ \ + \int \left[\int L^{\sF}_1(u_2 v_2)  \phi_2 d\m_\sF \right]\langle u_1', v_1'\rangle_\sB \frac{\phi_1}{f^2} \de\m^\sN_\sB
\end{align*}
Here we use 
$$\int v_1 L^{\sB, \sN, \lambda} u_1 \dint \m_\sB^\sN= \int L^{\sB, \sN, \lambda} v_1 u_1 \dint \m_\sB^\sN, \ \ u_1, v_1 \in D(L^{\sB, \sN, \lambda}).$$
Since $\langle u_1', v_1'\rangle_\sB\in C^\infty_c(\mathring B)\subset D(L^{\sB, \sN, \lambda})$, we have
$$L^{\sB, \sN, \lambda} \langle u_1', v_1'\rangle_\sB = L^{\sB, \sN} \langle u_1', v_1'\rangle_\sB - \frac{\lambda}{f^2} \langle u_1', v_1'\rangle_\sB.$$
For the last equality we notice that $u_2 v_2\in D(\eL_1^\sF)$ with $\eL_1^\sF=  (\eL ^\sF v_2 )u_2+ v_2 \eL^\sF u_2 + \langle \nabla u_2, \nabla v_2\rangle_\sF$ and it holds
$$\int u_2 v_2 \eL^\sF \phi_2 d\m_\sF = \int \eL^\sF_1 (u_2 v_2) \phi_2 \dint\m_\sF.$$
Moreover, we notice that $ \phi_1 \cdot L^{\sB, \sN}  \langle u_1', v_1'\rangle_\sB$ and $\langle u_1', v_1'\rangle_\sB \cdot \frac{\phi_1}{f^2}$  are compactly supported in $\mathring B= B\backslash f^{-1}(\{0\})$.  In particular, the behaviour of $\frac{1}{f^2}$ in $f^{-1}(\{0\})$ does not affect the computation. 
\smallskip

We also consider
\begin{align*}
(I)_2=& \int \frac{u_1v_1}{f^2} \langle \nabla u_2, \nabla v_2\rangle_{\sF} \left[ L^{\sB, \sN, \lambda} \phi_1\otimes \phi_2\right] \dint\m^\sN
\end{align*}
and compute 
\begin{align*} 
(I)_2=& \int\left[ \int  L^{\sB, \sN} \left(\frac{u_1v_1}{f^2} \right)\phi_1 \dint\m^\sN_\sB \right]\langle \nabla u_2, \nabla v_2 \rangle_{\sF} \phi_2 \dint\m_\sF \\
&\ \ \ \ \ \ \ \ \ \ \ \ \ +  \int \bigg[\int \langle \nabla u_2, \nabla v_2\rangle_\sF
{L^\sF \phi_2} d\m_{\sF}\bigg] \frac{u_1 v_1}{f^4} \phi_1 \de\m^\sN_\sB.
\end{align*}
Then we consider $(II)=\int\big\langle \nabla u, \nabla \big( (L^{\sB, f^\sN} v_1)v_2+ \frac{v_1}{f^2} \eL^\sF v_2\big)\big\rangle \phi \de\m^\sN$ and compute that
\begin{align*}
(II)
= & \int\langle \nabla u, \nabla (v_2L^{\sB, f^\sN} v_1) \rangle \phi d\m^\sN + \int \langle \nabla u, \nabla \left(\frac{v_1}{f^2} L^F v_2\right)\rangle \phi \de\m^\sN\\
=& \int\left[ \langle u_1',  (L^{\sB, f^\sN}v_1)'\rangle_\sB u_2 v_2 + \frac{u_1 L^{\sB,f^\sN} v_1}{f^2} \langle \nabla u_2, \nabla v_2\rangle_\sF\right]\phi \de\m^\sN \\
&+ \int \left[ \langle u_1', \left( \frac{v_1}{f^2}\right)' \rangle_\sB u_2 \eL^\sF v_2  + \frac{u_1 v_1}{f^4} \langle \nabla u_2, \nabla \eL^\sF v_2\rangle_\sF \right]\phi\de\m^\sN\\
=&  \int \langle u_1',  (L^{\sB, f^\sN}v_1)'\rangle_\sB \phi_1 \de\m_{\sB}^\sN  \int u_2 v_2 \phi_2 d\m_\sF \\
&+ \int   \frac{u_1 L^{\sB,f^\sN} v_1}{f^2} \phi_1 \de\m_\sB^\sN \int  \langle \nabla u_2, \nabla v_2\rangle_\sF\phi_2 \de\m_\sF \\
&+  \int \langle u_1', \left({\textstyle  \frac{v_1}{f^2}}\right)' \rangle_\sB \phi_1 \de \m_\sB^\sN\int  u_2 \eL^\sF v_2 \phi_2 d\m^\sN\\
& + \int  \frac{u_1 v_1}{f^4} \phi_1 \de\m_{\sB}^\sN \int \langle \nabla u_2, \nabla \eL^\sF v_2\rangle_\sF \phi_2\de\m_\sF
\end{align*}
In summary we have 
\begin{align*}
\Gamma_2(u, v; \phi)&=
\frac{1}{2} (I)_1 + \frac{1}{2} (I)_2 - (II)\\
&=\frac{1}{2} \int \left[\int \phi_1 L^{B, f^\sN}  \langle u_1', v_1'\rangle_\sB d\m^\sN_B\right] u_2 v_2 \phi_2 \de\m_F\\
& \ \ \ \ \  \ \ \ \ \ \ \ \ +\frac{1}{2} \int \left[\int \eL^\sF_1 (u_2 v_2) \phi_2 d\m_\sF \right]\langle u_1', v_1'\rangle_\sB \frac{\phi_1}{f^2} \de\m^\sN_\sB\\
&+\frac{1}{2} \int\left[ \int  L^{\sB, f^\sN} \left(\frac{u_1v_1}{f^2} \right)\phi_1 d\m^\sN_\sB \right]\langle \nabla u_2, \nabla v_2 \rangle_{\sF} \phi_2 d\m_\sF \\
&\ \ \ \ \ \ \ \ \ \ \ \ \ +  \frac{1}{2}\int \left[\int \langle \nabla u_2, \nabla v_2\rangle_\sF \eL^\sF \phi_2 d\m_{\sF}\right] \frac{u_1 v_1}{f^4} \phi_1 \de\m^\sN_\sB\\
&- \int \left[\int \langle u_1',  (L^{\sB, f^\sN}v_1)'\rangle_\sB \phi_1 \de\m_{\sB}^\sN \right] u_2 v_2 \phi_2 d\m_\sF \\
&\ \ \ \ \ \ \ \ \ \ \ \ \ - \int   \frac{u_1 L^{\sB,f^\sN} v_1}{f^2} \phi_1 \int  \langle \nabla u_2, \nabla v_2\rangle_\sF\phi_2 \de\m_\sF \de\m_\sB^\sN \\
&- \int \left[\int \langle u_1', \left({\textstyle  \frac{v_1}{f^2}}\right)' \rangle_\sB \phi_1 \de \m_\sB^\sN\right]  u_2 \eL^\sF v_2 \phi_2 d\m_\sF\\
& \ \ \ \ \ \ \ \ \ \ \ \ \ - \int  \frac{u_1 v_1}{f^4} \phi_1 \int \langle \nabla u_2, \nabla \eL^\sF v_2\rangle_\sF \phi_2\de\m^\sN \de\m_{\sB}^\sN
\end{align*}
Hence
\begin{align}\label{id:next}
\Gamma_2(u, v; \phi)
=& \int \Gamma_2^{\sB, f^\sN}(u, v; \phi_1) u_2 v_2 \phi_2 \de\m_\sF\nonumber\\
& + \int \Gamma_2^{\sF}(u_2, v_2;\phi_2) \frac{u_1 v_1}{f^4} \phi_1 d\m_\sB^{\sN}+ \int \J(u,v) \phi \de\m^\sN
\end{align}
where\begin{align*}
\J(u,v)=&\frac{1}{2}  \eL^\sF_1 (u_2 v_2) \langle u_1', v_1'\rangle_\sB \frac{1}{f^2}  -  \langle u_1', \left({\textstyle  \frac{v_1}{f^2}}\right)' \rangle_\sB u_2 \eL^\sF v_2 \\
&+\frac{1}{2}   L^{\sB, f^\sN} \left(\frac{u_1v_1}{f^2} \right)  \langle \nabla u_2, \nabla v_2 \rangle_{\sF}  -    \frac{u_1 L^{\sB,f^\sN} v_1}{f^2}    \langle \nabla u_2, \nabla v_2\rangle_\sF .
\end{align*}
We will compute 
$
\J(u,v) + \J(v,u).
$

Recall that 
$$
\langle u_1' ,\left(\frac{v_1}{f^2}\right)' \rangle_\sB \ = \  \frac{1}{f^2}\langle u_1', v_1'\rangle_\sB - \frac{2v_1}{f^3} \langle f', u_1'\rangle_\sB
$$Since $\eL^{\sB, f^\sN}$ is a diffusion operator, we have
\begin{align*}
L^{\sB, f^\sN}(\frac{u_1 v_1}{f^2})& \ = \  \frac{v_1}{f^2} L^{\sB, f^\sN} u_1 + \frac{u_1}{f^2} L^{\sB, f^\sN} v_1 - \frac{2u_1 v_1}{f^3} L^{\sB, f^\sN} f -\frac{4v_1}{f^3} \langle u_1', f'\rangle_\sB\\
&\ \ \ \ \ \ \  - \frac{4 u_1}{f^3} \langle v_1', f'\rangle_\sB + \frac{6u_1 v_1}{f^4} \langle f', f'\rangle_\sB + \frac{2}{f^2} \langle u_1', v_1'\rangle_\sB
\end{align*}
Moreover
$$L^{\sB, f^\sN} g=  \Delta^\sB g - \langle (\ln f^\sN)', g'\rangle_\sB= \Delta^\sB g- \frac{N}{f} \langle f', g'\rangle_\sB.$$
Hence 
\begin{align*}
\J(u,v)+ \J(v,u) = &\frac{2}{f^2}\langle \nabla u_2, \nabla v_2\rangle_\sF \langle u_1', v_1'\rangle_\sB   + \frac{2v_1}{f^3} \langle f', u_1'\rangle_\sB u_2 \eL^\sF v_2\\
& +  \frac{2u_1}{f^3} \langle f', v_1'\rangle_\sB v_2 \eL^\sF u_2 - \frac{2u_1 v_1}{f^3} L^{\sB, f^\sN} f  \langle \nabla u_2, \nabla v_2 \rangle_{\sF}\\
& -\frac{4v_1}{f^3} \langle u_1', f'\rangle_\sB \langle \nabla u_2, \nabla v_2 \rangle_{\sF}   - \frac{4 u_1}{f^3} \langle v_1', f'\rangle_\sB  \langle \nabla u_2, \nabla v_2 \rangle_{\sF} \\
&+ \frac{6u_1 v_1}{f^4} \langle f', f'\rangle_\sB   \langle \nabla u_2, \nabla v_2 \rangle_{\sF}+ \frac{2}{f^2} \langle u_1', v_1'\rangle_\sB  \langle \nabla u_2, \nabla v_2 \rangle_{\sF}\\
=& \frac{2v_1}{f^3} \langle f', u_1'\rangle_\sB u_2 \eL^\sF v_2+  \frac{2u_1}{f^3} \langle f', v_1'\rangle_\sB v_2 \eL^\sF u_2\\
& - \frac{2u_1 v_1}{f^2}f^\#\langle \nabla u_2, \nabla v_2 \rangle_{\sF}+2\I(u_1,v_1)\langle \nabla u_2, \nabla v_2 \rangle_{\sF}
\end{align*}
where  $f^\#= \frac{ \Delta^\sB f }{f}+ {(N-1)} \frac{\langle f', f'\rangle_\sB}{f^2}$ and
\begin{align*}\I(u_1,v_1)=&\frac{2}{f^4}\Big( {u_1 v_1} \langle f', f'\rangle_\sB   + {f^2} \langle u_1', v_1'\rangle_\sB   -{v_1} f \langle u_1', f'\rangle_\sB  - { u_1}f \langle v_1', f'\rangle_\sB \Big).\end{align*}
Note that 
$$
\I(u_1,v_2)=2\left(\left( \frac{u_1}{f}\right)'\right)\cdot \left(\left( \frac{v_1}{f}\right)'\right).
$$
Since $u$ and $v$ are products of $u_1$ and $u_2$, and $v_1$ and $v_2$ respectively, we can write
$$\I(u_1,v_1)\langle \nabla u_2, \nabla v_2 \rangle_{\sF}= 2 \left\langle  \nabla\left( \frac{u}{f}\right)',\nabla \left( \frac{v}{f}\right)'\right\rangle.$$

So we have now
\begin{align*}
&\Gamma_2(u, v; \phi)+ \Gamma_2(v,u;\phi)\\
 &=\int \Gamma_2^{\sB, f^\sN}(u, v; \phi)  \de\m_\sF + \int \Gamma_2^{\sF}(u, v;\phi) \frac{1}{f^4} d\m_\sB^{\sN}\\
&+\int \bigg[ \frac{2}{f^3} \langle f', u'\rangle_\sB  \eL^\sF v+  \frac{2}{f^3} \langle f', v'\rangle_\sB  \eL^\sF u\\
& - \frac{2}{f^2}f^\#\langle \nabla u, \nabla v \rangle_{\sF}+ 4\left\langle  \nabla\left( \frac{u}{f}\right)',\nabla \left( \frac{v}{f}\right)'\right\rangle\bigg] \phi d\m^\sN.
\end{align*}
By multilinearity in $u, v$ and $\phi$ we get the desired  formula for $u\in C^\infty_c(\mathring B)\otimes D_{\sW^{1,2}}(L^\sF)$ and for $\phi \in C^\infty(\mathring B)\otimes D_{\sW^{1,2}}(L^\sF)$ with  $\phi, L\phi \in L^\infty(\m^\sN)$.
\qed
\bigskip
\paragraph{\bf Data availability} This manuscript has no associated data.
\medskip
\paragraph{\bf Conflict of interest}
The author states that there is no conflict of interest.
\bibliography{new}

\begin{thebibliography}{10}

\bibitem{albi0}
Stephanie~B. Alexander and Richard~L. Bishop.
\newblock Warped products of {H}adamard spaces.
\newblock {\em Manuscripta Math.}, 96(4):487--505, 1998.

\bibitem{albi}
Stephanie~B. Alexander and Richard~L. Bishop.
\newblock Curvature bounds for warped products of metric spaces.
\newblock {\em Geom. Funct. Anal.}, 14(6):1143--1181, 2004.

\bibitem{conesplittingtheorems}
Stephanie~B. Alexander and Richard~L. Bishop.
\newblock A cone splitting theorem for {A}lexandrov spaces.
\newblock {\em Pacific J. Math.}, 218(1):1--15, 2005.

\bibitem{albi3}
Stephanie~B. Alexander and Richard~L. Bishop.
\newblock Warped products admitting a curvature bound.
\newblock {\em Adv. Math.}, 303:88--122, 2016.

\bibitem{agsgradient}
Luigi Ambrosio, Nicola Gigli, and Giuseppe Savar{\'e}.
\newblock {\em Gradient flows in metric spaces and in the space of probability
  measures}.
\newblock Lectures in Mathematics ETH Z\"urich. Birkh\"auser Verlag, Basel,
  second edition, 2008.

\bibitem{agsriemannian}
Luigi Ambrosio, Nicola Gigli, and Giuseppe Savar{\'e}.
\newblock Metric measure spaces with {R}iemannian {R}icci curvature bounded
  from below.
\newblock {\em Duke Math. J.}, 163(7):1405--1490, 2014.

\bibitem{agsbakryemery}
Luigi Ambrosio, Nicola Gigli, and Giuseppe Savar\'e.
\newblock {Bakry-\'Emery} curvature-dimension condition and {R}iemannian
  {R}icci curvature bounds.
\newblock {\em Ann. Probab.}, 43(1):339--404, 2015.

\bibitem{amsnonlinear}
Luigi Ambrosio, Andrea Mondino, and Giuseppe Savar\'{e}.
\newblock Nonlinear {D}iffusion {E}quations and {C}urvature {C}onditions in
  {M}etric {M}easure {S}paces.
\newblock {\em Mem. Amer. Math. Soc.}, 262(1270):0, 2019.

\bibitem{anderson}
Michael~T. Anderson.
\newblock Metrics of positive {R}icci curvature with large diameter.
\newblock {\em Manuscripta Math.}, 68(4):405--415, 1990.

\bibitem{bbi}
Dmitri Burago, Yuri Burago, and Sergei Ivanov.
\newblock {\em A course in metric geometry}, volume~33 of {\em Graduate Studies
  in Mathematics}.
\newblock American Mathematical Society, Providence, RI, 2001.

\bibitem{cks}
Matteo Calisti, Christian Ketterer, and Clemens Sämann.
\newblock Generalized cones admitting a curvature-dimension condition, 2025.

\bibitem{cavmil}
Fabio Cavalletti and Emanuel Milman.
\newblock The globalization theorem for the curvature-dimension condition.
\newblock {\em Invent. Math.}, 226(1):1--137, 2021.

\bibitem{cz}
Simone Cecchini and Rudolf Zeidler.
\newblock Scalar and mean curvature comparison via the {Dirac} operator.
\newblock {\em Geom. Topol.}, 28(3):1167--1212, 2024.

\bibitem{cheegerlipschitz}
Jeff Cheeger.
\newblock Differentiability of {L}ipschitz functions on metric measure spaces.
\newblock {\em Geom. Funct. Anal.}, 9(3):428--517, 1999.

\bibitem{almostrigidity}
Jeff Cheeger and Tobias~H. Colding.
\newblock Lower bounds on {R}icci curvature and the almost rigidity of warped
  products.
\newblock {\em Ann. of Math. (2)}, 144(1):189--237, 1996.

\bibitem{cheegercoldingI}
Jeff Cheeger and Tobias~H. Colding.
\newblock On the structure of spaces with {R}icci curvature bounded below. {I}.
\newblock {\em J. Differential Geom.}, 46(3):406--480, 1997.

\bibitem{chen_lina_almost_volume}
Lina Chen.
\newblock Almost volume cone implies almost metric cone for annuluses centered
  at a compact set in ${RCD}({K}, {N})$-spaces.
\newblock Preprint, {arXiv}:2112.09353 [math.{DG}] (2021), 2021.

\bibitem{coldingnaberII}
Tobias~Holck Colding and Aaron Naber.
\newblock Characterization of tangent cones of noncollapsed limits with lower
  {R}icci bounds and applications.
\newblock {\em Geom. Funct. Anal.}, 23(1):134--148, 2013.

\bibitem{cdnpsw}
Chris Connell, Xianzhe Dai, Jes{\'u}s N{\'u}{\~n}ez-Zimbr{\'o}n, Raquel
  Perales, Pablo Su{\'a}rez-Serrato, and Guofang Wei.
\newblock Maximal volume entropy rigidity for {{\(\mathsf{RCD}^\ast (- (N-1),
  N)\)}} spaces.
\newblock {\em J. Lond. Math. Soc., II. Ser.}, 104(4):1615--1681, 2021.

\bibitem{DGi}
Guido De~Philippis and Nicola Gigli.
\newblock From volume cone to metric cone in the nonsmooth setting.
\newblock {\em Geom. Funct. Anal.}, 26(6):1526--1587, 2016.

\bibitem{qindeng}
Qin Deng.
\newblock H{\"o}lder continuity of tangent cones in {{\(\operatorname{RCD}(K,
  N)\)}} spaces and applications to nonbranching.
\newblock {\em Geom. Topol.}, 29(2):1037--1114, 2025.

\bibitem{erbarkuwadasturm}
Matthias Erbar, Kazumasa Kuwada, and Karl-Theodor Sturm.
\newblock On the equivalence of the entropic curvature-dimension condition and
  {B}ochner's inequality on metric measure spaces.
\newblock {\em Invent. Math.}, 201(3):993--1071, 2015.

\bibitem{giglisplitting}
Nicola Gigli.
\newblock The splitting theorem in non-smooth context.
\newblock https://arxiv.org/abs/1302.5555, 2013.

\bibitem{giglisplittingshort}
Nicola Gigli.
\newblock An overview of the proof of the splitting theorem in spaces with
  non-negative {Ricci} curvature.
\newblock {\em Anal. Geom. Metr. Spaces}, 2:169--213, 2014.

\bibitem{giglistructure}
Nicola Gigli.
\newblock On the differential structure of metric measure spaces and
  applications.
\newblock {\em Mem. Amer. Math. Soc.}, 236(1113):vi+91, 2015.

\bibitem{giglinonsmooth}
Nicola Gigli.
\newblock Nonsmooth differential geometry---an approach tailored for spaces
  with {R}icci curvature bounded from below.
\newblock {\em Mem. Amer. Math. Soc.}, 251(1196):v+161, 2018.

\bibitem{gigli_gromov}
Nicola Gigli.
\newblock De {Giorgi} and {Gromov} working together.
\newblock Preprint, {arXiv}:2306.14604 [math.{MG}] (2023), 2023.

\bibitem{hangigliwarped}
Nicola Gigli and Bang-Xian Han.
\newblock Sobolev spaces on warped products.
\newblock {\em J. Funct. Anal.}, 275(8):2059--2095, 2018.

\bibitem{giglikuwadaohta}
Nicola Gigli, Kazumasa Kuwada, and Shin-Ichi Ohta.
\newblock Heat flow on {A}lexandrov spaces.
\newblock {\em Comm. Pure Appl. Math.}, 66(3):307--331, 2013.

\bibitem{gigli_marconi}
Nicola Gigli and Fabio Marconi.
\newblock A general splitting principle on {RCD} spaces and applications to
  spaces with positive spectrum.
\newblock Preprint, {arXiv}:2312.06252 [math.{MG}] (2023), 2023.

\bibitem{gmsstability}
Nicola Gigli, Andrea Mondino, and Giuseppe Savar{\'e}.
\newblock Convergence of pointed non-compact metric measure spaces and
  stability of {R}icci curvature bounds and heat flows.
\newblock {\em Proc. Lond. Math. Soc. (3)}, 111(5):1071--1129, 2015.

\bibitem{koskela}
Piotr Haj{\l}asz and Pekka Koskela.
\newblock Sobolev met {P}oincar\'e.
\newblock {\em Mem. Amer. Math. Soc.}, 145(688):x+101, 2000.

\bibitem{hnw}
Erik Hupp, Aaron Naber, and Kai-Hsiang Wang.
\newblock Lower {Ricci} curvature and nonexistence of manifold structure.
\newblock {\em Geom. Topol.}, 29(1):443--477, 2025.

\bibitem{Kap-Ket-18}
Vitali Kapovitch and Christian Ketterer.
\newblock C{D} meets {CAT}.
\newblock {\em J. Reine Angew. Math.}, 766:1--44, 2020.

\bibitem{ketterer}
Christian Ketterer.
\newblock Ricci curvature bounds for warped products.
\newblock {\em J. Funct. Anal.}, 265(2):266--299, 2013.

\bibitem{ketterer2}
Christian Ketterer.
\newblock Cones over metric measure spaces and the maximal diameter theorem.
\newblock {\em J. Math. Pures Appl. (9)}, 103(5):1228--1275, 2015.

\bibitem{ketterer3}
Christian Ketterer.
\newblock Obata's rigidity theorem for metric measure spaces.
\newblock {\em Anal. Geom. Metr. Spaces}, 3:278--295, 2015.

\bibitem{kettererwp}
Christian Ketterer.
\newblock Warped products over one-dimensional base spaces and the rcd
  condition, 2025.

\bibitem{koskelazhou}
Pekka Koskela and Yuan Zhou.
\newblock Geometry and analysis of {D}irichlet forms.
\newblock {\em Adv. Math.}, 231(5):2755--2801, 2012.

\bibitem{kush_bg}
Kazuhiro Kuwae and Takashi Shioya.
\newblock Infinitesimal {B}ishop-{G}romov condition for {A}lexandrov spaces.
\newblock In {\em Probabilistic approach to geometry}, volume~57 of {\em Adv.
  Stud. Pure Math.}, pages 293--302. Math. Soc. Japan, Tokyo, 2010.

\bibitem{lobaem}
John Lott.
\newblock Some geometric properties of the {B}akry-\'{E}mery-{R}icci tensor.
\newblock {\em Comment. Math. Helv.}, 78(4):865--883, 2003.

\bibitem{lottvillani}
John Lott and C{\'e}dric Villani.
\newblock Ricci curvature for metric-measure spaces via optimal transport.
\newblock {\em Ann. of Math. (2)}, 169(3):903--991, 2009.

\bibitem{lytchakstadler}
Alexander Lytchak and Stephan Stadler.
\newblock Ricci curvature in dimension 2.
\newblock {\em J. Eur. Math. Soc. (JEMS)}, 25(3):845--867, 2023.

\bibitem{ohtmea}
Shin-ichi Ohta.
\newblock On the measure contraction property of metric measure spaces.
\newblock {\em Comment. Math. Helv.}, 82(4):805--828, 2007.

\bibitem{rajala2}
Tapio Rajala.
\newblock Interpolated measures with bounded density in metric spaces
  satisfying the curvature-dimension conditions of {S}turm.
\newblock {\em J. Funct. Anal.}, 263(4):896--924, 2012.

\bibitem{reedsimon}
Michael Reed and Barry Simon.
\newblock {\em Methods of modern mathematical physics. {II}. {F}ourier
  analysis, self-adjointness}.
\newblock Academic Press [Harcourt Brace Jovanovich Publishers], New York,
  1975.

\bibitem{savareself}
Giuseppe Savar{\'e}.
\newblock Self-improvement of the {B}akry-\'{E}mery condition and {W}asserstein
  contraction of the heat flow in {${\rm RCD}(K,\infty)$} metric measure
  spaces.
\newblock {\em Discrete Contin. Dyn. Syst.}, 34(4):1641--1661, 2014.

\bibitem{sturmdirichlet2}
Karl-Theodor Sturm.
\newblock Analysis on local {D}irichlet spaces. {II}. {U}pper {G}aussian
  estimates for the fundamental solutions of parabolic equations.
\newblock {\em Osaka J. Math.}, 32(2):275--312, 1995.

\bibitem{sturmdirichlet3}
Karl-Theodor Sturm.
\newblock Analysis on local {D}irichlet spaces. {III}. {T}he parabolic
  {H}arnack inequality.
\newblock {\em J. Math. Pures Appl. (9)}, 75(3):273--297, 1996.

\bibitem{stugeo1}
Karl-Theodor Sturm.
\newblock On the geometry of metric measure spaces. {I}.
\newblock {\em Acta Math.}, 196(1):65--131, 2006.

\bibitem{stugeo2}
Karl-Theodor Sturm.
\newblock On the geometry of metric measure spaces. {II}.
\newblock {\em Acta Math.}, 196(1):133--177, 2006.

\bibitem{sturm25}
Karl-Theodor Sturm.
\newblock Bakry-{\'e}mery, {Hardy}, and spectral gap estimates on manifolds
  with conical singularities.
\newblock {\em Calc. Var. Partial Differ. Equ.}, 64(3):31, 2025.
\newblock Id/No 94.

\end{thebibliography}
\bibliographystyle{plain}
\end{document}